\makeatletter

\documentclass[
    11pt,
    a4paper,
    oneside,
    openright,
    center,
    chapterbib,
    crosshair,
    reqno,
    headcount,
    headline,
    indent,
    indentfirst=false,
    portrait,
    phonetic,
    oldernstyle,
    onecolumn,
    sfbold,
    upper,
]{amsart}

\let\th@plain\relax

\PassOptionsToPackage{T2A}{fontenc} 
\PassOptionsToPackage{utf8}{inputenc} 
\PassOptionsToPackage{cyrpart}{cyrillic}
\PassOptionsToPackage{british,ngerman,russian}{babel}
\PassOptionsToPackage{
    english,
    ngerman,
    russian,
    capitalise,
}{cleveref}
\PassOptionsToPackage{
    margin=10pt,
    position=bottom,
    justification=centering,
    font=small,
    labelformat=simple,
    labelfont={sc},
    labelsep={period},
    textfont={it},
}{caption}
\PassOptionsToPackage{
    margin=10pt,
    position=bottom,
    justification=centering,
    font=small,
    labelformat=parens,
    labelfont={bf},
    labelsep={space},
    textfont={it},
}{subcaption}
\PassOptionsToPackage{framemethod=TikZ}{mdframed}
\PassOptionsToPackage{
    bookmarks=true,
    bookmarksopen=false,
    bookmarksopenlevel=0,
    bookmarkstype=toc,
    raiselinks=true,
    colorlinks=true,
    hyperfigures=true,
    anchorcolor=red,
    citecolor=green,
    linkcolor=blue,
    urlcolor=blue,
}{hyperref} 
\PassOptionsToPackage{normalem}{ulem}
\PassOptionsToPackage{
    amsmath,
    thmmarks,
}{ntheorem}
\PassOptionsToPackage{
    table,
}{xcolor}
\PassOptionsToPackage{
    all,
    color,
    curve,
    frame,
    import,
    knot,
    line,
    movie,
    rotate,
    textures,
    tile,
    tips,
    web,
    xdvi,
}{xy}
\PassOptionsToPackage{
    reset,
    left=1in,
    right=1in,
    top=20mm,
    bottom=20mm,
    heightrounded,
}{geometry}
\PassOptionsToPackage{
    symbol*, 
    multiple, 
}{footmisc}
\PassOptionsToPackage{overload}{textcase}
\PassOptionsToPackage{indentafter}{titlesec}
\PassOptionsToPackage{
    shortcuts, 
}{extdash}

\usepackage{amsfonts}
\usepackage{amsmath}
\usepackage{amssymb}
\usepackage{ntheorem} 
\usepackage{array}
\usepackage{babel}
\usepackage{bbding}
\usepackage{bm}
\usepackage{bbm}
\usepackage{bibentry}
\usepackage{booktabs}
\usepackage{bold-extra} 
\usepackage{calc}
\usepackage{cancel}
\usepackage{caption} 
\usepackage{changepage}
\usepackage{cjhebrew}
\usepackage{cmlgc}
\usepackage{colonequals}
\usepackage{color}
\usepackage{comment}
\usepackage{datetime}
\usepackage{dsfont}
\usepackage{etex}
\usepackage{etoolbox}
\usepackage{eurosym}
\usepackage{extdash}
\usepackage{fancybox}
\usepackage{fancyhdr}
\usepackage{float}
\usepackage{fontenc}
\usepackage{footmisc}
\usepackage{fp}
\usepackage{geometry}
\usepackage{graphicx}
\usepackage{ifpdf}
\usepackage{ifthen}
\usepackage{ifoddpage}
\usepackage{ifnextok}  
\usepackage{inputenc}
\usepackage{latexsym}
\usepackage{lineno}
\usepackage{listings}
\usepackage{longtable}
\usepackage{lscape}
\usepackage{mathtools} 
\usepackage{mathrsfs}
\usepackage{multicol}
\usepackage{multirow}
\usepackage{nameref}
\usepackage{nowtoaux}
\usepackage{paralist}
\usepackage{enumerate} 
\usepackage{pgf}
\usepackage{pgfplots}
\usepackage{phonetic}
\usepackage{proof}
\usepackage{qtree}
\usepackage{refcount}
\usepackage{savesym}
\usepackage{stmaryrd}
\usepackage{synttree}
\usepackage{subcaption}
\usepackage{suffix}
\usepackage{yfonts} 
\usepackage{textcase} 
\usepackage{tikz}
\usepackage{xy}
\usepackage{wrapfig}
\usepackage{xcolor}
\usepackage{xspace}
\usepackage{xstring}
\usepackage{hyperref}
\usepackage{arydshln}
\usepackage{cleveref} 
\usepackage{verbatim} 

\pgfplotsset{compat=newest}
\usetikzlibrary{math}

\usetikzlibrary{
    angles,
    arrows,
    automata,
    calc,
    decorations,
    decorations.pathmorphing,
    decorations.pathreplacing,
    positioning,
    patterns,
    quotes,
}

\savesymbol{corresponds}
\savesymbol{Diamond}
\savesymbol{emptyset}
\savesymbol{ggg}
\savesymbol{int}
\savesymbol{lll}
\savesymbol{RectangleBold}
\savesymbol{langle}
\savesymbol{rangle}
\savesymbol{hookrightarrow}
\savesymbol{hookleftarrow}
\savesymbol{Asterisk}
\usepackage{mathabx}
\usepackage{wasysym}

\restoresymbol{x}{corresponds}
\restoresymbol{x}{Diamond}
\restoresymbol{x}{emptyset}
\restoresymbol{x}{ggg}
\restoresymbol{x}{int}
\restoresymbol{x}{lll}
\restoresymbol{x}{RectangleBold}
\restoresymbol{x}{langle}
\restoresymbol{x}{rangle}
\restoresymbol{x}{hookrightarrow}
\restoresymbol{x}{hookleftarrow}
\restoresymbol{x}{Asterisk}

\ifpdf
    \usepackage{pdfcolmk}
\fi

\usepackage{mdframed}

\DeclareFontFamily{U}{MnSymbolA}{}
\DeclareFontShape{U}{MnSymbolA}{m}{n}{
    <-6> MnSymbolA5
    <6-7> MnSymbolA6
    <7-8> MnSymbolA7
    <8-9> MnSymbolA8
    <9-10> MnSymbolA9
    <10-12> MnSymbolA10
    <12-> MnSymbolA12
}{}
\DeclareFontShape{U}{MnSymbolA}{b}{n}{
    <-6> MnSymbolA-Bold5
    <6-7> MnSymbolA-Bold6
    <7-8> MnSymbolA-Bold7
    <8-9> MnSymbolA-Bold8
    <9-10> MnSymbolA-Bold9
    <10-12> MnSymbolA-Bold10
    <12-> MnSymbolA-Bold12
}{}
\DeclareSymbolFont{MnSymA}{U}{MnSymbolA}{m}{n}
\DeclareMathSymbol{\lcirclearrowright}{\mathrel}{MnSymA}{252}
\DeclareMathSymbol{\lcirclearrowdown}{\mathrel}{MnSymA}{255}
\DeclareMathSymbol{\rcirclearrowleft}{\mathrel}{MnSymA}{250}
\DeclareMathSymbol{\rcirclearrowdown}{\mathrel}{MnSymA}{251}

\DeclareFontFamily{U}{MnSymbolC}{}
\DeclareSymbolFont{MnSyC}{U}{MnSymbolC}{m}{n}
\DeclareFontShape{U}{MnSymbolC}{m}{n}{
    <-6>  MnSymbolC5
    <6-7>  MnSymbolC6
    <7-8>  MnSymbolC7
    <8-9>  MnSymbolC8
    <9-10> MnSymbolC9
    <10-12> MnSymbolC10
    <12->   MnSymbolC12%
}{}
\DeclareMathSymbol{\powerset}{\mathord}{MnSyC}{180}
\DeclareMathSymbol{\righthalfcap}{\mathbin}{MnSyC}{186}

\DeclareMathAlphabet{\mathpzc}{OT1}{pzc}{m}{it}

\def\boolwahr{true}
\def\boolfalsch{false}
\def\boolleer{}

\let\boolinappendix\boolfalsch
\let\boolinmdframed\boolfalsch

\newcount\bufferctr
\newcount\bufferreplace

\newlength\rtab
\newlength\gesamtlinkerRand
\newlength\gesamtrechterRand
\newlength\ownspaceabovethm
\newlength\ownspacebelowthm
\newlength\aboveequation
\newlength\belowequation
\def\secnumberingpt{.}
\def\secnumberingseppt{.}
\def\subsecnumberingseppt{}
\def\thmnumberingpt{.}
\def\thmnumberingseppt{}
\def\thmForceSepPt{.}

\definecolor{leer}{gray}{1}
\definecolor{boxgrau}{gray}{0.85}
\definecolor{dunkelgrau}{gray}{0.5}
\definecolor{maroon}{rgb}{0.6901961,0.1882353,0.3764706}
\definecolor{dunkelgruen}{rgb}{0.015625,0.363281,0.109375}
\definecolor{dunkelrot}{rgb}{0.5450980392,0,0}
\definecolor{dunkelblau}{rgb}{0,0,0.5450980392}
\definecolor{blau}{rgb}{0,0,1}
\definecolor{newresult}{rgb}{0.6,0.6,0.6}
\definecolor{improvedresult}{rgb}{0.9,0.9,0.9}
\definecolor{hervorheben}{rgb}{0,0.9,0.7}
\definecolor{starkesblau}{rgb}{0.1019607843,0.3176470588,0.8156862745}
\definecolor{achtung}{rgb}{1,0.5,0.5}
\definecolor{frage}{rgb}{0.5,1,0.5}
\definecolor{schreibweise}{rgb}{0,0.7,0.9}
\definecolor{axiom}{rgb}{0,0.3,0.3}
\definecolor{drawing_light_grey}{gray}{0.85}
\definecolor{background_light_grey}{gray}{0.95}

\def\let@name#1#2{
    \expandafter\let\csname #1\expandafter\endcsname\csname #2\endcsname\relax
}
\DeclareRobustCommand\crfamily{\fontfamily{ccr}\selectfont}
\DeclareTextFontCommand{\textcr}{\crfamily}

\def\ifthenelseleer#1#2#3{\ifthenelse{\equal{#1}{}}{#2}{#1#3}}
\def\bedingtesspaceexpand#1#2#3{\ifthenelseleer{\csname #1\endcsname}{#3}{#2#3}}

\def\nvraum{\@ifnextchar\bgroup{\nvraum@c}{\nvraum@bes}}
    \def\nvraum@c#1{\vspace*{-#1\baselineskip}}
    \def\nvraum@bes{\vspace*{-\baselineskip}}
\def\forceaddspace{\relax\ifmmode\else\@\xspace\fi}
\def\forceremovespace{\relax\ifmmode\else\expandafter\@gobble\fi}

\def\send@toaux#1{\@bsphack\protected@write\@auxout{}{\string#1}\@esphack}

\def\rlabel#1[#2]#3#4#5{#5\rlabel@aux{#1}[#2]{#3}{#4}{#5}}
    \def\rlabel@aux#1[#2]#3#4#5{%
        \send@toaux{\newlabel{#1}{{\@currentlabel}{\thepage}{{\unexpanded{#5}}}{#2.\csname the#2\endcsname}{}}}\relax%
    }

\def\tag@rawscheme#1#2[#3]#4#5{\@ifnextchar[{\tag@rawscheme@{#1}{#2}[#3]{#4}{#5}}{\tag@rawscheme@{#1}{#2}[#3]{#4}{#5}[*]}}
    \def\tag@rawscheme@#1#2[#3]#4#5[#6]{\@ifnextchar\bgroup{\tag@rawscheme@@{#1}{#2}[#3]{#4}{#5}[#6]}{\tag@rawscheme@@{#1}{#2}[#3]{#4}{#5}[#6]{}}}
    \def\tag@rawscheme@@#1#2[#3]#4#5[#6]#7{%
        \ifthenelse{\equal{#6}{*}}{%
            \ifthenelse{\equal{#7}{\boolleer}}{\refstepcounter{#3}#4\csname the#3\endcsname#5}{#4#7#5}%
        }{%
            \refstepcounter{#3}#4%
            \ifthenelse{\equal{#7}{\boolleer}}{\rlabel{#6}[#3]{#1}{#2}{\csname the#3\endcsname}}{\rlabel{#6}[#3]{#1}{#2}{#7}}%
            #5%
        }%
    }
\def\tag@scheme#1#2[#3]{\tag@rawscheme{#1}{#2}[#3]{\upshape(}{\upshape)}}

\def\eqtag@post#1{\makebox[0pt][r]{#1}}
\def\eqtag@pre{\tag@scheme{Eq}{Equation}[equation]}
\def\eqtag{\@ifnextchar[{\eqtag@}{\eqtag@[*]}}
    \def\eqtag@[#1]{\@ifnextchar\bgroup{\eqtag@@[#1]}{\eqtag@@[#1]{}}}
    \def\eqtag@@[#1]#2{\eqtag@post{\eqtag@pre[#1]{#2}}}

\def\eqcref#1{\text{(\ref{#1})}}
\def\punktlabel#1{\label{it:#1:\beweislabel}}
\def\punktcref#1{\eqcref{it:#1:\beweislabel}}

\def\opfromto[#1]_#2^#3{\underset{#2}{\overset{#3}{#1}}}
\def\textoverset#1#2{\overset{\text{#1}}{#2}}

\def\eqcrefoverset#1#2{\textoverset{\eqcref{#1}}{#2}}

\def\mathclap#1{#1}
\def\oberunterset#1{\@ifnextchar^{\oberunterset@oben{#1}}{\oberunterset@unten{#1}}}
    \def\oberunterset@oben#1^#2_#3{\underset{\mathclap{#3}}{\overset{\mathclap{#2}}{#1}}}
    \def\oberunterset@unten#1_#2^#3{\underset{\mathclap{#2}}{\overset{\mathclap{#3}}{#1}}}
    \def\breitunderbrace#1_#2{\underbrace{#1}_{\mathclap{#2}}}
    \def\breitoverbrace#1^#2{\overbrace{#1}^{\mathclap{#2}}}
    \def\breitunderbracket#1_#2{\underbracket{#1}_{\mathclap{#2}}}
    \def\breitoverbracket#1^#2{\overbracket{#1}^{\mathclap{#2}}}

\def\generatenestedsecnumbering#1#2#3{%
    \expandafter\gdef\csname thelong#3\endcsname{%
        \expandafter\csname the#2\endcsname%
        \secnumberingpt%
        \expandafter\csname #1\endcsname{#3}%
    }%
    \expandafter\gdef\csname theshort#3\endcsname{%
        \expandafter\csname #1\endcsname{#3}%
    }%
}
\def\generatenestedthmnumbering#1#2#3{%
    \expandafter\gdef\csname the#3\endcsname{%
        \expandafter\csname the#2\endcsname%
        \thmnumberingpt%
        \expandafter\csname #1\endcsname{#3}%
    }%
    \expandafter\gdef\csname theshort#3\endcsname{%
        \expandafter\csname #1\endcsname{#3}%
    }%
}

\providecommand{\setcounternach}{}
\renewcommand{\setcounternach}[2]{\setcounter{#1}{#2}\addtocounter{#1}{-1}}

\def\forcepunkt#1{#1\IfEndWith{#1}{.}{}{.}}

\def\matrix#1{\left(\begin{array}{#1}}
    \def\endmatrix{\end{array}\right)}
\def\smatrix{\left(\begin{smallmatrix}}
    \def\endsmatrix{\end{smallmatrix}\right)}

\def\multiargrekursiverbefehl#1#2#3#4#5#6#7#8{%
    \expandafter\gdef\csname#1\endcsname #2##1#4{\csname #1@anfang\endcsname##1#3\egroup}
    \expandafter\def\csname #1@anfang\endcsname##1#3{#5##1\@ifnextchar\egroup{\csname #1@ende\endcsname}{#7\csname #1@mitte\endcsname}}
    \expandafter\def\csname #1@mitte\endcsname##1#3{#6##1\@ifnextchar\egroup{\csname #1@ende\endcsname}{#7\csname #1@mitte\endcsname}}
    \expandafter\def\csname #1@ende\endcsname##1{#8}
}
\multiargrekursiverbefehl{svektor}{[}{;}{]}{\begin{smatrix}}{}{\\}{\\\end{smatrix}}
\multiargrekursiverbefehl{vektor}{[}{;}{]}{\begin{matrix}{c}}{}{\\}{\\\end{matrix}}
\multiargrekursiverbefehl{vektorzeile}{}{,}{;}{}{&}{}{}
\multiargrekursiverbefehl{matlabmatrix}{[}{;}{]}{\begin{smatrix}\vektorzeile}{\vektorzeile}{;\\}{;\end{smatrix}}

\def\BeweisRichtung[#1]{\@ifnextchar\bgroup{\@BeweisRichtung@c[#1]}{\@BeweisRichtung@bes[#1]}}
    \def\@BeweisRichtung@bes[#1]{{\bfseries (#1)}}
    \def\@BeweisRichtung@c[#1]#2#3{#2~#1~#3}
\def\erzeugeBeweisRichtungBefehle#1#2{
    \expandafter\gdef\csname #1text\endcsname##1##2{\BeweisRichtung[#2]{##1}{##2}}
    \expandafter\gdef\csname #1\endcsname{%
        \@ifnextchar\bgroup{\csname #1@\endcsname}{\csname #1text\endcsname{}{}}%
    }
    \expandafter\gdef\csname #1@\endcsname##1##2{%
        \csname #1text\endcsname{\punktcref{##1}}{\punktcref{##2}}%
    }
}
\erzeugeBeweisRichtungBefehle{hinRichtung}{$\Rightarrow$}
\erzeugeBeweisRichtungBefehle{herRichtung}{$\Leftarrow$}
\erzeugeBeweisRichtungBefehle{hinherRichtung}{$\Leftrightarrow$}

\def\brkt#1{\langle{}#1{}\rangle}
\def\mathfrak#1{\mbox{\usefont{U}{euf}{m}{n}#1}}

\def\rectangleblack{\text{\RectangleBold}}

\def\squareblack{\blacksquare}

\def\create@abbreviation#1#2{
    \expandafter\gdef\csname #1\endcsname{%
        #2\@ifnextchar.{%
            \relax\ifmmode\else\expandafter\@gobble\fi%
        }{%
            \relax\ifmmode\else\@\xspace\fi%
        }%
    }
}

\create@abbreviation{cf}{\emph{cf.}}
\create@abbreviation{Cf}{\emph{Cf.}}
\create@abbreviation{idest}{\emph{i.e.}}
\create@abbreviation{Idest}{\emph{I.e.}}
\create@abbreviation{exempli}{\emph{e.g.}}
\create@abbreviation{Exempli}{\emph{E.g.}}
\create@abbreviation{etcetera}{\emph{etc.}}
\create@abbreviation{etAlia}{\emph{et al.}}
\create@abbreviation{viz}{\emph{viz.}}

\create@abbreviation{Tfae}{T.\,f.\,a.\,e.}
\create@abbreviation{tfae}{t.\,f.\,a.\,e.}
\create@abbreviation{wrt}{wrt.}
\create@abbreviation{randomvar}{r.\,v.}
\create@abbreviation{iid}{i.\,i.\,d.}
\create@abbreviation{almostEverywhere}{a.\,e.}
\create@abbreviation{almostSurely}{a.\,s.}
\create@abbreviation{withoutlog}{w.\,l.\,o.\,g.}
\create@abbreviation{Withoutlog}{W.\,l.\,o.\,g.}
\create@abbreviation{respectively}{resp.}

\def\crefname@full#1#2#3#4#5{%
    \crefname{#1}{#2}{#3}
    \Crefname{#1}{#4}{#5}
}
\def\crefname@fullmod#1#2#3#4#5{%
    \crefname@full{#1}{#2}{#3}{#4}{#5}
    \crefname@full{#1@basic}{#2}{#3}{#4}{#5}
    \crefname@full{#1@withName}{#2}{#3}{#4}{#5}
}
\crefname@full{chapter}{chapter}{chapters}{Chapter}{Chapters}
\crefname@full{appendix}{appendix}{appendices}{Appendix}{Appendices}
\crefname@full{section}{section}{sections}{Section}{Sections}
\crefname@full{subsection}{section}{sections}{Section}{Sections}
\crefname@full{subsubsection}{section}{sections}{Section}{Sections}
\crefname@full{subsubsubsection}{section}{sections}{Section}{Sections}
\crefname@full{table}{table}{tables}{Table}{Tables}
\crefname@full{figure}{figure}{figures}{Figure}{Figures}
\crefname@full{subfigure}{figure}{figures}{Figure}{Figures}

\crefname@fullmod{thm}{theorem}{theorems}{Theorem}{Theorems}
\crefname@fullmod{thmStar}{theorem}{theorems}{Theorem}{Theorems}
\crefname@fullmod{conj}{conjecture}{conjectures}{Conjecture}{Conjectures}
\crefname@fullmod{cor}{corollary}{corollaries}{Corollary}{Corollaries}
\crefname@fullmod{defn}{definition}{definitions}{Definition}{Definitions}
\crefname@fullmod{conv}{convention}{conventions}{Convention}{Conventions}
\crefname@fullmod{e.g.}{example}{examples}{Example}{Examples}
\crefname@fullmod{prop}{proposition}{propositions}{Proposition}{Propositions}
\crefname@fullmod{proof}{proof}{proofs}{Proof}{Proofs}
\crefname@fullmod{lemm}{lemma}{lemmata}{Lemma}{Lemmata}
\crefname@fullmod{problem}{problem}{problems}{Problem}{Problems}
\crefname@fullmod{qstn}{question}{questions}{Question}{Questions}
\crefname@fullmod{rem}{remark}{remarks}{Remark}{Remarks}

\def\qedEIGEN#1{\@ifnextchar[{\qedEIGEN@c{#1}}{\qedEIGEN@bes{#1}}}
\def\qedEIGEN@bes#1{%
    \parfillskip=0pt
    \widowpenalty=10000
    \displaywidowpenalty=10000
    \finalhyphendemerits=0
    \leavevmode
    \unskip
    \nobreak
    \hfil
    \penalty50
    \hskip.2em
    \null
    \hfill
    #1
    \par%
}
\def\qedEIGEN@c#1[#2]{%
    \parfillskip=0pt
    \widowpenalty=10000
    \displaywidowpenalty=10000
    \finalhyphendemerits=0
    \leavevmode
    \unskip
    \nobreak
    \hfil
    \penalty50
    \hskip.2em
    \null
    \hfill
    {#1~{\small\bfseries\upshape (#2)}}%
    \par%
}
\def\qedVARIANT#1#2{
    \expandafter\def\csname ennde#1Sign\endcsname{#2}
    \expandafter\def\csname ennde#1\endcsname{\@ifnextchar[{\qedEIGEN@c{#2}}{\qedEIGEN@bes{#2}}} 
}
\qedVARIANT{OfProof}{$\squareblack$}
\qedVARIANT{OfWork}{\rectangleblack}
\qedVARIANT{OfSomething}{$\dashv$}
\qedVARIANT{OnNeutral}{} 

\def\ra@pretheoremwork{
    \setlength{\theorempreskipamount}{\ownspaceabovethm}
    \setlength{\theorempostskipamount}{\ownspacebelowthm}
}
\def\rathmtransfer#1#2{
    \expandafter\def\csname #2\endcsname{\csname #1\endcsname}
    \expandafter\def\csname end#2\endcsname{\csname end#1\endcsname}
}

\def\ranewthm#1#2#3[#4]{
    \theoremstyle{\current@theoremstyle}
    \theoremseparator{\current@theoremseparator}
    \theoremprework{\ra@pretheoremwork}
    \@ifundefined{#1@basic}{\newtheorem{#1@basic}[#4]{#2}}{\renewtheorem{#1@basic}[#4]{#2}}
    \theoremstyle{\current@theoremstyle}
    \theoremseparator{\thmForceSepPt}
    \theoremprework{\ra@pretheoremwork}
    \@ifundefined{#1@withName}{\newtheorem{#1@withName}[#4]{#2}}{\renewtheorem{#1@withName}[#4]{#2}}
    \theoremstyle{nonumberplain}
    \theoremseparator{\thmForceSepPt}
    \theoremprework{\ra@pretheoremwork}
    \@ifundefined{#1@star@basic}{\newtheorem{#1@star@basic}[#4]{#2}}{\renewtheorem{#1@star@basic}[#4]{#2}}
    \theoremstyle{nonumberplain}
    \theoremseparator{\thmForceSepPt}
    \theoremprework{\ra@pretheoremwork}
    \@ifundefined{#1@star@withName}{\newtheorem{#1@star@withName}[#4]{#2}}{\renewtheorem{#1@star@withName}[#4]{#2}}
    \umbauenenv{#1}{#3}[#4]
    \umbauenenv{#1@star}{#3}[#4]
    \rathmtransfer{#1@star}{#1*}
}

\def\umbauenenv#1#2[#3]{%
    \expandafter\def\csname #1\endcsname{\relax%
        \@ifnextchar[{\csname #1@\endcsname}{\csname #1@\endcsname[*]}%
    }
    \expandafter\def\csname #1@\endcsname[##1]{\relax%
        \@ifnextchar[{\csname #1@@\endcsname[##1]}{\csname #1@@\endcsname[##1][*]}%
    }
    \expandafter\def\csname #1@@\endcsname[##1][##2]{%
        \ifx*##1%
            \def\enndeOfBlock{\csname end#1@basic\endcsname}
            \csname #1@basic\endcsname%
        \else%
            \def\enndeOfBlock{\csname end#1@withName\endcsname}
            \csname #1@withName\endcsname[##1]%
        \fi%
        \def\makelabel####1{%
            \gdef\beweislabel{####1}%
            \label{\beweislabel}%
        }%
        \ifx*##2%
            \def\enndeSymbol{\qedEIGEN{#2}}
        \else%
            \def\enndeSymbol{\qedEIGEN{#2}[##2]}
        \fi
    }
    \expandafter\gdef\csname end#1\endcsname{\enndeSymbol\enndeOfBlock}
}

    \def\current@theoremstyle{plain}
    \def\current@theoremseparator{\thmnumberingseppt}
    \theoremstyle{\current@theoremstyle}
    \theoremseparator{\current@theoremseparator}
    \theoremsymbol{}

\generatenestedthmnumbering{arabic}{section}{equation}

\generatenestedthmnumbering{arabic}{section}{X}
\generatenestedthmnumbering{Roman}{section}{Xsp}

    \theoremheaderfont{\upshape\bfseries}
    \theorembodyfont{\slshape}

\ranewthm{thm}{Theorem}{\enndeOnNeutralSign}[X]
\ranewthm{lemm}{Lemma}{\enndeOnNeutralSign}[X]
\ranewthm{cor}{Corollary}{\enndeOnNeutralSign}[X]
\ranewthm{prop}{Proposition}{\enndeOnNeutralSign}[X]

    \theorembodyfont{\upshape}

\ranewthm{defn}{Definition}{\enndeOnNeutralSign}[X]
\ranewthm{conv}{Convention}{\enndeOnNeutralSign}[X]
\ranewthm{e.g.}{Example}{\enndeOnNeutralSign}[X]
\ranewthm{fact}{Fact}{\enndeOnNeutralSign}[X]
\ranewthm{problem}{Problem}{\enndeOnNeutralSign}[X]
\ranewthm{qstn}{Question}{\enndeOnNeutralSign}[X]
\ranewthm{rem}{Remark}{\enndeOnNeutralSign}[X]

    \theoremheaderfont{\itshape}
    \theorembodyfont{\upshape}

\ranewthm{proof@tmp}{Proof}{\enndeOfProofSign}[Xdisplaynone]
\rathmtransfer{proof@tmp*}{proof}

\def\shortclaim@claim{%
    \iflanguage{british}{Claim}{%
    \iflanguage{english}{Claim}{%
    \iflanguage{ngerman}{Behauptung}{%
    \iflanguage{russian}{Утверждение}{%
    Claim%
    }}}}%
}
\def\shortclaim@pf@kurz{%
    \iflanguage{british}{Pf}{%
    \iflanguage{english}{Pf}{%
    \iflanguage{ngerman}{Bew}{%
    \iflanguage{russian}{Доказательство}{%
    Pf%
    }}}}%
}
\def\shortclaim{\@ifnextchar\bgroup{\shortclaim@c}{\shortclaim@bes}}
    \def\shortclaim@c#1{\item[{\bfseries \shortclaim@claim\forceaddspace #1.}]}
    \def\shortclaim@bes{\item[{\bfseries \shortclaim@claim.}]}
\def\proofofshortclaim{\item[{\bfseries\itshape\shortclaim@pf@kurz.}]}

\newdateformat{standardshort}{\oldstylenums{\THEYEAR}.\oldstylenums{\THEMONTH}.\oldstylenums{\THEDAY}}
\newdateformat{standardcompact}{\THEYEAR\twodigit{\THEMONTH}\twodigit{\THEDAY}}
\newdateformat{standardlong}{\THEYEAR\ \monthname\ \THEDAY}
\newcolumntype{\RECHTS}[1]{>{\raggedleft}p{#1}}
\newcolumntype{\LINKS}[1]{>{\raggedright}p{#1}}
\newcolumntype{m}{>{$}l<{$}}
\newcolumntype{C}{>{$}c<{$}}
\newcolumntype{L}{>{$}l<{$}}
\newcolumntype{R}{>{$}r<{$}}
\newcolumntype{0}{@{\hspace{0pt}}}
\newcolumntype{\LINKSRAND}{@{\hspace{\@totalleftmargin}}}
\newcolumntype{h}{@{\extracolsep{\fill}}}
\newcolumntype{i}{>{\itshape}}
\newcolumntype{t}{@{\hspace{\tabcolsep}}}
\newcolumntype{q}{@{\hspace{1em}}}
\newcolumntype{n}{@{\hspace{-\tabcolsep}}}
\newcolumntype{M}[2]{%
    >{\begin{minipage}{#2}\begin{math}}%
    {#1}%
    <{\end{math}\end{minipage}}%
}
\newcolumntype{T}[2]{%
    >{\begin{minipage}{#2}}%
    {#1}%
    <{\end{minipage}}%
}
\def\punkteumgebung@genbefehl#1#2#3{
    \punkteumgebung@genbefehl@{#1}{#2}{#3}{}{}
    \punkteumgebung@genbefehl@{multi#1}{#2}{#3}{
        \setlength{\columnsep}{10pt}%
        \setlength{\columnseprule}{0pt}%
        \begin{multicols}{\thecolumnanzahl}%
    }{\end{multicols}\nvraum{1}}
}
\def\punkteumgebung@genbefehl@#1#2#3#4#5{
    \expandafter\gdef\csname #1\endcsname{
        \@ifnextchar\bgroup{\csname #1@c\endcsname}{\csname #1@bes\endcsname}
    }
        \expandafter\def\csname #1@c\endcsname##1{
            \@ifnextchar[{\csname #1@c@\endcsname{##1}}{\csname #1@c@\endcsname{##1}[\z@]}
        }
        \expandafter\def\csname #1@c@\endcsname##1[##2]{
            \@ifnextchar[{\csname #1@c@@\endcsname{##1}[##2]}{\csname #1@c@@\endcsname{##1}[##2][\z@]}
        }
        \expandafter\def\csname #1@c@@\endcsname##1[##2][##3]{
            \let\alterlinkerRand\gesamtlinkerRand
            \let\alterrechterRand\gesamtrechterRand
            \addtolength{\gesamtlinkerRand}{##2}
            \addtolength{\gesamtrechterRand}{##3}
            \advance\linewidth -##2%
            \advance\linewidth -##3%
            \advance\@totalleftmargin ##2%
            \parshape\@ne \@totalleftmargin\linewidth%
            #4
            \begin{#2}[\upshape ##1]%
                \setlength{\parskip}{0.5\baselineskip}\relax%
                \setlength{\topsep}{\z@}\relax%
                \setlength{\partopsep}{\z@}\relax%
                \setlength{\parsep}{\parskip}\relax%
                \setlength{\itemsep}{#3}\relax%
                \setlength{\listparindent}{\z@}\relax%
                \setlength{\itemindent}{\z@}\relax%
        }
        \expandafter\def\csname #1@bes\endcsname{
            \@ifnextchar[{\csname #1@bes@\endcsname}{\csname #1@bes@\endcsname[\z@]}
        }
        \expandafter\def\csname #1@bes@\endcsname[##1]{
            \@ifnextchar[{\csname #1@bes@@\endcsname[##1]}{\csname #1@bes@@\endcsname[##1][\z@]}
        }
        \expandafter\def\csname #1@bes@@\endcsname[##1][##2]{
            \let\alterlinkerRand\gesamtlinkerRand
            \let\alterrechterRand\gesamtrechterRand
            \addtolength{\gesamtlinkerRand}{##1}
            \addtolength{\gesamtrechterRand}{##2}
            \advance\linewidth -##1%
            \advance\linewidth -##2%
            \advance\@totalleftmargin ##1%
            \parshape\@ne \@totalleftmargin\linewidth%
            #4
            \begin{#2}%
                \setlength{\parskip}{0.5\baselineskip}\relax%
                \setlength{\topsep}{\z@}\relax%
                \setlength{\partopsep}{\z@}\relax%
                \setlength{\parsep}{\parskip}\relax%
                \setlength{\itemsep}{#3}\relax%
                \setlength{\listparindent}{\z@}\relax%
                \setlength{\itemindent}{\z@}\relax%
        }
    \expandafter\gdef\csname end#1\endcsname{%
        \end{#2}#5
        \setlength{\gesamtlinkerRand}{\alterlinkerRand}
        \setlength{\gesamtlinkerRand}{\alterrechterRand}
    }
}

\def\ritempunkt{{\Large \textbullet}} 
\setdefaultitem{\ritempunkt}{\ritempunkt}{\ritempunkt}{\ritempunkt}
\punkteumgebung@genbefehl{itemise}{compactitem}{\parskip}{}{}
\punkteumgebung@genbefehl{kompaktitem}{compactitem}{\z@}{}{}
\punkteumgebung@genbefehl{enumerate}{compactenum}{\parskip}{}{}
\punkteumgebung@genbefehl{kompaktenum}{compactenum}{\z@}{}{}

\def\shorteqnarray{%
    \bgroup
    \setlength{\abovedisplayshortskip}{\aboveequation}%
    \setlength{\belowdisplayshortskip}{\belowequation}%
    \setlength{\abovedisplayskip}{\aboveequation}%
    \setlength{\belowdisplayskip}{\belowequation}%
    \begin{eqnarray*}%
}
\def\endshorteqnarray{%
    \end{eqnarray*}%
    \egroup
}
\def\longeqnarray{%
    \bgroup%
    \allowdisplaybreaks%
    \setlength{\abovedisplayshortskip}{\aboveequation}%
    \setlength{\belowdisplayshortskip}{\belowequation}%
    \setlength{\abovedisplayskip}{\aboveequation - \baselineskip}%
    \setlength{\belowdisplayskip}{\belowequation}%
    \begin{eqnarray*}
}
\def\endlongeqnarray{%
    \end{eqnarray*}%
    \egroup%
}
\def\displayarray[#1]#2{
    \bgroup
    \everymath={\displaystyle}
    \begin{array}[#1]{#2}
}
\def\enddisplayarray{
    \end{array}
    \egroup
}

\def\matrix#1{\left(\begin{array}[mc]{#1}}
    \def\endmatrix{\end{array}\right)}
\def\smatrix{\left(\begin{smallmatrix}}
    \def\endsmatrix{\end{smallmatrix}\right)}

\def\multiargrekursiverbefehl#1#2#3#4#5#6#7#8{%
    \expandafter\gdef\csname#1\endcsname #2##1#4{\csname #1@anfang\endcsname##1#3\egroup}
    \expandafter\def\csname #1@anfang\endcsname##1#3{#5##1\@ifnextchar\egroup{\csname #1@ende\endcsname}{#7\csname #1@mitte\endcsname}}
    \expandafter\def\csname #1@mitte\endcsname##1#3{#6##1\@ifnextchar\egroup{\csname #1@ende\endcsname}{#7\csname #1@mitte\endcsname}}
    \expandafter\def\csname #1@ende\endcsname##1{#8}
}
\multiargrekursiverbefehl{svektor}{[}{;}{]}{\begin{smatrix}}{}{\\}{\\\end{smatrix}}
\multiargrekursiverbefehl{vektor}{[}{;}{]}{\begin{matrix}{c}}{}{\\}{\\\end{matrix}}
\multiargrekursiverbefehl{vektorzeile}{}{,}{;}{}{&}{}{}
\multiargrekursiverbefehl{matlabmatrix}{[}{;}{]}{\begin{smatrix}\vektorzeile}{\vektorzeile}{;\\}{;\end{smatrix}}

\def\underbracenodisplay#1{%
    \mathop{\vtop{\m@th\ialign{##\crcr
    $\hfil\displaystyle{#1}\hfil$\crcr
    \noalign{\kern3\p@\nointerlineskip}%
    \upbracefill\crcr\noalign{\kern3\p@}}}}\limits%
}

\def\changemargins{\@ifnextchar[{\indents@}{\indents@[\z@]}}
\def\indents@[#1]{\@ifnextchar[{\indents@@[#1]}{\indents@@[#1][\z@]}}
\def\indents@@[#1][#2]{%
    \begin{list}{}{\relax
        \setlength{\topsep}{\z@}\relax
        \setlength{\partopsep}{\z@}\relax
        \setlength{\parsep}{\parskip}\relax
        \setlength{\listparindent}{\z@}\relax
        \setlength{\itemindent}{\z@}\relax
        \setlength{\leftmargin}{#1}\relax
        \setlength{\rightmargin}{#2}\relax
        \let\alterlinkerRand\gesamtlinkerRand
        \let\alterrechterRand\gesamtrechterRand
        \addtolength{\gesamtlinkerRand}{#1}
        \addtolength{\gesamtrechterRand}{#2}
    }\relax
        \item[]\relax
}
    \def\endchangemargins{%
        \setlength{\gesamtlinkerRand}{\alterlinkerRand}
        \setlength{\gesamtlinkerRand}{\alterrechterRand}
        \end{list}%
    }

\def\indentonce{\begin{changemargins}[\rtab][\rtab]}
    \def\endindentonce{\end{changemargins}}

\def\restoremargins{\begin{changemargins}[-\gesamtlinkerRand][-\gesamtrechterRand]}
    \def\endrestoremargins{\end{changemargins}}

\def\programmiercode{
    \modulolinenumbers[1]
    \begin{changemargins}[\rtab][\rtab]%
    \begin{linenumbers}%
        \fontfamily{cmtt}\fontseries{m}\fontshape{u}\selectfont%
        \setlength{\parskip}{1\baselineskip}%
        \setlength{\parindent}{0pt}%
}
    \def\endprogrammiercode{
        \end{linenumbers}
        \end{changemargins}
    }

\def\schattiertebox@genbefehl#1#2#3{
    \expandafter\gdef\csname #1\endcsname{%
        \@ifnextchar[{\csname #1@args\endcsname}{\csname #1@args\endcsname[#3]}
    }
        \expandafter\def\csname #1@args\endcsname[##1]{%
            \@ifnextchar[{\csname #1@args@l\endcsname[##1]}{\csname #1@args@n\endcsname[##1]}
        }
        \expandafter\def\csname #1@args@l\endcsname[##1][##2]{%
            \@ifnextchar[{\csname #1@args@l@r\endcsname[##1][##2]}{\csname #1@args@l@n\endcsname[##1][##2]}
        }
        \expandafter\def\csname #1@args@n\endcsname[##1]{%
            \let\boolinmdframed\boolwahr
            \begin{mdframed}[#2leftmargin=0,rightmargin=0,outermargin=0,innermargin=0,##1]
        }
        \expandafter\def\csname #1@args@l@n\endcsname[##1][##2]{%
            \let\boolinmdframed\boolwahr
            \begin{mdframed}[#2leftmargin=##2/2,rightmargin=##2/2,outermargin=##2/2,innermargin=##2/2,##1]
        }
        \expandafter\def\csname #1@args@l@r\endcsname[##1][##2][##3]{%
            \let\boolinmdframed\boolwahr
            \begin{mdframed}[#2leftmargin=##2,rightmargin=##3,outermargin=##2,innermargin=##3,##1]
        }
    \expandafter\gdef\csname end#1\endcsname{%
        \end{mdframed}
        \let\boolinmdframed\boolfalsch
    }
}
    \schattiertebox@genbefehl{schattiertebox}{
        splittopskip=0,%
        splitbottomskip=0,%
        frametitleaboveskip=0,%
        frametitlebelowskip=0,%
        skipabove=1\baselineskip,%
        skipbelow=1\baselineskip,%
        linewidth=2pt,%
        linecolor=black,%
        roundcorner=4pt,%
    }{
        backgroundcolor=leer,%
        nobreak=true,%
    }

    \schattiertebox@genbefehl{schattierteboxdunn}{
        splittopskip=0,%
        splitbottomskip=0,%
        frametitleaboveskip=0,%
        frametitlebelowskip=0,%
        skipabove=1\baselineskip,%
        skipbelow=1\baselineskip,%
        linewidth=1pt,%
        linecolor=black,%
        roundcorner=2pt,%
    }{
        backgroundcolor=leer,%
        nobreak=true,%
    }

\def\tikzsetzepfeil#1{%
    \begin{tikzpicture}[remember picture,overlay,>=latex]%
        \draw #1;%
    \end{tikzpicture}%
}

\def\tikzsetzekreise[#1]#2#3{%
    \tikzsetzepfeil{%
    [rounded corners,#1]%
        ([shift={(-\tabcolsep,0.75\baselineskip)}]#2)%
        rectangle%
        ([shift={(\tabcolsep,-0.5\baselineskip)}]#3)
    }%
}

\tikzset{
    >=stealth,
    auto,
    node distance=1cm,
    thick,
    main node/.style={
        circle,draw,font=\sffamily\Large\bfseries,minimum size=0pt
    },
    state/.style={minimum size=0pt}
    loop above right/.style={loop,out=30,in=60,distance=0.5cm},
    loop above left/.style={above left,out=150,in=120,loop},
    loop below right/.style={below right,out=330,in=300,loop},
    loop below left/.style={below left,out=240,in=210,loop},
    itria/.style={
        draw,dashed,shape border uses incircle,
        isosceles triangle,shape border rotate=90,yshift=-1.45cm
    },
    rtria/.style={
        draw,dashed,shape border uses incircle,
        isosceles triangle,isosceles triangle apex angle=90,
        shape border rotate=-45,yshift=0.2cm,xshift=0.5cm
    },
    ritria/.style={
        draw,dashed,shape border uses incircle,
        isosceles triangle,isosceles triangle apex angle=110,
        shape border rotate=-55,yshift=0.1cm
    },
    litria/.style={
        draw,dashed,shape border uses incircle,
        isosceles triangle,isosceles triangle apex angle=110,
        shape border rotate=235,yshift=0.1cm
    }
}

\renewenvironment{cases}[0]{\left\{\begin{array}[c]{0lcl}}{\end{array}\right.}

\providecommand{\usesinglequotes}{}
\renewcommand{\usesinglequotes}[1]{`#1'}
\providecommand{\zeroone}{}
\renewcommand{\zeroone}[0]{\textup{0\=/1}\forceaddspace}
\providecommand{\onetoone}{}
\renewcommand{\onetoone}[0]{\ensuremath{1\!\!:\!\!1}\forceaddspace}
\providecommand{\bhatskeide}{}
\renewcommand{\bhatskeide}[0]{\text{B--Sk}}
\providecommand{\envPreMathsLong}{}
\renewcommand{\envPreMathsLong}[0]{%
    \bgroup\relax%
    \let\old@arraystretch\arraystretch\relax%
    \renewcommand\arraystretch{1.2}\relax\relax%
}
\providecommand{\envPostMathsLong}{}
\renewcommand{\envPostMathsLong}[0]{%
    \renewcommand\arraystretch{\old@arraystretch}\relax%
    \egroup\relax%
}
\providecommand{\complex}{}
\renewcommand{\complex}[0]{\mathbb{C}}
\providecommand{\Torus}{}
\renewcommand{\Torus}[0]{\mathbb{T}}

\providecommand{\reals}{}
\renewcommand{\reals}[0]{\mathbb{R}}
\providecommand{\realsNonNeg}{}
\renewcommand{\realsNonNeg}[0]{\reals_{\geq 0}}

\providecommand{\rationals}{}
\renewcommand{\rationals}[0]{\mathbb{Q}}

\providecommand{\integers}{}
\renewcommand{\integers}[0]{\mathbb{Z}}
\providecommand{\naturals}{}
\renewcommand{\naturals}[0]{\mathbb{N}}
\providecommand{\naturalsZero}{}
\renewcommand{\naturalsZero}[0]{\mathbb{N}_{0}}
\providecommand{\naturalsPos}{}
\renewcommand{\naturalsPos}[0]{\mathbb{N}}
\providecommand{\HilbertRaum}{}
\renewcommand{\HilbertRaum}[0]{\mathcal{H}}
\providecommand{\BanachRaum}{}
\renewcommand{\BanachRaum}[0]{\mathcal{E}}

\providecommand{\topSOT}{}
\renewcommand{\topSOT}[0]{\text{\upshape\scshape sot}}
\providecommand{\topWOT}{}
\renewcommand{\topWOT}[0]{\text{\upshape\scshape wot}}

\providecommand{\topPW}{}
\renewcommand{\topPW}[0]{\text{\upshape\scshape pw}}
\providecommand{\card}{}
\renewcommand{\card}[1]{\lvert #1 \rvert}

\providecommand{\einser}{}
\renewcommand{\einser}[0]{1\!\!1}

\providecommand{\floor}{}
\renewcommand{\floor}[1]{{\lfloor #1 \rfloor}}

\providecommand{\fractional}{}
\renewcommand{\fractional}[1]{\{\!\!\{#1\}\!\!\}}
\providecommand{\iunit}{}
\renewcommand{\iunit}[0]{\imath}
\providecommand{\abs}{}
\renewcommand{\abs}[1]{\lvert #1 \rvert}

\providecommand{\linspann}{}
\renewcommand{\linspann}[0]{\textup{lin}}
\providecommand{\adjoint}{}
\renewcommand{\adjoint}[0]{\text{\upshape ad}}

\providecommand{\onematrix}{}
\renewcommand{\onematrix}[0]{\text{\upshape\bfseries I}}
\providecommand{\zeromatrix}{}
\renewcommand{\zeromatrix}[0]{\mathbf{0}}
\providecommand{\onevector}{}
\renewcommand{\onevector}[0]{\mathbf{1}}
\providecommand{\zerovector}{}
\renewcommand{\zerovector}[0]{\mathbf{0}}

\providecommand{\brkt}{}
\renewcommand{\brkt}[2]{\langle{}#1,\:#2{}\rangle}

\providecommand{\brktLong}{}
\renewcommand{\brktLong}[2]{\Big\langle{}#1,\:#2{}\Big\rangle}

\providecommand{\norm}{}
\renewcommand{\norm}[1]{\lVert #1 \rVert}
\providecommand{\normLong}{}
\renewcommand{\normLong}[1]{\Big\| #1 \Big\|}

\providecommand{\BoundedOpsSymbol}{}
\renewcommand{\BoundedOpsSymbol}[0]{\mathfrak{L}}
\providecommand{\Cnought}{}
\renewcommand{\Cnought}[0]{\mathcal{C}_{0}}

\providecommand{\Repr}{}
\renewcommand{\Repr}[2]{\mathrm{Repr}(#1 \!:\! #2)}
\providecommand{\OpSpaceC}{}
\renewcommand{\OpSpaceC}[1]{\mathcal{C}(#1)}
\providecommand{\OpSpaceU}{}
\renewcommand{\OpSpaceU}[1]{\mathcal{U}(#1)}

\providecommand{\SpHomCs}{}
\renewcommand{\SpHomCs}[0]{\mathcal{S}^{c}_{s}}

\providecommand{\SpHomUs}{}
\renewcommand{\SpHomUs}[0]{\mathcal{S}^{u}_{s}}

\providecommand{\similarToUnitary}{}
\renewcommand{\similarToUnitary}[0]{\mathrel{\sim_{\textup{u}}}}

\providecommand{\orbitUnitary}{}
\renewcommand{\orbitUnitary}[1]{[#1]_{u}}

\providecommand{\universalElementsUnitary}{}
\renewcommand{\universalElementsUnitary}[1]{\Omega_{u}(#1)}

\def\Cts{\@ifnextchar_{\Cts@tief}{\Cts@tief_{}}}
    \def\Cts@tief_#1#2{\@ifnextchar\bgroup{\Cts@two_{#1}{#2}}{\Cts@one_{#1}{#2}}}
    \def\Cts@one_#1#2{C_{#1}\big(#2\big)}
    \def\Cts@two_#1#2#3{C_{#1}\big(#2,~#3\big)}

\def\BoundedOps#1{\@ifnextchar\bgroup{\BoundedOps@two{#1}}{\mathop{\BoundedOpsSymbol}(#1)}}
    \def\BoundedOps@two#1#2{\mathop{\BoundedOpsSymbol}(#1,#2)}
\def\BoundedOpsInv#1{\@ifnextchar\bgroup{\BoundedOps@two{#1}}{\mathop{\BoundedOpsSymbol}(#1)^{\times}}}
    \def\BoundedOpsInv@two#1#2{\mathop{\BoundedOpsSymbol}(#1,#2)^{\times}}

\def\id{\mathrm{\textit{id}}}

\def\restr#1{\vert_{#1}}
\def\without{\mathbin{\setminus}}

\def\eps{\varepsilon}
\let\altphi\phi
\let\altvarphi\varphi
    \def\phi{\altvarphi}
    \def\varphi{\altphi}

\def\quer#1{\overline{#1}}

\def\lim{\mathop{\ell\mathrm{im}}}
\def\supp{\mathop{\textup{supp}}}
\def\dim{\mathop{\textup{dim}}}

\def\ran{\mathop{\textup{ran}}}

\def\tinytopWOT{\text{\scriptsize\upshape \scshape wot}}

\def\toplocWOT{\text{{\itshape k}}_{\text{\tiny\upshape \scshape wot}}}

\def\tinytoplocWOT{\text{\scriptsize{{\itshape k}}-{\upshape \scshape wot}}}
\def\tinytopSOT{\text{\scriptsize\upshape \scshape sot}}

\def\toplocSOT{\text{{\itshape k}}_{\text{\tiny\upshape \scshape sot}}}

\def\tinytoplocSOT{\text{\scriptsize{{\itshape k}}-{\upshape \scshape sot}}}

\def\tinytopPW{\text{\scriptsize\upshape \scshape pw}}

\@ifundefined{setcitestyle}{%
}{%
    \setcitestyle{numeric-comp,open={[},close={]}}
}

\allowdisplaybreaks 
\renewcommand{\arraystretch}{1}
\def\firstparagraph{\noindent}
\def\continueparagraph{\noindent}

    \generatenestedsecnumbering{arabic}{section}{subsection}
    \generatenestedsecnumbering{arabic}{subsection}{subsubsection}
    \def\theunitnamesection{\thesection}

    \def\sectionname{}

    \let\appendix@orig\appendix
    \def\appendix{%
        \appendix@orig%
        \let\boolinappendix\boolwahr
        \addcontentsline{toc}{part}{\appendixname}%
        \addtocontents{toc}{\protect\setcounter{tocdepth}{0}}
        \def\sectionname{Appendix}%
        \def\theunitnamesection{\Alph{section}}%
    }
    \def\notappendix{%
        \let\boolinappendix\boolfalse
        \addtocontents{toc}{\protect\setcounter{tocdepth}{1 }}
        \def\sectionname{}%
        \def\theunitnamesection{\arabic{section}}%
    }

    \def\@seccntformat#1{%
        \protect\textup{%
            \protect\@secnumfont
            \expandafter\protect\csname format#1\endcsname%
            \csname the#1\endcsname
            \expandafter\protect\csname format#1@pt\endcsname%
            \space
        }%
    }

    \def\formatsection@text{\centering\Large\scshape}
    \def\formatsection@pt{\secnumberingseppt}
    \def\section{\@startsection{section}{1}{\z@}{.7\linespacing\@plus\linespacing}{.5\linespacing}{\formatsection@text}}

    \def\formatsubsection@text{\flushleft\bfseries\scshape}
    \def\formatsubsection@pt{\subsecnumberingseppt}
    \def\subsection{\@startsection{subsection}{2}{\z@}{\z@}{\z@\hspace{1em}}{\formatsubsection@text}}

    \renewcommand{\paragraph}[1]{%
        {\bfseries #1}\:%
    }

\DefineFNsymbols*{customAlphabet}{abcdefghijklmnopqrstuvwxyz}
\setfnsymbol{customAlphabet}

\def\footnotemark[#1]{\text{\textsuperscript{\getrefnumber{#1}}}}

\def\rafootnotectr{20}
\providecommand{\incrftnotectr}{}
\renewcommand{\incrftnotectr}[1]{%
    \addtocounter{#1}{1}%
    \ifnum\value{#1}>\rafootnotectr\relax
        \setcounter{#1}{0}%
    \fi%
}
\providecommand{\footnoteref}{}
\renewcommand{\footnoteref}[1]{\protected@xdef\@thefnmark{\ref{#1}}\@footnotemark}
\let\@old@footnotetext\footnotetext
\def\footnotetext[#1]#2{%
    \incrftnotectr{footnote}%
    \@old@footnotetext[\value{footnote}]{\label{#1}#2}%
}

\def\kopfzeiledefault{
    \lhead[]{}
    \lhead[]{}
    \chead[]{}
    \rhead[]{}
    \lfoot[]{}
    \cfoot{\footnotesize\thepage}
    \rfoot[]{}
}

\def\aktuellesfont{\csname lmodern\endcsname}
\def\documentfont{%
    \gdef\aktuellesfont{\csname lmodern\endcsname}%
    \fontfamily{lmr}\fontseries{m}\selectfont%
    \renewcommand{\sfdefault}{phv}%
    \renewcommand{\ttdefault}{pcr}%
    \renewcommand{\rmdefault}{cmr}
    \renewcommand{\bfdefault}{bx}%
    \renewcommand{\itdefault}{it}%
    \renewcommand{\sldefault}{sl}%
    \renewcommand{\scdefault}{sc}%
    \renewcommand{\updefault}{n}%
}

\allowdisplaybreaks

\def\startdocumentlayoutoptions{
    \selectlanguage{british}
    \setlength{\parskip}{0.25\baselineskip}
    \setlength{\parindent}{2em}
    \kopfzeiledefault
    \documentfont
    \normalsize
}

\providecommand{\highlightTerm}{}
\renewcommand{\highlightTerm}[1]{\emph{#1}}

\renewenvironment{highlightForReview}[0]{\color{blue}}{}

\def\addresseshere{%
  \bgroup
  \setlength{\parindent}{0pt}
  \enddoc@text
  \egroup
  \let\enddoc@text\relax
}

\makeatother

\begin{document}
\startdocumentlayoutoptions

\thispagestyle{plain}



\def\abstractname{Abstract}
\begin{abstract}
    We generalise a technique of Bhat and Skeide (2015)
    to interpolate commuting families
        $\{S_{i}\}_{i \in \mathcal{I}}$
    of contractions on a Hilbert space $\HilbertRaum$,
    to commuting families
        $\{T_{i}\}_{i \in \mathcal{I}}$
    of contractive $\Cnought$-semigroups
    on $L^{2}(\prod_{i \in \mathcal{I}}\Torus) \otimes \HilbertRaum$.
    As an excursus, we provide applications of the interpolations
    to time-discretisation and the embedding problem.
    Applied to Parrott's construction (1970),
    we then demonstrate
    for $d \in \naturals$ with $d \geq 3$
    the existence of commuting families
        $\{T_{i}\}_{i=1}^{d}$
    of contractive $\Cnought$-semigroups
    which admit no simultaneous unitary dilation.
    As an application of these counter-examples,
    we obtain the residuality
        \wrt the topology of uniform $\topWOT$-convergence on compact subsets of $\realsNonNeg^{d}$
    of
            non-unitarily dilatable
            and
            non-unitarily approximable
        $d$-parameter contractive $\Cnought$-semigroups
        on separable infinite-dimensional Hilbert spaces
    for each $d \geq 3$.
    Similar results are also developed for $d$-tuples of commuting contractions.

    And by building on the counter-examples of Varopoulos--Kaijser (1973--74),
    a \zeroone-result is obtained for the von Neumann inequality.

    Finally, we discuss applications to rigidity
    as well as the embedding problem,
    \viz that \usesinglequotes{typical} pairs
    of commuting operators can be simultaneously embedded
    into commuting pairs of $\Cnought$-semigroups,
    which extends results of Eisner (2009--10).
\end{abstract}



\def\subjclassname{Mathematics Subject Classification (2020)} 
\subjclass{47A20, 47D06}
\title[Interpolation and non-dilatable families of $\Cnought$-semigroups]{Interpolation and non-dilatable families of $\Cnought$-semigroups}
\author{Raj Dahya}
\email{raj\,[\!\![dot]\!\!]\,dahya\:[\!\![at]\!\!]\:web\,[\!\![dot]\!\!]\,de}
\address{Fakult\"at f\"ur Mathematik und Informatik\newline
Universit\"at Leipzig, Augustusplatz 10, D-04109 Leipzig, Germany}

\maketitle



\setcounternach{section}{1}



\section[Introduction]{Introduction}
\label{sec:introduction:sig:article-interpolation-raj-dahya}

\firstparagraph
Dilation theory concerns itself with the problem
of embedding single operators on Hilbert or Banach spaces
(as well as operator-theoretic entities such as linear maps between trace-class operators on Hilbert spaces)
or parameterised families of operators (such as $\Cnought$-semigroups)
into larger systems which are in some suitable sense \usesinglequotes{reversible}.
One can also consider the problem of finding \textit{non-dilatable} systems.
In the discrete setting,
the existence of non-dilatable commuting families of three or more contractions
on a Hilbert space is well-known
thanks to Parrott's construction \cite{Parrott1970counterExamplesDilation}.
This raises the natural question, whether a continuous analogue for families of semigroups exists.
The present paper answers this positively as follows:
We generalise a relatively recent technique
developed by Bhat and Skeide \cite{BhatSkeide2015PureCounterexamples}
to interpolate families of contractions by families of contractive $\Cnought$-semigroups.
Applying this to Parrott's counter-examples,
we obtain the existence of commuting families
of at least $3$ contractive $\Cnought$-semigroups
which admit no dilations in the sense of Sz.-Nagy \cite{Nagy1953,Nagy1970}.

In addition to looking at the properties of individual families of
contractions or contractive $\Cnought$-semigroups,
we are concerned with spaces thereof,
topologised in a certain weak sense (defined more precisely below).
In this framework, various residuality results about $1$-parameter $\Cnought$-semigroups
and spaces of contractions on Hilbert spaces have been developed,
    \exempli
    \cite{Eisner2008categoryThmStableOpHilbert,Eisnersereny2009catThmStableSemigroups},
    \cite[Theorem~2.2]{Eisner2010typicalContraction},
    \cite[Corollary~3.2]{Krol2009},
    \cite[Theorem~4.1]{Eisnermaitrai2013typicalOperators}
    (%
        which builds on the density result established in
        \cite{Peller1976,Peller1981estimatesOperatorPolyLp}%
    ),
    \cite[Theorem~1.3]{Dahya2022weakproblem}.
Residuality results in the Banach space setting
have also been studied in connection with
the \highlightTerm{invariant subspace problem}
and \highlightTerm{hypercyclicity}
(see
    \cite{Grivaux2021invariantsubspaceLp,Grivaux2021typicalexamples,Grivaux2022localspecLp}%
).
Continuing this pursuit, the present paper develops a general \zeroone-dichotomy,
which we apply to our non-dilatable constructions
to prove within the space of
    $d$-parameter contractive $\Cnought$-semigroups
    on a separable infinite-dimensional Hilbert space,
the residuality of
    dilatable and unitarily approximable
    (\respectively non-dilatable and non-unitarily approximable)
    semigroups
for $d \in \{1,2\}$ (\respectively $d \geq 3$).

In the discrete setting an analogous split is obtained
for spaces of tuples of commuting contractions.
Furthermore, going beyond the question of dilatability,
stronger \zeroone-results are possible.

Based on numerical experimentation,
Orr Shalit has intuited via \emph{concentration of measure} phenomena
that it must be the case that either
most tuples satisfy von Neumann polynomial inequalities
or that most do not.
Our methods appear to confirm this intuition.
Indeed for $d \geq 3$, building on the
results of Varopoulos, Kaijser, \etAlia
\cite{Varopoulos1973Article,Varopoulos1974counterexamples,CrabbDavie1975Article,Shalit2021DilationBook},
we prove that the negative conclusion holds.


\subsection[Definitions and main problems]{Definitions and main problems}
\label{sec:introduction:definitions:sig:article-interpolation-raj-dahya}

\firstparagraph
Let $\HilbertRaum$ be a Hilbert space,
and
    $\BoundedOps{\HilbertRaum},\OpSpaceC{\HilbertRaum},\OpSpaceU{\HilbertRaum}$
denote the spaces of bounded, contractive, and unitary operators
on $\HilbertRaum$ respectively.
Let $\mathcal{I}$ be a non-empty index set,
    $\{S_{i}\}_{i \in \mathcal{I}} \subseteq \OpSpaceC{\HilbertRaum}$,
and
    $\{T_{i}\}_{i \in \mathcal{I}}$
be a commuting family of contractive $\Cnought$-semigroups
on $\HilbertRaum$.
We say that
    $\{S_{i}\}_{i \in \mathcal{I}}$
has a \highlightTerm{power-dilation}
if there is a commuting family
    $\{V_{i}\}_{i \in \mathcal{I}} \subseteq \OpSpaceU{\HilbertRaum^{\prime}}$
of unitaries on a Hilbert space $\HilbertRaum^{\prime}$
as well an isometry $r\in\BoundedOps{\HilbertRaum}{\HilbertRaum^{\prime}}$,
such that

    \begin{displaymath}
        \prod_{i \in \mathcal{I}}S_{i}^{n_{i}}
        = r^{\ast}\Big(\prod_{i \in \mathcal{I}}V_{i}^{n_{i}}\Big)r
    \end{displaymath}

\continueparagraph
holds for all $\mathbf{n} = (n_{i})_{i\in\mathcal{I}}\in\prod_{i\in\mathcal{I}}\naturalsZero$
with $\supp(\mathbf{n}) = \{i \in \mathcal{I} \mid n_{i} \neq 0\}$ finite.%
\footnote{%
    throughout we use the convention $S^{0} \coloneqq \onematrix_{\HilbertRaum}$
    for any bounded operator $S$ on a Hilbert space $\HilbertRaum$.
}
We say that
    $\{T_{i}\}_{i \in \mathcal{I}}$
has a \highlightTerm{simultaneous unitary dilation}
if there exists a commuting family
    $\{U_{i}\}_{i \in \mathcal{I}} \subseteq \Repr{\reals}{\HilbertRaum^{\prime}}$
of $\topSOT$-continuous unitary representations of $\reals$
on a Hilbert space $\HilbertRaum^{\prime}$
as well as an isometry $r\in\BoundedOps{\HilbertRaum}{\HilbertRaum^{\prime}}$,
such that

    \begin{displaymath}
        \prod_{i \in \mathcal{I}}T_{i}(t_{i})
        = r^{\ast}\Big(\prod_{i \in \mathcal{I}}U_{i}(t_{i})\Big)r
    \end{displaymath}

\continueparagraph
holds for all $\mathbf{t} = (t_{i})_{i\in\mathcal{I}}\in\prod_{i\in\mathcal{I}}\realsNonNeg$
with $\supp(\mathbf{t}) = \{i \in \mathcal{I} \mid t_{i} \neq 0\}$ finite.

Generalising slightly, if $(G,\cdot,e)$ is a topological group and $M \subseteq G$ a topological submonoid,
we can consider continuous contractive/unitary homomorphisms
    ${T \colon M \to \BoundedOps{\HilbertRaum}}$
over $M$ on $\HilbertRaum$,
\idest
    $T$ is an $\topSOT$-continuous map
such that
    $T(e) = \onematrix$,
    $T(xy) = T(x)T(y)$ for $x,y \in M$,
    and
    $T(x)$ is contractive/unitary for $x\in M$.
In the case of $(G,M) = (\integers^{d},\naturalsZero^{d})$, $d\in\naturals$,
there is a natural \onetoone-correspondence%
\footnote{%
    via $S(\mathbf{n}) = \prod_{i=1}^{d}S_{i}^{n_{i}}$
    for all $\mathbf{n} = (n_{i})_{i=1}^{d} \in \naturalsZero^{d}$
    and $S_{i} = S(0,0,\ldots,\underset{i}{1},\ldots,0)$
    for $i\in\{1,2,\ldots,d\}$ and all $n\in\naturalsZero$.
}
between
    contractive/unitary homomorphisms $S$ over $\naturalsZero^{d}$ on $\HilbertRaum$
and
    commuting families
        $\{S_{i}\}_{i=1}^{d}$
    of contractions/unitaries on $\HilbertRaum$,
and in the unitary case, $S$ extends uniquely%
\footnote{%
    via
        $
            S(\mathbf{n})
            = S((\max\{-n_{i},0\})_{i=1}^{d})^{\ast}
                S((\max\{n_{i},0\})_{i=1}^{d})
        $
    for $\mathbf{n} = (n_{i})_{i=1}^{d} \in \integers^{d}$.
}
to a unitary representation of $G = \integers^{d}$ on $\HilbertRaum$.
In the case of $(G,M) = (\reals^{d},\realsNonNeg^{d})$, $d\in\naturals$,
there is a natural \onetoone-correspondence%
\footnote{%
    via $T(\mathbf{t}) = \prod_{i=1}^{d}T_{i}(t_{i})$
    for all $\mathbf{t} = (t_{i})_{i=1}^{d} \in \realsNonNeg^{d}$
    and $T_{i}(t) = T(0,0,\ldots,\underset{i}{t},\ldots,0)$
    for $i\in\{1,2,\ldots,d\}$ and all $t\in\realsNonNeg$.
}
between
    continuous contractive/unitary homomorphisms $T$ over $\realsNonNeg^{d}$ on $\HilbertRaum$
and
    commuting families
        $\{T_{i}\}_{i=1}^{d}$
    of $d$-parameter contractive/unitary $\Cnought$-semigroups on $\HilbertRaum$,
and in the unitary case, $T$ extends uniquely%
\footnote{%
    via
        $
            T(\mathbf{t})
            = T((\max\{-t_{i},0\})_{i=1}^{d})^{\ast}
                T((\max\{t_{i},0\})_{i=1}^{d})
        $
    for $\mathbf{t} = (t_{i})_{i=1}^{d} \in \reals^{d}$.
}
to a continuous unitary representation of $G = \reals^{d}$ on $\HilbertRaum$.
Putting these observations together, we can define unitary dilations \`{a} la Sz.-Nagy in a more uniform manner:
Let $G$ be a topological group and $M \subseteq G$ a submonoid.
A continuous contractive homomorphism
    ${T \colon M \to \BoundedOps{\HilbertRaum}}$
is said to have a \highlightTerm{unitary dilation}
if there exist a continuous unitary representation
    ${U \in \Repr{G}{\HilbertRaum^{\prime}}}$
of $G$ on a Hilbert space $\HilbertRaum^{\prime}$,
as well an isometry $r\in\BoundedOps{\HilbertRaum}{\HilbertRaum^{\prime}}$,
such that

    \begin{displaymath}
        T(x) = r^{\ast}\:U(x)\:r
    \end{displaymath}

\continueparagraph
for all $x \in M$.
In this case, we refer to the tuple $(\HilbertRaum^{\prime},U,r)$
as the unitary dilation.
One can readily check that this coincides with the above definitions
in the cases of
    $(G,M) = (\reals^{d},\realsNonNeg^{d})$
    or
    $(\integers^{d},\naturalsZero^{d})$, $d\in\naturals$.

For any topological space $X$ and Hilbert space $\HilbertRaum$,
the topology of uniform $\topWOT$-convergence on compact subsets of $X$
is given as follows:
Let
    $(T^{(\alpha)})_{\alpha\in\Lambda}$
be a net of operator-valued functions
    ${T^{(\alpha)} \colon X \to \BoundedOps{\HilbertRaum}}$,
and let
    ${T \colon X \to \BoundedOps{\HilbertRaum}}$
be a further operator-valued function.
Then
    ${T^{(\alpha)} \underset{\alpha}{\overset{\tinytoplocWOT}{\longrightarrow}} T}$
if and only if

    \begin{displaymath}
        \sup_{x \in K}
            \abs{
                \brkt{
                    (T^{(\alpha)}(x) - T(x))\xi
                }{\eta}
            }
        \underset{\alpha}{\longrightarrow} 0
    \end{displaymath}

\continueparagraph
for all $\xi,\eta\in\HilbertRaum$ and compact $K \subseteq X$.
We refer to this as the $\toplocWOT$-topology.

Let $M$ be a topological monoid.
We shall consider the space
    $(\SpHomCs(M,\HilbertRaum),\toplocWOT)$
of continuous contractive homomorphisms
    ${T \colon M \to \BoundedOps{\HilbertRaum}}$
over $M$ on $\HilbertRaum$,
and endowed with the $\toplocWOT$-topology.
We shall also consider the subspace
    $(\SpHomUs(M,\HilbertRaum),\toplocWOT)$
of continuous unitary homomorphisms
    ${T \colon M \to \BoundedOps{\HilbertRaum}}$
over $M$ on $\HilbertRaum$.

\begin{rem}
\makelabel{rem:pw-top:sig:article-interpolation-raj-dahya}
    Let $d\in\naturals$ and
    consider the space
        $(\OpSpaceC{\HilbertRaum}^{d}_{\textup{comm}} \subseteq \OpSpaceC{\HilbertRaum}^{d}, \topPW)$
        of $d$-tuples of commuting contractions on $\HilbertRaum$
    under the \highlightTerm{weak polynomial} topology.
    This is given by the convergence condition
        ${
            \{S^{(\alpha)}_{i}\}_{i=1}^{d}
            \underset{\alpha}{\overset{\tinytopPW}{\longrightarrow}}
            \{S_{i}\}_{i=1}^{d}
        }$
    if and only if
        ${
            \prod_{i=1}^{d}(S^{(\alpha)}_{i})^{n_{i}}
            \underset{\alpha}{\overset{\tinytopWOT}{\longrightarrow}}
            \prod_{i=1}^{d}S^{n_{i}}_{i}
        }$
    for all $\mathbf{n} = (n_{i})_{i=1}^{d} \in \naturalsZero^{d}$.
    It is straightforward to see that
        $(\SpHomCs(\naturalsZero^{d},\HilbertRaum),\toplocWOT)$
        is topologically isomorphic
        to $(\OpSpaceC{\HilbertRaum}^{d}_{\textup{comm}},\topPW)$
    via the above mentioned correspondence
        (\viz ${S \mapsto \{S(0,0,\ldots,\underset{i}{1},\ldots,0)\}_{i=1}^{d}}$).
\end{rem}

Given the above topological definitions, consider the following problem:

\begin{qstn}
\label{qstn:residuality-of-unitaries:sig:article-interpolation-raj-dahya}
    Let $M$ be a topological monoid and $\HilbertRaum$ a Hilbert space.
    Is $\SpHomUs(M,\HilbertRaum)$ residual in $(\SpHomCs(M,\HilbertRaum),\toplocWOT)$?
\end{qstn}

Suppose now that $\HilbertRaum$ is separable and infinite-dimensional
and that $M$ is locally compact Polish
(\exempli
    $
        M \in \{
            \realsNonNeg^{d}, \naturalsZero^{d}
            \mid  d\in\naturals
        \}
    $%
).
Let $C$ be the space of $\topWOT$-continuous contractive-valued functions
of the form
    ${T \colon M \to \OpSpaceC{\HilbertRaum}}$.
Further let
    $X \coloneqq \SpHomCs(M,\HilbertRaum)$
    and
    $A \coloneqq \SpHomUs(M,\HilbertRaum)$.
It is known that $(C,\toplocWOT)$ and $(A,\toplocWOT)$
are Polish spaces
(see \cite[Proposition~1.16 and Proposition~1.18]{Dahya2022weakproblem}).
In particular, the subspace $X$ of $(C,\toplocWOT)$
is a second-countable metrisable space,
and by Alexandroff's lemma
(see
    \cite[Theorem~3.11]{Kechris1995BookDST},
    \cite[Lemmata~3.33--34]{Aliprantis2006BookAn}%
),
$A$ is a $G_{\delta}$-subset in $(X,\toplocWOT)$.
Under these assumptions on $M$ and $\HilbertRaum$,
the above problem,
\viz whether $A$ is residual in $(X,\toplocWOT)$,
is thus equivalent to
whether $A$ is dense inside $(X,\toplocWOT)$.

Now, we often consider topological monoids $M$ as submonoids of topological groups, $G$.
In case unitary semigroups extend uniquely to unitary representations
(this holds \exempli for
    $
        (G,M) \in \{
            (\reals^{d},\realsNonNeg^{d}),
            (\integers^{d},\naturalsZero^{d})
            \mid
            d\in\naturals
        \}
    $%
),
recent work in \cite{Dahya2024approximation}
has shown that the density question
is closely related to the existence of unitary dilations
(this will be made clear in \Cref{prop:dilatability-equiv-to-weak-unitary-approximability:sig:article-interpolation-raj-dahya} below).
This leads to a related problem, which is interesting in and of itself:

\begin{qstn}
\label{qstn:existence-of-dilations:sig:article-interpolation-raj-dahya}
    Let $G$ be a topological group, $M \subseteq G$ a submonoid,
    and $\HilbertRaum$ a Hilbert space.
    Does every $T \in \SpHomCs(M,\HilbertRaum)$ have a unitary dilation?
\end{qstn}



\subsection[Summary of main results]{Summary of main results}
\label{sec:introduction:results:sig:article-interpolation-raj-dahya}

\firstparagraph
This paper addresses the above two problems
for the concrete cases of
    $(G,M) = (\integers^{d},\naturalsZero^{d})$
    and $(\reals^{d},\realsNonNeg^{d})$
    for $d\in\naturals$
and (separable) infinite-dimensional $\HilbertRaum$.
In discrete-time, Parrott's construction
(see
    \cite[\S{}3]{Parrott1970counterExamplesDilation},
    \cite[\S{}I.6.3]{Nagy1970}%
)
showed that there always exists non-power dilatable families of $d \geq 3$ contractions,
provided the dimension of the Hilbert spaces was suitably large.
For a long time no counter-examples to simultaneously unitarily dilatable
families seemed to be be known in the continuous setting.
In light of
    And\^{o}'s power-dilation theorem for pairs of contractions
    (see \cite{Ando1963pairContractions})
    and S\l{}oci\'{n}ski's theorem on
    simultaneous unitary dilations of pairs of commuting contractive $\Cnought$-semigroups
    (see
        \cite{Slocinski1974},
        \cite[Theorem~2]{Slocinski1982},
        and
        \cite[Theorem~2.3]{Ptak1985}%
    ),
the bound $d=3$ is optimal in the discrete setting
and could also be postulated to be optimal in the continuous setting.
We shall show that this is indeed the case.
To achieve this, our first goal is to reduce
    non-dilatability in the continuous-time setting
    to
    non-dilatability in the discrete-time setting.
This in turn is achieved via interpolation.

Let $\Torus$ denote the unit circle $\{z\in\complex \mid \abs{z} = 1\}$
endowed with the unique Haar-probability measure.
For any non-empty index set $\mathcal{I}$,
the product space
    $\prod_{i\in\mathcal{I}}\Torus$
is equipped with the product (probability) measure.%
\footnote{%
    For the construction of the product of arbitrarily many probability spaces,
    see \exempli \cite[Example~14.37]{Klenke2008probTheory}.
}
We consider the Hilbert space tensor product
$L^{2}(\prod_{i\in\mathcal{I}}\Torus) \otimes \HilbertRaum$,%
\footnote{%
    For Hilbert spaces $H_{1},H_{2}$
    the Hilbert space tensor product
        $H_{1} \otimes H_{2}$
    can be viewed as being isomorphic to
        $\ell^{2}(B_{1} \times B_{2})$,
    where
        $B_{i} \subseteq H_{i}$
        is an orthonormal basis (ONB)
        for each $i$.
    For
        $\xi_{1},\eta_{1} \in H_{1}$
    and
        $\xi_{2},\eta_{2} \in H_{2}$
    we have
        $
            \brkt{
                \xi_{1} \otimes \xi_{2}
            }{
                \eta_{1} \otimes \eta_{2}
            }_{H_{1} \otimes H_{2}}
            = \brkt{\xi_{1}}{\eta_{1}}_{H_{1}}
                \brkt{\xi_{2}}{\eta_{2}}_{H_{2}}
        $,
    where \exempli
        $
            \xi_{1} \otimes \xi_{2}
            = (
                \brkt{\xi_{1}}{e_{1}}_{H_{1}}
                \brkt{\xi_{2}}{e_{2}}_{H_{2}}
            )_{(e_{1},e_{2}) \in B_{1} \times B_{2}}
        $.
    For the algebraic and geometric definitions
    as well as tensor products of bounded operators,
    we refer the reader to
        \cite[\S{}2.6]{KadisonRingrose1983volI},
        \cite[\S{}6.3]{Murphy1990}
        \cite[E3.2.19--21]{Pedersen1989analysisBook}.
}
viewing $L^{2}(\prod_{i\in\mathcal{I}}\Torus)$
as an \usesinglequotes{auxiliary space}.
Our first main result extends
a recent construction due to Bhat and Skeide
\cite{BhatSkeide2015PureCounterexamples}:

\begin{schattierteboxdunn}[backgroundcolor=leer,nobreak=true]
\begin{lemm}[Generalised Bhat--Skeide Interpolation]
\makelabel{lemm:bhat-skeide:multi:sig:article-interpolation-raj-dahya}
    Let $\mathcal{I}$ be a non-empty index set and
        $\{S_{i}\}_{i \in \mathcal{I}}$
    be a commuting family of bounded operators
    (\respectively contractions)
    on a Hilbert space $\HilbertRaum$.
    There exists a commuting family
        $\{T_{i}\}_{i \in \mathcal{I}}$
    of $\Cnought$-semigroups
    (\respectively contractive $\Cnought$-semigroups)
    on $L^{2}(\prod_{i\in\mathcal{I}}\Torus) \otimes \HilbertRaum$,
    such that

        \begin{restoremargins}
        \begin{equation}
        \label{eq:interpolation:sig:article-interpolation-raj-dahya}
            \prod_{i\in\mathcal{I}}T_{i}(n_{i})
            = \prod_{i\in\mathcal{I}}
                \onematrix \otimes S_{i}^{n_{i}}
        \end{equation}
        \end{restoremargins}

    \continueparagraph
    for all $\mathbf{n} = (n_{i})_{i\in\mathcal{I}}\in\prod_{i\in\mathcal{I}}\naturalsZero$
    with $\supp(\mathbf{n}) = \{i \in \mathcal{I} \mid n_{i} \neq 0\}$ finite.
\end{lemm}
\end{schattierteboxdunn}

As an excursus we apply the Bhat--Skeide interpolation to the embedding problem
in \S{}\ref{sec:interpolation:embed:sig:article-interpolation-raj-dahya}.
And we discuss in \S{}\ref{sec:interpolation:application:sig:article-interpolation-raj-dahya}
a time-discretised version and show that the interpolation and its discretisations
enjoy a certain natural property (see \Cref{prop:multilinear-interpolation:sig:article-interpolation-raj-dahya}).
Then in \S{}\ref{sec:interpolation:counterexamples:sig:article-interpolation-raj-dahya}
we use the interpolation to answer
    \Cref{qstn:existence-of-dilations:sig:article-interpolation-raj-dahya}
negatively for $d$-parameter semigroups with $d \geq 3$:

\begin{schattierteboxdunn}[backgroundcolor=leer,nobreak=true]
\begin{thm}[Non-dilatable families of semigroups]
\makelabel{thm:existence-of-non-dilatable:sig:article-interpolation-raj-dahya}
    Let $\mathcal{I}$ be a non-empty index set
    and $\HilbertRaum$ be an infinite-dimensional Hilbert space.
    If $\card{\mathcal{I}} \geq 3$,
    there exists a commuting family
        $\{T_{i}\}_{i\in\mathcal{I}}$
    of contractive $\Cnought$-semigroups on $\HilbertRaum$,
    which admit no simultaneous unitary dilation.
    Furthermore, the $T_{i}$ can be chosen to have bounded generators.
\end{thm}
\end{schattierteboxdunn}

\begin{rem}
\label{rem:shalit-skeide-result:sig:article-interpolation-raj-dahya}
    In the non-classical setting of CP-semigroups (on unital $C^{\ast}$-algebras),
    Shalit and Skeide \cite[Corollary~18.10]{ShalitSkeide2022multiparam}
    have recently constructed
    contractive CP- \respectively Markov semigroups over $\realsNonNeg^{3}$
    which admit no \usesinglequotes{strong} \respectively \usesinglequotes{weak} dilations
    (see
        \cite[\S{}8.1]{Shalit2021DilationBook},
        \cite[Definition~2.3]{ShalitSkeide2022multiparam}%
    ).
    We note here two things:
    These counter-examples and our ones in \Cref{thm:existence-of-non-dilatable:sig:article-interpolation-raj-dahya}
    do not appear to be derivable from one other.
    Moreover, the approach in \cite{ShalitSkeide2022multiparam}
    used to construct multi-parameter counter-examples
    is not based on the interpolation of Bhat and Skeide \cite{BhatSkeide2015PureCounterexamples},
    rather other interpolation techniques are developed.
\end{rem}

In \S{}\ref{sec:residuality:sig:article-interpolation-raj-dahya}
we proceed to develop a \zeroone-dichotomy for the general setting of homomorphisms.
Applying this to our counter-examples,
we obtain the following answer to \Cref{qstn:residuality-of-unitaries:sig:article-interpolation-raj-dahya}:

\begin{schattierteboxdunn}[backgroundcolor=leer,nobreak=true]
\begin{cor}
\makelabel{cor:zero-one-for-unitaries:sig:article-interpolation-raj-dahya}
    Let $d \in \naturals$ and $\HilbertRaum$ be a separable infinite-dimensional Hilbert space.
    Consider the space
        $X \coloneqq \SpHomCs(\realsNonNeg^{d},\HilbertRaum)$
        of $d$-parameter contractive $\Cnought$-semigroups on $\HilbertRaum$
        under the $\toplocWOT$-topology
    as well as the subspaces
        $A \coloneqq \SpHomUs(\realsNonNeg^{d},\HilbertRaum)$
        of $d$-parameter unitary $\Cnought$-semigroups on $\HilbertRaum$
    and $D$ of $d$-parameter semigroups admitting a unitary dilation.
    Then $D = \quer{A}$ and

    \begin{kompaktenum}{\bfseries (a)}[\rtab]
    \item\punktlabel{1}
        if $d \leq 2$,
            then $A$ is a dense $G_{\delta}$-subset in $(X,\toplocWOT)$.
    \item\punktlabel{2}
        if $d \geq 3$, then $\quer{A}$ is meagre in $(X,\toplocWOT)$.
    \end{kompaktenum}

    Consider the space
        $X \coloneqq \OpSpaceC{\HilbertRaum}^{d}_{\textup{comm}} \subseteq \OpSpaceC{\HilbertRaum}^{d}$
        of $d$-tuples of commuting contractions on $\HilbertRaum$
        under the $\topPW$-topology
    as well as the subspaces
        $A \coloneqq \OpSpaceU{\HilbertRaum}^{d}_{\textup{comm}} \subseteq X$,
        of $d$-tuples of commuting unitaries on $\HilbertRaum$
    and $D$ of $d$-tuples admitting a power-dilation.
    Then $D = \quer{A}$ and

    \begin{kompaktenum}{\bfseries (a')}[\rtab]
    \item
        if $d \leq 2$, then $A$ is a dense $G_{\delta}$-subset in $(X,\topPW)$.
    \item
        if $d \geq 3$, then $\quer{A}$ is meagre in $(X,\topPW)$.
    \end{kompaktenum}

    \nvraum{1}

\end{cor}
\end{schattierteboxdunn}

In the discrete setting,
we shall apply similar reasoning to obtain
an analogous \zeroone-dichotomy
for $d$-tuples of commuting contractions
satisfying the von Neumann inequality
(see \Cref{rem:vN-inequality:sig:article-interpolation-raj-dahya}).



\subsection[Notation]{Notation}
\label{sec:introduction:notation:sig:article-interpolation-raj-dahya}

\firstparagraph
Throughout we fix the following notation and conventions:

\begin{itemize}
\item
    We write
        $\naturalsPos = \{1,2,\ldots\}$,
        $\naturalsZero = \{0,1,2,\ldots\}$,
        $\realsNonNeg = \{r\in\reals \mid r\geq 0\}$,
        and
        $\Torus = \{z\in\complex \mid \abs{z} = 1\}$.
    To distinguish from indices $i$ we use $\iunit$ for the imaginary unit $\sqrt{-1}$.
\item
    We write elements of product spaces in bold
    and denote their components in light face fonts with appropriate indices,
    \exempli the $i$\textsuperscript{th} components of
        $\mathbf{t} \in \realsNonNeg^{d}$
        and
        $\mathbf{z} \in \prod_{i \in \mathcal{I}}\Torus$
        are denoted
            $t_{i}$ and $z_{i}$
        respectively.
\item
    All Hilbert spaces shall be takten to have dimension at least $1$.
    This is to avoid the issue that on $0$-dimensional spaces both
    the identity and zero operators are equal,
    which unnecessarily complicates matters.
\item
    For Hilbert spaces $\HilbertRaum,\HilbertRaum^{\prime}$
    the sets
        $\BoundedOps{\HilbertRaum}$
        and
        $\BoundedOps{\HilbertRaum}{\HilbertRaum^{\prime}}$
    denote the bounded linear operators on $\HilbertRaum$
    \respectively the bounded linear operators between $\HilbertRaum$ and $\HilbertRaum^{\prime}$.
    Further,
        $\OpSpaceC{\HilbertRaum}$
        and
        $\OpSpaceU{\HilbertRaum}$
    denote the set of contractions and unitaries on $\HilbertRaum$ respectively.
    And for $d\in\naturals$,
        $\OpSpaceC{\HilbertRaum}^{d}_{\textup{comm}}$
        and
        $\OpSpaceU{\HilbertRaum}^{d}_{\textup{comm}}$
    are denote the set of commuting $d$-tuples of contractions and unitaries on $\HilbertRaum$ respectively.
\item
    For Hilbert spaces $H_{1},H_{2},\ldots,H_{n}$, $n\in\naturals$,
    the tensor product ${H_{1} \otimes H_{2} \otimes \ldots \ldots \otimes H_{n}}$
    shall always denote the Hilbert space tensor product,
    \idest the completion of the algebraic tensor product
    under the induced norm.
\item
    For a Hilbert space $\HilbertRaum$ and topological group $G$,
    the set $\Repr{G}{\HilbertRaum}$
    denotes the set of $\topSOT$-continuous unitary representations
        ${U \colon G \to \BoundedOps{\HilbertRaum}}$
    of $G$ on $\HilbertRaum$.
\item
    For a Hilbert space $\HilbertRaum$,
    we denote the identity operator as $\onematrix_{\HilbertRaum}$.
    If it is clear from the context, we simply write $\onematrix$.
    For example, in the context of $\HilbertRaum = H_{0} \otimes H_{1}$
    for Hilbert spaces $H_{0}, H_{1}$,
    given
        $R\in\BoundedOps{H_{0}}$,
        $S\in\BoundedOps{H_{1}}$,
    the expression
        $(R \otimes \onematrix)(\onematrix \otimes S) + \onematrix$
    is to be interpreted as
        $(R \otimes \onematrix_{H_{1}})(\onematrix_{H_{0}} \otimes S) + \onematrix_{\HilbertRaum}$.
\item
    $\toplocWOT$ denotes the topology of uniform $\topWOT$-convergence
    on compact subsets (see \S{}\ref{sec:introduction:definitions:sig:article-interpolation-raj-dahya}).
\item
    $\topPW$ denotes the weak polynomial topology on the space
    of tuples of contractions (see \Cref{rem:pw-top:sig:article-interpolation-raj-dahya}).
\item
    For $n\in\naturals$ and $i,j\in\{1,2,\ldots,n\}$,
        $E_{ij} \in M_{n}(\complex)$
    denotes the matrix containing a $1$ in entry $(i,j)$ and $0$'s elsewhere.
\end{itemize}

In any context, given a choice $(G,M)$ of a topological group and submonoid $M \subseteq G$,
we shall repeatedly refer to the spaces
    $A^{\ast} \subseteq A \subseteq X \subseteq C$
    and
    $D \subseteq X$,
where

\begin{itemize}
\item $C \coloneqq$ the set of $\topWOT$-continuous contraction-valued functions
    ${T \colon M \to \OpSpaceC{\HilbertRaum}}$;
\item $X \coloneqq \SpHomCs(M,\HilbertRaum)$,
    the set of $\topSOT$-continuous contractive homomorphisms;
\item $A \coloneqq \SpHomUs(M,\HilbertRaum)$,
    the set of $\topSOT$-continuous unitary homomorphisms;
\item $A^{\ast} \coloneqq \{ U\restr{M} \mid U \in\Repr{G}{\HilbertRaum}\}$;
\item $D \coloneqq \{T \in X \mid T~\text{has a unitary dilation}\}$.
\end{itemize}

These spaces shall generally be equipped with the $\toplocWOT$-topology
and
    $\quer{A}, \quer{A^{\ast}}$
shall always denote the closures of $A, A^{\ast}$ within $(X,\toplocWOT)$.
Note further that $C, X, A$ can be defined for topological monoids $M$
which are not submonoids of topological groups.




\section[Interpolation and non-dilatability]{Interpolation and non-dilatability}
\label{sec:interpolation:sig:article-interpolation-raj-dahya}

\firstparagraph
We first present the interpolation result of Bhat and Skeide,
which extends a single contraction to a $\Cnought$-semigroup.
We then generalise this to extend arbitrary commuting families of contractions
to commuting families of contractive $\Cnought$-semigroups.
In the following we use the notation
$\fractional{t}$ and $\floor{t}$
to denote the fractional and integer part of a number $t\in\reals$,
\idest
    $\floor{t} = \max\{n\in\integers \mid n \leq t\}$
    and
    $\fractional{t} = t - \floor{t} \in [0,\:1)$.


\subsection[Bhat--Skeide interpolation for contractions]{Bhat--Skeide interpolation for contractions}
\label{sec:interpolation:single:sig:article-interpolation-raj-dahya}

\firstparagraph
In \cite[Lemma~2.4~(1)]{BhatSkeide2015PureCounterexamples}
Bhat and Skeide developed an interpolation result for isometries
over Hilbert $C^{\ast}$-modules (a generalisation of Hilbert spaces).
Their construction, however, also works for contractions over Hilbert spaces.
We slightly modify the spaces and notation,
but the constructions remain essentially the same.
To start off, for $t\in\reals$ we let
    ${\phi_{t} \colon \Torus \to \Torus}$
    and
    ${p_{t} \colon \Torus \to \reals}$
be defined via

    \begin{shorteqnarray}
        \phi_{t}(z) &\coloneqq &e^{-\iunit 2\pi t}z = e^{\iunit 2\pi (s-t)}\\
        p_{t}(z) &\coloneqq &\einser_{[0,\:1-\fractional{t})}(s)\\
    \end{shorteqnarray}

\continueparagraph
for $z = e^{\iunit 2\pi s} \in \Torus$, $s\in[0,\:1)$.
Since $\phi_{t}$ is a measure-preserving bijection,
the Koopman-operator $U(t)\in\BoundedOps{L^{2}(\Torus)}$
defined by
    $U(t)f \coloneqq f \circ \phi_{t}$
for $f \in L^{2}(\Torus)$ is unitary.
It is straightforward to see that
    ${\reals\ni t \mapsto f \circ \phi_{t} \in L^{2}(\Torus)}$
is continuous for $f \in \Cts{\Torus}$,
and that
    $\phi_{s + t} = \phi_{s} \circ \phi_{t}$
for $s,t\in\reals$.
Thus $U \in \Repr{\reals}{L^{2}(\Torus)}$.
Since $p_{t}$ for each $t\in\reals$ is a $\{0,1\}$-valued measurable function,
the multiplication operator
    $P(t) \coloneqq M_{p_{t}} \in \BoundedOps{L^{2}(\Torus)}$
is a projection.
For $t_{0}\in\reals$ and ${(t_{0},\:\infty) \ni t \longrightarrow t_{0}}$
we have $\fractional{t} \longrightarrow \fractional{t_{0}}$
and hence
    $
        P(t) = M_{
            \einser_{
                \{
                    e^{\iunit 2\pi s}
                    \mid
                    s \in [0,\:1-\fractional{t})
                \}
            }
        }
        \overset{\tinytopSOT}{\longrightarrow}
            M_{
                \einser_{
                    \{
                        e^{\iunit 2\pi s}
                        \mid
                        s \in [0,\:1-\fractional{t_{0}})
                    \}
                }
            }
            = P(t_{0})
    $.
Thus $P$ is right-$\topSOT$-continuous.

Note also that $\phi$ and $p$
and hence $U$ and $P$ are periodic,
\viz
        $U(t) = U(\fractional{t})$
    and $P(t) = P(\fractional{t})$
for all $t\in\reals$.
In particular $P(n) = P(0) = M_{\einser_{\Torus}} = \onematrix$ for $n\in\integers$.
The operator-valued functions $U$ and $P$ also enjoy the following property:

\begin{prop}[Bhat--Skeide commutation relation]
\makelabel{prop:bscr:sig:article-interpolation-raj-dahya}
    It holds that

        \begin{restoremargins}
        \begin{equation}
        \label{eq:bscr:sig:article-interpolation-raj-dahya}
            P(s)U(t)
            = U(t)
            \cdot \begin{cases}
                    \onematrix - (P(t) - P(s+t))
                        &\colon &\fractional{s} + \fractional{t} < 1\\
                    P(s+t) - P(t)
                        &\colon &\fractional{s} + \fractional{t} \geq 1\\
                \end{cases}
        \tag{BSCR}
        \end{equation}
        \end{restoremargins}

    \continueparagraph
    for $s,t\in\reals$.
    We shall refer to this as the \highlightTerm{Bhat--Skeide commutation relation}
    (or simply: \highlightTerm{BSCR}).
\end{prop}

    \begin{proof}
        If $s\in\integers$, then due to periodicity,
            $P(s) = P(0) = \onematrix$
            and
            $P(s + t) = P(t)$.
        Since
            $\fractional{s} + \fractional{t} = \fractional{t} < 1$,
        the right-hand expression is equal to
            $
                U(t)\cdot(\onematrix - (P(t) - P(s+t)))
                = U(t)\cdot(\onematrix - (P(t) - P(t)))
                = U(t)
                = P(s)U(t)
            $.
        If $t\in\integers$, then due to periodicity,
            $P(t) = P(0) = \onematrix$,
            $U(t) = U(0) = \onematrix$,
            and
            $P(s + t) = P(s)$.
        Since
            $\fractional{s} + \fractional{t} = \fractional{s} < 1$,
        the right-hand expression is equal to
            $
                U(t)\cdot(\onematrix - (P(t) - P(s+t)))
                = \onematrix\cdot(\onematrix - (\onematrix - P(s)))
                = P(s)
                = P(s)U(t)
            $.
        If $s,t \in \reals\without\integers$,
        then the proof of \eqcref{eq:bscr:sig:article-interpolation-raj-dahya}
        is essentially the same as in \cite[Lemma~2.4~(1)]{BhatSkeide2015PureCounterexamples}.
        (The visualisation in \Cref{fig:bhat-skeide:sig:article-interpolation-raj-dahya}
        of the effect of $U(t)^{\ast}P(s)U(t)$
        on elements of $L^{2}(\Torus)$
        also depicts this computation.)
    \end{proof}


\begin{figure}[!ht]
\begin{mdframed}[%
    backgroundcolor=background_light_grey,%
    linewidth=0pt,%
]
    \centering
    \begin{subfigure}[b]{0.3\textwidth}
        \centering
        \begin{tikzpicture}[node distance=1cm,thick]
            \pgfmathsetmacro\hunit{0.5}
            \pgfmathsetmacro\vunit{0.5}
            \pgfmathsetmacro\hscale{6.283185307179586}
            \pgfmathsetmacro\vscale{1}

            \draw [draw=black, line width=0.5pt, fill=none]
                ({-0.2 * \hscale * \hunit}, {-2.25 * \vscale * \vunit})
                -- ({1.2 * \hscale * \hunit}, {-2.25 * \vscale * \vunit})
                -- ({1.2 * \hscale * \hunit}, {4.0 * \vscale * \vunit})
                -- ({-0.2 * \hscale * \hunit}, {4.0 * \vscale * \vunit})
                -- cycle;

            \draw [draw=none, fill=none]
                ({0 * \hscale * \hunit}, {4.0 * \vscale * \vunit})
                node[label={[below right, xshift=0pt, yshift=-2pt] \scriptsize $f(e^{\iunit \theta})$ vs. $\theta$}]{};

            \draw[->, draw=black, line width=0.5pt, fill=black]
                ({-0.1 * \hscale * \hunit}, {0 * \vscale * \vunit})
                -- ({1.1 * \hscale * \hunit}, {0 * \vscale * \vunit})
                node[label={[below, xshift=0pt, yshift=-2pt] \scriptsize $\theta$}]{};

            \draw [draw=black, line width=1.5pt, fill=none]
                ({0.0 * \hscale * \hunit}, {-0.75 * \vunit * \vscale})
                node[label={[below, xshift=0pt, yshift=-2pt]\scriptsize \rotatebox[origin=c]{-90}{$0$}}]{}
                -- ({0.0 * \hscale * \hunit}, {-0.25 * \vscale * \vunit});
            \draw [draw=black, line width=1.5pt, fill=none]
                ({1.0 * \hscale * \hunit}, {-0.75 * \vunit * \vscale})
                node[label={[below, xshift=0pt, yshift=-2pt]\scriptsize \rotatebox[origin=c]{-90}{$2\pi$}}]{}
                -- ({1.0 * \hscale * \hunit}, {-0.25 * \vscale * \vunit});

            \draw[draw=none, line width=0pt, fill=black] ({0.0 * \hscale * \hunit}, {2.1809444573256793 * \vscale * \vunit}) circle (2pt);

            \draw[draw=black, line width=1pt, fill=none]
                ({0.0 * \hscale * \hunit}, {2.1809444573256793 * \vscale * \vunit})
                -- ({0.04 * \hscale * \hunit}, {2.7593654556135236 * \vscale * \vunit})
                -- ({0.08 * \hscale * \hunit}, {2.0930940465535786 * \vscale * \vunit})
                -- ({0.12 * \hscale * \hunit}, {2.1536968350648995 * \vscale * \vunit})
                -- ({0.16 * \hscale * \hunit}, {2.0310870311347475 * \vscale * \vunit})
                -- ({0.2 * \hscale * \hunit}, {1.8776400281344068 * \vscale * \vunit})
                -- ({0.24 * \hscale * \hunit}, {1.7657712943879427 * \vscale * \vunit})
                -- ({0.28 * \hscale * \hunit}, {2.2568778364637416 * \vscale * \vunit})
                -- ({0.32 * \hscale * \hunit}, {2.012366485978751 * \vscale * \vunit})
                -- ({0.36 * \hscale * \hunit}, {1.5681222791371292 * \vscale * \vunit})
                -- ({0.4 * \hscale * \hunit}, {1.96944178517839 * \vscale * \vunit})
                -- ({0.44 * \hscale * \hunit}, {1.5510712933545452 * \vscale * \vunit})
                -- ({0.48 * \hscale * \hunit}, {2.0885407526360256 * \vscale * \vunit})
                -- ({0.52 * \hscale * \hunit}, {2.2113303695203994 * \vscale * \vunit})
                -- ({0.56 * \hscale * \hunit}, {1.9011959811801344 * \vscale * \vunit})
                -- ({0.6 * \hscale * \hunit}, {0.9636135216665631 * \vscale * \vunit})
                -- ({0.64 * \hscale * \hunit}, {2.397522383621427 * \vscale * \vunit})
                -- ({0.68 * \hscale * \hunit}, {1.1233020661612172 * \vscale * \vunit})
                -- ({0.72 * \hscale * \hunit}, {1.2940431182244299 * \vscale * \vunit})
                -- ({0.76 * \hscale * \hunit}, {2.144293658942074 * \vscale * \vunit})
                -- ({0.8 * \hscale * \hunit}, {2.753043391732871 * \vscale * \vunit})
                -- ({0.84 * \hscale * \hunit}, {2.454921971614578 * \vscale * \vunit})
                -- ({0.88 * \hscale * \hunit}, {2.175109235559883 * \vscale * \vunit})
                -- ({0.92 * \hscale * \hunit}, {2.2572062167331226 * \vscale * \vunit})
                -- ({0.96 * \hscale * \hunit}, {2.1613703578538948 * \vscale * \vunit})
                -- ({1.0 * \hscale * \hunit}, {2.8367977269110867 * \vscale * \vunit});

            \draw[draw=black, line width=1pt, fill=none] ({1.0 * \hscale * \hunit}, {2.8367977269110867 * \vscale * \vunit}) circle (2pt);
        \end{tikzpicture}
        \caption{Arbitrary $f \in L^{2}(\Torus)$}
        \label{fig:bhat-skeide:func:sig:article-interpolation-raj-dahya}
    \end{subfigure}

    \null

    \begin{subfigure}[b]{0.3\textwidth}
        \centering
        \begin{tikzpicture}[node distance=1cm,thick]
            \pgfmathsetmacro\hunit{0.5}
            \pgfmathsetmacro\vunit{0.5}
            \pgfmathsetmacro\hscale{6.283185307179586}
            \pgfmathsetmacro\vscale{1}

            \draw [draw=black, line width=0.5pt, fill=none]
                ({-0.2 * \hscale * \hunit}, {-2.25 * \vscale * \vunit})
                -- ({1.2 * \hscale * \hunit}, {-2.25 * \vscale * \vunit})
                -- ({1.2 * \hscale * \hunit}, {4.0 * \vscale * \vunit})
                -- ({-0.2 * \hscale * \hunit}, {4.0 * \vscale * \vunit})
                -- cycle;

            \draw [draw=none, fill=none]
                ({0 * \hscale * \hunit}, {4.0 * \vscale * \vunit})
                node[label={[below right, xshift=0pt, yshift=-2pt] \scriptsize $U(t)^{\ast}P(s)U(t)f\:(e^{\iunit \theta})$ vs. $\theta$}]{};

            \draw[->, draw=black, line width=0.5pt, fill=black]
                ({-0.1 * \hscale * \hunit}, {0 * \vscale * \vunit})
                -- ({1.1 * \hscale * \hunit}, {0 * \vscale * \vunit})
                node[label={[below, xshift=0pt, yshift=-2pt] \scriptsize $\theta$}]{};

            \draw [draw=black, line width=1.5pt, fill=none]
                ({0.4 * \hscale * \hunit}, {-0.75 * \vunit * \vscale})
                node[label={[below, xshift=0pt, yshift=-2pt]\scriptsize \rotatebox[origin=c]{-90}{$2\pi t_{1}$}}]{}
                -- ({0.4 * \hscale * \hunit}, {-0.25 * \vscale * \vunit});
            \draw [draw=black, line width=1.5pt, fill=none]
                ({0.08000000000000007 * \hscale * \hunit}, {-0.75 * \vunit * \vscale})
                node[label={[below, xshift=0pt, yshift=-2pt]\scriptsize \rotatebox[origin=c]{-90}{$2\pi t_{2}$}}]{}
                -- ({0.08000000000000007 * \hscale * \hunit}, {-0.25 * \vscale * \vunit});
            \draw [draw=black, line width=1.5pt, fill=none]
                ({0.0 * \hscale * \hunit}, {-0.75 * \vunit * \vscale})
                node[label={[below, xshift=0pt, yshift=-2pt]\scriptsize \rotatebox[origin=c]{-90}{$0$}}]{}
                -- ({0.0 * \hscale * \hunit}, {-0.25 * \vscale * \vunit});
            \draw [draw=black, line width=1.5pt, fill=none]
                ({1.0 * \hscale * \hunit}, {-0.75 * \vunit * \vscale})
                node[label={[below, xshift=0pt, yshift=-2pt]\scriptsize \rotatebox[origin=c]{-90}{$2\pi$}}]{}
                -- ({1.0 * \hscale * \hunit}, {-0.25 * \vscale * \vunit});

            \draw[draw=none, line width=0pt, fill=black] ({0.0 * \hscale * \hunit}, {2.1809444573256793 * \vscale * \vunit}) circle (2pt);

            \draw[draw=black, line width=1pt, fill=none]
                ({0.0 * \hscale * \hunit}, {2.1809444573256793 * \vscale * \vunit})
                -- ({0.040000000000000036 * \hscale * \hunit}, {2.7593654556135236 * \vscale * \vunit})
                -- ({0.07999999999999996 * \hscale * \hunit}, {2.0930940465535786 * \vscale * \vunit});

            \draw[draw=black, line width=1pt, fill=none] ({0.07999999999999996 * \hscale * \hunit}, {2.0930940465535786 * \vscale * \vunit}) circle (2pt);

            \draw[draw=black, line width=1pt, dotted]
                ({0.07999999999999996 * \hscale * \hunit}, {2.0930940465535786 * \vscale * \vunit})
                -- ({0.07999999999999996 * \hscale * \hunit}, {0.0 * \vscale * \vunit});

            \draw[draw=none, line width=0pt, fill=black] ({0.07999999999999996 * \hscale * \hunit}, {0.0 * \vscale * \vunit}) circle (2pt);

            \draw[draw=black, line width=1pt, fill=none]
                ({0.07999999999999996 * \hscale * \hunit}, {0.0 * \vscale * \vunit})
                -- ({0.12 * \hscale * \hunit}, {0.0 * \vscale * \vunit})
                -- ({0.16000000000000003 * \hscale * \hunit}, {0.0 * \vscale * \vunit})
                -- ({0.20000000000000007 * \hscale * \hunit}, {0.0 * \vscale * \vunit})
                -- ({0.24 * \hscale * \hunit}, {0.0 * \vscale * \vunit})
                -- ({0.28 * \hscale * \hunit}, {0.0 * \vscale * \vunit})
                -- ({0.31999999999999995 * \hscale * \hunit}, {0.0 * \vscale * \vunit})
                -- ({0.36 * \hscale * \hunit}, {0.0 * \vscale * \vunit})
                -- ({0.4 * \hscale * \hunit}, {0.0 * \vscale * \vunit});

            \draw[draw=black, line width=1pt, fill=none] ({0.4 * \hscale * \hunit}, {0.0 * \vscale * \vunit}) circle (2pt);

            \draw[draw=black, line width=1pt, dotted]
                ({0.4 * \hscale * \hunit}, {0.0 * \vscale * \vunit})
                -- ({0.4 * \hscale * \hunit}, {1.96944178517839 * \vscale * \vunit});

            \draw[draw=none, line width=0pt, fill=black] ({0.4 * \hscale * \hunit}, {1.96944178517839 * \vscale * \vunit}) circle (2pt);

            \draw[draw=black, line width=1pt, fill=none]
                ({0.4 * \hscale * \hunit}, {1.96944178517839 * \vscale * \vunit})
                -- ({0.44 * \hscale * \hunit}, {1.5510712933545452 * \vscale * \vunit})
                -- ({0.48 * \hscale * \hunit}, {2.0885407526360256 * \vscale * \vunit})
                -- ({0.52 * \hscale * \hunit}, {2.2113303695203994 * \vscale * \vunit})
                -- ({0.56 * \hscale * \hunit}, {1.9011959811801344 * \vscale * \vunit})
                -- ({0.6 * \hscale * \hunit}, {0.9636135216665631 * \vscale * \vunit})
                -- ({0.64 * \hscale * \hunit}, {2.397522383621427 * \vscale * \vunit})
                -- ({0.68 * \hscale * \hunit}, {1.1233020661612172 * \vscale * \vunit})
                -- ({0.72 * \hscale * \hunit}, {1.2940431182244299 * \vscale * \vunit})
                -- ({0.76 * \hscale * \hunit}, {2.144293658942074 * \vscale * \vunit})
                -- ({0.8 * \hscale * \hunit}, {2.753043391732871 * \vscale * \vunit})
                -- ({0.84 * \hscale * \hunit}, {2.454921971614578 * \vscale * \vunit})
                -- ({0.88 * \hscale * \hunit}, {2.175109235559883 * \vscale * \vunit})
                -- ({0.92 * \hscale * \hunit}, {2.2572062167331226 * \vscale * \vunit})
                -- ({0.96 * \hscale * \hunit}, {2.1613703578538948 * \vscale * \vunit})
                -- ({1.0 * \hscale * \hunit}, {2.8367977269110867 * \vscale * \vunit});

            \draw[draw=black, line width=1pt, fill=none] ({1.0 * \hscale * \hunit}, {2.8367977269110867 * \vscale * \vunit}) circle (2pt);
        \end{tikzpicture}
        \caption{$\fractional{s} + \fractional{t} < 1$}
        \label{fig:bhat-skeide:less:sig:article-interpolation-raj-dahya}
    \end{subfigure}
    \begin{subfigure}[b]{0.3\textwidth}
        \centering
        \begin{tikzpicture}[node distance=1cm,thick]
            \pgfmathsetmacro\hunit{0.5}
            \pgfmathsetmacro\vunit{0.5}
            \pgfmathsetmacro\hscale{6.283185307179586}
            \pgfmathsetmacro\vscale{1}

            \draw [draw=black, line width=0.5pt, fill=none]
                ({-0.2 * \hscale * \hunit}, {-2.25 * \vscale * \vunit})
                -- ({1.2 * \hscale * \hunit}, {-2.25 * \vscale * \vunit})
                -- ({1.2 * \hscale * \hunit}, {4.0 * \vscale * \vunit})
                -- ({-0.2 * \hscale * \hunit}, {4.0 * \vscale * \vunit})
                -- cycle;

            \draw [draw=none, fill=none]
                ({0 * \hscale * \hunit}, {4.0 * \vscale * \vunit})
                node[label={[below right, xshift=0pt, yshift=-2pt] \scriptsize $U(t)^{\ast}P(s)U(t)f\:(e^{\iunit \theta})$ vs. $\theta$}]{};

            \draw[->, draw=black, line width=0.5pt, fill=black]
                ({-0.1 * \hscale * \hunit}, {0 * \vscale * \vunit})
                -- ({1.1 * \hscale * \hunit}, {0 * \vscale * \vunit})
                node[label={[below, xshift=0pt, yshift=-2pt] \scriptsize $\theta$}]{};

            \draw [draw=black, line width=1.5pt, fill=none]
                ({0.32000000000000006 * \hscale * \hunit}, {-0.75 * \vunit * \vscale})
                node[label={[below, xshift=0pt, yshift=-2pt]\scriptsize \rotatebox[origin=c]{-90}{$2\pi t_{1}$}}]{}
                -- ({0.32000000000000006 * \hscale * \hunit}, {-0.25 * \vscale * \vunit});
            \draw [draw=black, line width=1.5pt, fill=none]
                ({1.0 * \hscale * \hunit}, {-0.75 * \vunit * \vscale})
                node[label={[below, xshift=0pt, yshift=-2pt]\scriptsize \rotatebox[origin=c]{-90}{$2\pi$}}]{}
                -- ({1.0 * \hscale * \hunit}, {-0.25 * \vscale * \vunit});
            \draw [draw=black, line width=1.5pt, fill=none]
                ({0.0 * \hscale * \hunit}, {-0.75 * \vunit * \vscale})
                node[label={[below, xshift=0pt, yshift=-2pt]\scriptsize \rotatebox[origin=c]{-90}{$0$}}]{}
                -- ({0.0 * \hscale * \hunit}, {-0.25 * \vscale * \vunit});

            \draw[draw=none, line width=0pt, fill=black] ({0.0 * \hscale * \hunit}, {0.0 * \vscale * \vunit}) circle (2pt);

            \draw[draw=black, line width=1pt, fill=none]
                ({0.0 * \hscale * \hunit}, {0.0 * \vscale * \vunit})
                -- ({0.040000000000000036 * \hscale * \hunit}, {0.0 * \vscale * \vunit})
                -- ({0.07999999999999996 * \hscale * \hunit}, {0.0 * \vscale * \vunit})
                -- ({0.12 * \hscale * \hunit}, {0.0 * \vscale * \vunit})
                -- ({0.16000000000000003 * \hscale * \hunit}, {0.0 * \vscale * \vunit})
                -- ({0.19999999999999996 * \hscale * \hunit}, {0.0 * \vscale * \vunit})
                -- ({0.24 * \hscale * \hunit}, {0.0 * \vscale * \vunit})
                -- ({0.28 * \hscale * \hunit}, {0.0 * \vscale * \vunit})
                -- ({0.32000000000000006 * \hscale * \hunit}, {0.0 * \vscale * \vunit});

            \draw[draw=black, line width=1pt, fill=none] ({0.32000000000000006 * \hscale * \hunit}, {0.0 * \vscale * \vunit}) circle (2pt);

            \draw[draw=black, line width=1pt, dotted]
                ({0.32000000000000006 * \hscale * \hunit}, {0.0 * \vscale * \vunit})
                -- ({0.32000000000000006 * \hscale * \hunit}, {2.012366485978751 * \vscale * \vunit});

            \draw[draw=none, line width=0pt, fill=black] ({0.32000000000000006 * \hscale * \hunit}, {2.012366485978751 * \vscale * \vunit}) circle (2pt);

            \draw[draw=black, line width=1pt, fill=none]
                ({0.32000000000000006 * \hscale * \hunit}, {2.012366485978751 * \vscale * \vunit})
                -- ({0.36 * \hscale * \hunit}, {1.5681222791371292 * \vscale * \vunit})
                -- ({0.4 * \hscale * \hunit}, {1.96944178517839 * \vscale * \vunit})
                -- ({0.44 * \hscale * \hunit}, {1.5510712933545452 * \vscale * \vunit})
                -- ({0.48 * \hscale * \hunit}, {2.0885407526360256 * \vscale * \vunit})
                -- ({0.52 * \hscale * \hunit}, {2.2113303695203994 * \vscale * \vunit})
                -- ({0.56 * \hscale * \hunit}, {1.9011959811801344 * \vscale * \vunit})
                -- ({0.6 * \hscale * \hunit}, {0.9636135216665631 * \vscale * \vunit})
                -- ({0.64 * \hscale * \hunit}, {2.397522383621427 * \vscale * \vunit})
                -- ({0.68 * \hscale * \hunit}, {1.1233020661612172 * \vscale * \vunit})
                -- ({0.72 * \hscale * \hunit}, {1.2940431182244299 * \vscale * \vunit})
                -- ({0.76 * \hscale * \hunit}, {2.144293658942074 * \vscale * \vunit})
                -- ({0.8 * \hscale * \hunit}, {2.753043391732871 * \vscale * \vunit})
                -- ({0.84 * \hscale * \hunit}, {2.454921971614578 * \vscale * \vunit})
                -- ({0.88 * \hscale * \hunit}, {2.175109235559883 * \vscale * \vunit})
                -- ({0.92 * \hscale * \hunit}, {2.2572062167331226 * \vscale * \vunit})
                -- ({0.96 * \hscale * \hunit}, {2.1613703578538948 * \vscale * \vunit})
                -- ({1.0 * \hscale * \hunit}, {2.8367977269110867 * \vscale * \vunit});

            \draw[draw=black, line width=1pt, fill=none] ({1.0 * \hscale * \hunit}, {2.8367977269110867 * \vscale * \vunit}) circle (2pt);
        \end{tikzpicture}
        \caption{$\fractional{s} + \fractional{t} = 1$}
        \label{fig:bhat-skeide:equals:sig:article-interpolation-raj-dahya}
    \end{subfigure}
    \begin{subfigure}[b]{0.3\textwidth}
        \centering
        \begin{tikzpicture}[node distance=1cm,thick]
            \pgfmathsetmacro\hunit{0.5}
            \pgfmathsetmacro\vunit{0.5}
            \pgfmathsetmacro\hscale{6.283185307179586}
            \pgfmathsetmacro\vscale{1}

            \draw [draw=black, line width=0.5pt, fill=none]
                ({-0.2 * \hscale * \hunit}, {-2.25 * \vscale * \vunit})
                -- ({1.2 * \hscale * \hunit}, {-2.25 * \vscale * \vunit})
                -- ({1.2 * \hscale * \hunit}, {4.0 * \vscale * \vunit})
                -- ({-0.2 * \hscale * \hunit}, {4.0 * \vscale * \vunit})
                -- cycle;

            \draw [draw=none, fill=none]
                ({0 * \hscale * \hunit}, {4.0 * \vscale * \vunit})
                node[label={[below right, xshift=0pt, yshift=-2pt] \scriptsize $U(t)^{\ast}P(s)U(t)f\:(e^{\iunit \theta})$ vs. $\theta$}]{};

            \draw[->, draw=black, line width=0.5pt, fill=black]
                ({-0.1 * \hscale * \hunit}, {0 * \vscale * \vunit})
                -- ({1.1 * \hscale * \hunit}, {0 * \vscale * \vunit})
                node[label={[below, xshift=0pt, yshift=-2pt] \scriptsize $\theta$}]{};

            \draw [draw=black, line width=1.5pt, fill=none]
                ({0.16000000000000003 * \hscale * \hunit}, {-0.75 * \vunit * \vscale})
                node[label={[below, xshift=0pt, yshift=-2pt]\scriptsize \rotatebox[origin=c]{-90}{$2\pi t_{1}$}}]{}
                -- ({0.16000000000000003 * \hscale * \hunit}, {-0.25 * \vscale * \vunit});
            \draw [draw=black, line width=1.5pt, fill=none]
                ({0.8400000000000001 * \hscale * \hunit}, {-0.75 * \vunit * \vscale})
                node[label={[below, xshift=0pt, yshift=-2pt]\scriptsize \rotatebox[origin=c]{-90}{$2\pi t_{2}$}}]{}
                -- ({0.8400000000000001 * \hscale * \hunit}, {-0.25 * \vscale * \vunit});
            \draw [draw=black, line width=1.5pt, fill=none]
                ({0.0 * \hscale * \hunit}, {-0.75 * \vunit * \vscale})
                node[label={[below, xshift=0pt, yshift=-2pt]\scriptsize \rotatebox[origin=c]{-90}{$0$}}]{}
                -- ({0.0 * \hscale * \hunit}, {-0.25 * \vscale * \vunit});
            \draw [draw=black, line width=1.5pt, fill=none]
                ({1.0 * \hscale * \hunit}, {-0.75 * \vunit * \vscale})
                node[label={[below, xshift=0pt, yshift=-2pt]\scriptsize \rotatebox[origin=c]{-90}{$2\pi$}}]{}
                -- ({1.0 * \hscale * \hunit}, {-0.25 * \vscale * \vunit});

            \draw[draw=none, line width=0pt, fill=black] ({0.0 * \hscale * \hunit}, {0.0 * \vscale * \vunit}) circle (2pt);

            \draw[draw=black, line width=1pt, fill=none]
                ({0.0 * \hscale * \hunit}, {0.0 * \vscale * \vunit})
                -- ({0.040000000000000036 * \hscale * \hunit}, {0.0 * \vscale * \vunit})
                -- ({0.07999999999999996 * \hscale * \hunit}, {0.0 * \vscale * \vunit})
                -- ({0.12 * \hscale * \hunit}, {0.0 * \vscale * \vunit})
                -- ({0.16000000000000003 * \hscale * \hunit}, {0.0 * \vscale * \vunit});

            \draw[draw=black, line width=1pt, fill=none] ({0.16000000000000003 * \hscale * \hunit}, {0.0 * \vscale * \vunit}) circle (2pt);

            \draw[draw=black, line width=1pt, dotted]
                ({0.16000000000000003 * \hscale * \hunit}, {0.0 * \vscale * \vunit})
                -- ({0.16000000000000003 * \hscale * \hunit}, {2.0310870311347475 * \vscale * \vunit});

            \draw[draw=none, line width=0pt, fill=black] ({0.16000000000000003 * \hscale * \hunit}, {2.0310870311347475 * \vscale * \vunit}) circle (2pt);

            \draw[draw=black, line width=1pt, fill=none]
                ({0.16000000000000003 * \hscale * \hunit}, {2.0310870311347475 * \vscale * \vunit})
                -- ({0.2 * \hscale * \hunit}, {1.8776400281344068 * \vscale * \vunit})
                -- ({0.24 * \hscale * \hunit}, {1.7657712943879427 * \vscale * \vunit})
                -- ({0.28 * \hscale * \hunit}, {2.2568778364637416 * \vscale * \vunit})
                -- ({0.32 * \hscale * \hunit}, {2.012366485978751 * \vscale * \vunit})
                -- ({0.36 * \hscale * \hunit}, {1.5681222791371292 * \vscale * \vunit})
                -- ({0.4 * \hscale * \hunit}, {1.96944178517839 * \vscale * \vunit})
                -- ({0.44 * \hscale * \hunit}, {1.5510712933545452 * \vscale * \vunit})
                -- ({0.48 * \hscale * \hunit}, {2.0885407526360256 * \vscale * \vunit})
                -- ({0.52 * \hscale * \hunit}, {2.2113303695203994 * \vscale * \vunit})
                -- ({0.56 * \hscale * \hunit}, {1.9011959811801344 * \vscale * \vunit})
                -- ({0.6 * \hscale * \hunit}, {0.9636135216665631 * \vscale * \vunit})
                -- ({0.64 * \hscale * \hunit}, {2.397522383621427 * \vscale * \vunit})
                -- ({0.68 * \hscale * \hunit}, {1.1233020661612172 * \vscale * \vunit})
                -- ({0.72 * \hscale * \hunit}, {1.2940431182244299 * \vscale * \vunit})
                -- ({0.76 * \hscale * \hunit}, {2.144293658942074 * \vscale * \vunit})
                -- ({0.8 * \hscale * \hunit}, {2.753043391732871 * \vscale * \vunit})
                -- ({0.84 * \hscale * \hunit}, {2.454921971614578 * \vscale * \vunit});

            \draw[draw=black, line width=1pt, fill=none] ({0.84 * \hscale * \hunit}, {2.454921971614578 * \vscale * \vunit}) circle (2pt);

            \draw[draw=black, line width=1pt, dotted]
                ({0.84 * \hscale * \hunit}, {2.454921971614578 * \vscale * \vunit})
                -- ({0.84 * \hscale * \hunit}, {0.0 * \vscale * \vunit});

            \draw[draw=none, line width=0pt, fill=black] ({0.84 * \hscale * \hunit}, {0.0 * \vscale * \vunit}) circle (2pt);

            \draw[draw=black, line width=1pt, fill=none]
                ({0.84 * \hscale * \hunit}, {0.0 * \vscale * \vunit})
                -- ({0.88 * \hscale * \hunit}, {0.0 * \vscale * \vunit})
                -- ({0.92 * \hscale * \hunit}, {0.0 * \vscale * \vunit})
                -- ({0.96 * \hscale * \hunit}, {0.0 * \vscale * \vunit})
                -- ({1.0 * \hscale * \hunit}, {0.0 * \vscale * \vunit});

            \draw[draw=black, line width=1pt, fill=none] ({1.0 * \hscale * \hunit}, {0.0 * \vscale * \vunit}) circle (2pt);
        \end{tikzpicture}
        \caption{$\fractional{s} + \fractional{t} > 1$}
        \label{fig:bhat-skeide:greater:sig:article-interpolation-raj-dahya}
    \end{subfigure}

    \caption{%
        Effect of
            $U(t)^{\ast}P(s)U(t)$
        on $L^{2}(\Torus)$
        for $s,t\in\reals \without \integers$.\\
        Here $t_{1} = 1 - \fractional{t}$
        and $t_{2} = 1 - \fractional{s + t}$.
    }
    \label{fig:bhat-skeide:sig:article-interpolation-raj-dahya}
\end{mdframed}
\end{figure}


\begin{lemm}[Bhat--Skeide, 2015]
\makelabel{lemm:bhat-skeide:single:sig:article-interpolation-raj-dahya}
    Let $S\in\BoundedOps{\HilbertRaum}$
    be a bounded operator (\respectively contraction)
    on a Hilbert space $\HilbertRaum$.
    Then
        ${T \colon \realsNonNeg \to \BoundedOps{L^{2}(\Torus) \otimes \HilbertRaum}}$
    given by

        \begin{restoremargins}
        \begin{equation}
        \label{eq:bs-dilation-single:sig:article-interpolation-raj-dahya}
            T(t) \coloneqq (U(t) \otimes \onematrix)
                \Big(
                    P(t) \otimes S^{\floor{t}}
                    + (\onematrix - P(t)) \otimes S^{\floor{t} + 1}
                \Big)
        \end{equation}
        \end{restoremargins}

    \continueparagraph
    for $t\in\realsNonNeg$,
    is a $\Cnought$-semigroup
    (\respectively contractive $\Cnought$-semigroup)
    on $L^{2}(\Torus) \otimes \HilbertRaum$
    satisfying $T(n) = \onematrix \otimes S^{n}$
    for $n\in\naturalsZero$.
\end{lemm}

    \begin{proof}[Sketch]
        First observe that $T$ has bounded growth:
        Since $P$ is projection valued and $U$ is a unitary semigroup,
        one has
            $
                \norm{T(t)}
                    \leq \norm{P(t)}\norm{S}^{\floor{t}}
                        + \norm{\onematrix - P(t)}\norm{S}^{\floor{t} + 1}
                    \leq 2\max\{1,\norm{S}\}^{\floor{t} + 1}
                    \leq C e^{\omega t}
            $
        for all $t\in\realsNonNeg$,
        where $C = 2\max\{1,\norm{S}\}$
        und $\omega = \log(\max\{1,\norm{S}\})$.

        Secondly, observe that if $S$ is a contraction,
        then $T$ is a contraction-valued function:
        Letting $t \in \realsNonNeg$,
        since $P(t)$ a projection,
        it is clear that
            $P(t) \otimes \onematrix + (\onematrix - P(t)) \otimes S$
        is a contraction.
        And since $U(t)$ is unitary,
        it follows that
            $
                T(t) = (U(t) \otimes \onematrix)
                    (P(t) \otimes \onematrix + (\onematrix - P(t)) \otimes S)
                    (\onematrix \otimes S^{\floor{t}})
            $
        is contractive.
        We now demonstrate the remaining properties.

        \paragraph{Interpolation property:}
        For $n\in\naturalsZero$ we have
            $
                T(n)
                = (U(n) \otimes \onematrix)
                    (P(n) \otimes S^{\floor{n}} + (\onematrix - P(n)) \otimes S^{\floor{n} + 1})
                = (U(0) \otimes \onematrix)
                    (P(0) \otimes S^{n} + (\onematrix - P(0)) \otimes S^{n+1})
                = (\onematrix \otimes \onematrix)
                    (\onematrix \otimes S^{n} + (\onematrix - \onematrix) \otimes S^{n+1})
                = \onematrix \otimes S^{n}
            $.

        \paragraph{Semigroup law:}
        We have
            $
                T(0) = \onematrix \otimes S^{0}
                = \onematrix \otimes \onematrix
                = \onematrix
            $.
        Let $s,t\in\realsNonNeg$ be arbitrary.
        We show that $T(s)T(t) = T(s + t)$.
        By \Cref{prop:bscr:sig:article-interpolation-raj-dahya},
            $P(s)U(t) = U(t)Q(s,t)$,
        where
            $Q(s,t) = \onematrix - (P(t) - P(s+t))$ if $\fractional{s} + \fractional{t} < 1$
            and
            $Q(s,t) = P(s+t) - P(t)$ otherwise.
        This yields

            \begin{shorteqnarray}
                T(s)T(t)
                    &= &(U(s) \otimes \onematrix)
                        (P(s) \otimes S^{\floor{s}} + (\onematrix - P(s)) \otimes S^{\floor{s} + 1})\\
                    &&\cdot
                        (U(t) \otimes \onematrix)
                        (P(t) \otimes S^{\floor{t}} + (\onematrix - P(t)) \otimes S^{\floor{t} + 1})\\
                    &=
                        &(U(s) \otimes \onematrix)
                        (U(t) \otimes \onematrix)\\
                    &&\cdot
                        (
                            Q(s,t) \otimes S^{\floor{s}}
                            + (\onematrix - Q(s,t)) \otimes S^{\floor{s} + 1}
                        )\\
                    &&\cdot
                        (P(t) \otimes S^{\floor{t}} + (\onematrix - P(t)) \otimes S^{\floor{t} + 1}).\\
            \end{shorteqnarray}

        \continueparagraph
        Relying entirely on the relations between projections $P(t)$ and $P(s+t)$,
        the product of the final two expressions in parentheses
        simplifies to $P(s+t)\otimes S^{\floor{s+t}} + (\onematrix - P(s+t))\otimes S^{\floor{s+t} + 1}$,
        and thus the final expression is equal to $T(s+t)$.
        For details of this computation we refer the reader to the proof of \cite[Lemma~2.4~(1)]{BhatSkeide2015PureCounterexamples}.
        Note in particular that the properties of $S$ do not play any role in this simplification.

        \paragraph{Continuity of of $T$:}
        Since $T$ satisfies the semigroup law,
        it suffices to prove that $T$ is continuous in $0$.
        Since $U$ is $\topSOT$-continuous
        and $P$ is right-$\topSOT$-continuous,
        one has for $t \in (0,\:1)$ that
            $
                T(t)
                = U(t)P(t) \otimes S^{\floor{t}} + U(t)(\onematrix - P(t)) \otimes S^{\floor{t} + 1}
                = U(t)P(t) \otimes \onematrix + U(t)(\onematrix - P(t)) \otimes S
                \overset{\tinytopSOT}{\longrightarrow}
                U(0)P(0) \otimes \onematrix + U(0)(\onematrix - P(0)) \otimes S
                = \onematrix
                = T(0)
            $
        as ${t \longrightarrow 0^{+}}$.
        Hence, $T$ is a $\Cnought$-semigroup.
    \end{proof}

Before proceeding, we take note of some basic but interesting
properties of the Bhat--Skeide interpolation.
Let $S \in \BoundedOps{\HilbertRaum}$ be a bounded operator on a Hilbert space $\HilbertRaum$
and ${T \colon \realsNonNeg \to \BoundedOps{L^{2}(\Torus) \otimes \HilbertRaum}}$
its interpolation as defined in \Cref{lemm:bhat-skeide:single:sig:article-interpolation-raj-dahya}.
As in the proof of this result,
if $S$ is a contraction, then so is $T(t)$ for all $t\in\realsNonNeg$.
Conversely, if $T$ is contraction-valued,
then since
    $\onematrix \otimes S = T(1)$,
it follows that
    $
        1 \geq \norm{T(1)}
        = \norm{\onematrix \otimes S}
        = \norm{\onematrix}\norm{S}
    $,
whence $S$ is a contraction.
Whilst our focus in this paper is on
contractions and contractive $\Cnought$-semigroups,
it is worth noting that the interpolation preserves other operator-theoretic properties.
For example, if $S$ is an isometry,
then relying on the orthogonality of the projections
$\{P(t),\onematrix-P(t)\}$ for each $t\in\realsNonNeg$,
one obtains

    \begin{shorteqnarray}
        T(t)^{\ast}T(t)
            &=
                &\begin{displayarray}[t]{0l}
                    \Big(
                        P(t) \otimes (S^{\ast})^{\floor{t}}
                        + (\onematrix - P(t)) \otimes (S^{\ast})^{\floor{t} + 1}
                    \Big)
                    (U(t)^{\ast} \otimes \onematrix)\\
                    \cdot
                    (U(t) \otimes \onematrix)
                    \Big(
                        P(t) \otimes S^{\floor{t}}
                        + (\onematrix - P(t)) \otimes S^{\floor{t} + 1}
                    \Big)
                    \\
                \end{displayarray}\\
            &=
                &P(t) \otimes (S^{\ast})^{\floor{t}}S^{\floor{t}}
                + (\onematrix - P(t)) \otimes (S^{\ast})^{\floor{t} + 1}S^{\floor{t} + 1}\\
            &=
                &P(t) \otimes \onematrix
                + (\onematrix - P(t)) \otimes \onematrix
            =
                \onematrix,\\
    \end{shorteqnarray}

\continueparagraph
\idest $T(t)$ is an isometry for all $t\in\realsNonNeg$.
Conversely, if $T$ is isometry-valued,
then
    $
        \onematrix \otimes S^{\ast} S
        = (\onematrix \otimes S)^{\ast}(\onematrix \otimes S)
        = T(1)^{\ast}T(1)
        = \onematrix
    $,
from which it follows that $S$ is an isometry.
In an analogous manner one obtains that $T$ is unitary-valued
if and only if $S$ is a unitary.

The interpolation also preserves properties in the stochastic setting.
Considering $L^{p}$-spaces over probability spaces,
it has been established that the properties of operators (or their duals)
preserving the constant $\einser$-function
and \almostEverywhere-positive functions
provide operator-theoretic counterparts
to the stochastic notions of transition kernels.
For details on these connections,
see \exempli
    \cite[\S{}0.9, (9.6--7)]{Doob1990BookStochProc},
    \cite[\S{}2.3 and Theorem~2.1]{Dynkin1965BookMarkovVolI},
    \cite[\S{}1.1--1.2]{Vershik2006Article},
    \cite[Chapter~13]{EisnerNagel2015BookErgTh}.
Here we simply focus on the operator-theoretic properties
in and of themselves, in the $L^{2}$-setting.

So consider a Hilbert space
    $L^{2}(X, \mu)$
for some probability space $(X, \mu)$,
where $X$ is a locally compact Hausdorff space.
Let
    $L^{2}(X, \mu)^{+}$
denote the closed convex subspace of
    $f \in L^{2}(X, \mu)$
which are \almostEverywhere-positive,
\idest
    $f(x) \in \realsNonNeg$
for $\mu$-\almostEverywhere $x \in X$.
Say that an operator
    $S \in \BoundedOps{L^{2}(X, \mu)}$
preserves \almostEverywhere-positive functions
if $S L^{2}(X, \mu)^{+} \subseteq L^{2}(X, \mu)^{+}$.
And say that $S$ is unity-preserving
if $S \einser_{X} = \einser_{X}$.
Now, the tensor product
    $L^{2}(\Torus) \otimes L^{2}(X, \mu)$.
may be identified with
    $L^{2}(\Torus \times X)$
in the canonical way,%
\footnote{
    For
        $f \in L^{\infty}(\Torus)$
        and
        $g \in L^{\infty}(X,\mu)$
    let $f \odot g \in L^{\infty}(\Torus \times X)$
    be defined by
        ${(\theta,x) \mapsto f(\theta) g(x)}$.
    Considering the algebraic tensor product
    $\Cts{\Torus} \otimes \Cts{X}$,
    one can readily observe that the linear map
    ${
        \Cts{\Torus} \otimes \Cts_{b}{X}
        \ni \sum_{i} f_{i} \otimes g_{i}
        \mapsto
        \sum_{i} f_{i} \odot g_{i}
        \in \Cts_{b}{\Torus \times X}
    }$
    is a well-defined $L^{2}$-isometry.
    By Stone-Weierstra\ss, it also has an
    $L^{\infty}$- and thus $L^{2}$-dense image inside $\Cts_{b}{\Torus \times X}$.
    By density, it follows that this map
    extends uniquely to a unitary map between
    $L^{2}(\Torus) \otimes L^{2}(X, \mu)$
    and
    $L^{2}(\Torus \times X)$.
}
where
    $\Torus$ is endowed with the Haar measure
and
    $\Torus \times X$ is endowed with the product probability measure.
Under this identification one further has that
    $\einser_{\Torus \times X}$
is identical to the element
    $\einser_{\Torus} \otimes \einser_{X}$.

With this in mind, we now consider
a bounded operator $S \in \BoundedOps{L^{2}(X, \mu)}$
and its Bhat--Skeide interpolation
    ${T \colon \realsNonNeg \to \BoundedOps{L^{2}(\Torus) \otimes L^{2}(X, \mu)}}$.
If $S$ is unity-preserving, then

    \begin{shorteqnarray}
        T(t)\einser_{X \times \Torus}
        &=
            &(U(t) \otimes \onematrix)
            \Big(
                P(t) \otimes S^{\floor{t}}
                + (\onematrix - P(t)) \otimes S^{\floor{t} + 1}
            \Big)
            \:
            (\einser_{\Torus} \otimes \einser_{X})\\
        &=
            &(U(t) \otimes \onematrix)
            \Big(
                P(t)\einser_{\Torus} \otimes S^{\floor{t}}\einser_{X}
                + (\onematrix - P(t))\einser_{\Torus} \otimes S^{\floor{t} + 1}\einser_{X}
            \Big)\\
        &=
            &(U(t) \otimes \onematrix)
            \Big(
                P(t)\einser_{\Torus} \otimes \einser_{X}
                + (\onematrix - P(t))\einser_{\Torus} \otimes \einser_{X}
            \Big)\\
        &=
            &(U(t) \otimes \onematrix)
            \:(\einser_{\Torus} \otimes \einser_{X})\\
        &=
            &U(t)\einser_{\Torus} \otimes \einser_{X}\\
        &=
            &\einser_{\Torus} \otimes \einser_{X}
            ~\text{(since $U(t)$ is a Koopman-operator)}\\
        &= &\einser_{\Torus \times X},\\
    \end{shorteqnarray}

\continueparagraph
\idest $T(t)$ is unity-preserving for all $t\in\realsNonNeg$.
Conversely, if this holds, then
    $
        \einser_{\Torus} \otimes S\einser_{X}
        = (\onematrix \otimes S)\:(\einser_{\Torus} \otimes \einser_{X})
        = T(1)\:\einser_{\Torus \times X}
        = \einser_{\Torus \times X}
        = \einser_{\Torus} \otimes \einser_{X}
    $,
from which one can derive that
$S\einser_{X} = \einser_{X}$.
Similarly, $S^{\ast}$ is unity-preserving
if and only if $T(t)^{\ast}$ is unity-preserving for all $t\in\realsNonNeg$.

Suppose now that $S$ preserves \almostEverywhere-positive functions.
For $t \in \realsNonNeg$,
since $U(t)$ is a Koopman-operator
and $P(t)$ a projection and multiplication operator,
one obtains

    \begin{shorteqnarray}
        \brkt{
            T(t)^{\ast}\:(f \otimes g)
        }{
            \tilde{f} \otimes \tilde{g}
        }
        &= &\brkt{
                f \otimes g
            }{
                T(t)\:(\tilde{f} \otimes \tilde{g})
            }\\
        &=
            &\begin{displayarray}[t]{0l}
                \brkt{f}{U(t)P(t)\tilde{f}}
                \brkt{g}{S^{\floor{t}}\tilde{g}}\\
                +\:\brkt{f}{U(t)(\onematrix - P(t))\tilde{f}}
                \brkt{g}{S^{\floor{t} + 1}\tilde{g}}\\
            \end{displayarray}\\
        &\geq &0\\
    \end{shorteqnarray}

\continueparagraph
for all
    $g,\tilde{g} \in L^{2}(\Torus)^{+}$
    and
    $f,\tilde{f} \in L^{2}(X, \mu)^{+}$,
which implies that
    $T(t)^{\ast}\:(f \otimes g) \geq 0$ \almostEverywhere.
From this it follows that
    $
        \brkt{T(t)F}{f \otimes g}
        = \brkt{F}{T(t)^{\ast}\:(f \otimes g)}
        \geq 0
    $
for all
    $F \in L^{2}(\Torus \times X)^{+}$,
    $f \in L^{2}(\Torus)^{+}$,
    and
    $g \in L^{2}(X, \mu)^{+}$,
which in turn implies that
    $T(t)F \geq 0$ \almostEverywhere.
It thus follows that $T(t)$ preserves \almostEverywhere-positive functions
for all $t\in\realsNonNeg$.
Conversely, if this holds, then
    $
        \brkt{Sf}{\tilde{f}}
        = \brkt{
            (\onematrix \otimes S)
            \:(\einser_{\Torus} \otimes f)
        }{
            \einser_{\Torus} \otimes \tilde{f}
        }
        = \brkt{
            T(1)\:(\einser_{\Torus} \otimes f)
        }{
            \einser_{\Torus} \otimes \tilde{f}
        }
        \geq 0
    $
for \almostEverywhere-positive functions
$f,\tilde{f} \in L^{2}(X, \mu)$,
from which it follows that $S$ preserves \almostEverywhere-positive functions.
Similarly,
    $S^{\ast}$ preserves \almostEverywhere-positive functions
    if and only if
    $T(t)^{\ast}$ preserves \almostEverywhere-positive functions
for all $t\in\realsNonNeg$.

We thus see that the Bhat--Skeide interpolation
preserves operator-theoretic properties relevant
in the characterisation Markov and bi-Markov processes.
In particular, by applying
    \cite[Theorem 13.2 and Proposition~13.6]{EisnerNagel2015BookErgTh}
one has that $S \in \BoundedOps{L^{2}(X,\mu)}$
is a \usesinglequotes{bi-Markov} operator%
\footref{ft:prose:bimarkov:sig:article-interpolation-raj-dahya}
if and only if $T(t)$ is for all $t\in\realsNonNeg$.

\footnotetext[ft:prose:bimarkov:sig:article-interpolation-raj-dahya]{%
    we refer the interested reader to the start of \cite[Chapter~13]{EisnerNagel2015BookErgTh}
    for the definition of bi-Markov operators.
}



\subsection[Interpolation for families of contractions]{Interpolation for families of contractions}
\label{sec:interpolation:family:sig:article-interpolation-raj-dahya}

\firstparagraph
Now consider a non-empty index set $\mathcal{I}$
and a commuting family
    $\{S_{i}\}_{i \in \mathcal{I}}$
of contractions on a Hilbert space $\HilbertRaum$.
The Bhat--Skeide interpolation can be extended
to arbitrary families by defining
the interpolants in such a way,
that the auxiliary parts act on independent co-ordinates.

\def\beweislabel{lemm:bhat-skeide:multi:sig:article-interpolation-raj-dahya}
\begin{proof}[of \Cref{\beweislabel}]
    Consider the probability space
        $Z \coloneqq \prod_{i \in \mathcal{I}}\Torus$.
    For each $i\in \mathcal{I}$ and $t\in\reals$ define
        ${\phi^{(i)}_{t} \colon Z \to Z}$
        and
        ${p^{(i)}_{t} \colon Z \to \reals}$
    via

        \begin{shorteqnarray}
            \phi^{(i)}_{t}(\mathbf{z})
                &\coloneqq &\Big(
                        \begin{cases}
                            \phi_{t}(z_{i}) &\colon &j = i\\
                            z_{j} &\colon &j \neq i\\
                        \end{cases}
                    \Big)_{j\in\mathcal{I}}\\
            p^{(i)}_{t}(\mathbf{z})
                &\coloneqq &p_{t}(z_{i})\\
        \end{shorteqnarray}

    \continueparagraph
    for $\mathbf{z} = (z_{j})_{j\in\mathcal{I}} \in Z$,
    where
        $\phi_{t}$, $p_{t}$
    are defined as in \S{}\ref{sec:interpolation:single:sig:article-interpolation-raj-dahya}.
    Clearly, $\phi^{(i)}_{t}$ is a measure-preserving bijection
    and $p^{(i)}_{t}$ is measurable.
    As in \S{}\ref{sec:interpolation:single:sig:article-interpolation-raj-dahya},
    defining
        ${U_{i}, P_{i} \colon \reals \to \BoundedOps{L^{2}(Z)}}$
    via
        $U_{i}(t)f = f \circ \phi^{(i)}_{t}$ (Koopman-operator)
        and
        $P_{i}(t)f = M_{p^{(i)}_{t}}f = p^{(i)}_{t}f$
    for $t\in\reals$, $f\in L^{2}(Z)$,
    we have that
        $U_{i} \in \Repr{\reals}{L^{2}(Z)}$
    and $P_{i}$ is right-$\topSOT$-continuous and projection-valued.
    Moreover $U_{i},P_{i}$ are periodic with
        $U_{i}(n) = U_{i}(0) = P_{i}(n) = P_{i}(0) = \onematrix$
    for $n\in\naturals$.

    Now, the commutation relation \eqcref{eq:bscr:sig:article-interpolation-raj-dahya}
    in \Cref{prop:bscr:sig:article-interpolation-raj-dahya}
    is clearly satisfied by $P_{i},U_{i}$.
    The proof of \Cref{lemm:bhat-skeide:single:sig:article-interpolation-raj-dahya}
    is thus clearly applicable to $P_{i},U_{i},S_{i}$.
    This yields that the operator-valued function
        ${T_{i} \colon \realsNonNeg \to \BoundedOps{L^{2}(Z) \otimes \HilbertRaum}}$
    given for $t\in\realsNonNeg$ by

        \begin{restoremargins}
        \begin{equation}
        \label{eq:bs-dilation-family:sig:article-interpolation-raj-dahya}
            T_{i}(t) \coloneqq (U_{i}(t) \otimes \onematrix)
                    \Big(
                        P_{i}(t) \otimes S_{i}^{\floor{t}}
                        + (\onematrix - P_{i}(t)) \otimes S_{i}^{\floor{t} + 1}
                    \Big)
        \end{equation}
        \end{restoremargins}

    \continueparagraph
    (\cf \eqcref{eq:bs-dilation-single:sig:article-interpolation-raj-dahya})
    defines a $\Cnought$-semigroup
    (%
        \respectively contractive $\Cnought$-semigroup,
        if $S_{i}$ is contractive%
    )
    on $L^{2}(Z) \otimes \HilbertRaum$
    satisfying $T_{i}(n) = \onematrix \otimes S_{i}^{n}$
    for $n\in\naturalsZero$.

    Let $i,j\in\mathcal{I}$ with $i \neq j$.
    Since the operators act on independent co-ordinates,
    it is routine to show that each of the operators in
        $\{P_{i}(s),U_{i}(t) \mid s,t\in\reals\}$
    commutes with each of the operators in
        $\{P_{j}(s),U_{j}(t) \mid s,t\in\reals\}$.
    Since $S_{i}$ and $S_{j}$ commute,
    it follows from the above construction
    that
        $T_{i}$ and $T_{j}$
    commute.
    Hence $\{T_{i}\}_{i\in\mathcal{I}}$
    is a commuting family of $\Cnought$-semigroups
    (%
        \respectively contractive $\Cnought$-semigroups,
        if each of the $S_{i}$ are contractive%
    )
    satisfying
        $
            \prod_{i\in\mathcal{I}}
                T_{i}(n_{i})
            = \prod_{i\in\mathcal{I}}
                (\onematrix \otimes S_{i}^{n_{i}})
        $
    for all $\mathbf{n}\in \prod_{i\in\mathcal{I}}\naturalsZero$
    with $\supp(\mathbf{n})$ finite.
\end{proof}

\begin{rem}[Non-commutative setting]
\makelabel{rem:non-commutative-interpolation:sig:article-interpolation-raj-dahya}
    Suppose $\{S_{i}\}_{i \in \mathcal{I}}$ is a \textit{non-commuting} family of contractions.
    We may still construct the interpolations $\{T_{i}\}_{i \in \mathcal{I}}$ of $\{S_{i}\}_{i \in \mathcal{I}}$
    exactly as in the proof of \Cref{lemm:bhat-skeide:multi:sig:article-interpolation-raj-dahya},
    except that this will no longer be a commuting family.
    If we now impose a scheme of commutation relations on the discrete family $\{S_{i}\}_{i \in \mathcal{I}}$,
    one may consider whether an appropriate schema of commutation relations is satisfied by the continuous family.
    For example, consider the Weyl form of the canonical commutation relations (CCR)
    presented in
    \cite[Example~1.8]{Dahya2024approximation}.
    The discrete counterpart of this can be presented as follows:
    Let
        $d\in\naturals$
        and
        $\{S_{i}\}_{i=1}^{d} \subseteq \BoundedOps{\HilbertRaum}$
        be a family of contractions
    satisfying
        $S_{j}S_{i} = V_{ij}S_{i}S_{j}$
    for $i,j\in\{1,2,\ldots,d\}$ with $i \neq j$.
    Here
        $\{V_{ij}\}_{i,j=1}^{d} \subseteq \BoundedOps{\HilbertRaum}$
    is a
    family of invertible operators,
    each of which commutes with each of the $S_{i}$.
    Now, by a simple induction argument one obtains the commutation relations

        \begin{restoremargins}
        \begin{equation}
        \label{eq:1:\beweislabel}
            S_{j}^{n}S_{i}^{m} = V_{ij}^{mn}S_{i}^{m}S_{j}^{n}
        \end{equation}
        \end{restoremargins}

    \continueparagraph
    for $m,n\in\naturalsZero$
    and $i,j\in\{1,2,\ldots,d\}$ with $i \neq j$.

    Consider now the interpolation $\{T_{i}\}_{i=1}^{d}$ of $\{S_{i}\}_{i=1}^{d}$
    as constructed in \eqcref{eq:bs-dilation-family:sig:article-interpolation-raj-dahya}.
    Let $i,j\in\{1,2,\ldots,d\}$ with $i \neq j$.
    Let $s,t\in\realsNonNeg$ and set $m\coloneqq\floor{s}$ and $n\coloneqq\floor{t}$.
    Since
        $\{U_{i},U_{j}\}$,
        $\{P_{i},P_{j}\}$,
        $\{U_{j},P_{i}\}$,
        and
        $\{U_{i},P_{j}\}$
    are each commuting families,
    one may arrive at the following commutation relation for
        $T_{i}(s),T_{j}(t)$:

        \begin{longeqnarray}
            T_{j}(t)T_{i}(s)
                &\eqcrefoverset{eq:bs-dilation-family:sig:article-interpolation-raj-dahya}{=}
                    &(
                        U_{j}(t)P_{j}(t) \otimes S_{j}^{n}
                        +
                        U_{j}(t)(\onematrix - P_{j}(t)) \otimes S_{j}^{n+1}
                    )\\
                    &&\cdot
                    (
                        U_{i}(s)P_{i}(s) \otimes S_{i}^{m}
                        +
                        U_{i}(s)(\onematrix - P_{i}(s)) \otimes S_{i}^{m+1}
                    )\\
                &\eqcrefoverset{eq:1:\beweislabel}{=}
                    &U_{i}(s)
                    U_{j}(t)
                    P_{i}(s)
                    P_{j}(t)
                    \otimes
                    V_{ij}^{mn}
                    S_{i}^{m}
                    S_{j}^{n}\\
                    &&+
                    U_{i}(s)
                    U_{j}(t)
                    (\onematrix - P_{i}(s))
                    P_{j}(t)
                    \otimes
                    V_{ij}^{(m+1)n}
                    S_{i}^{m+1}
                    S_{j}^{n}\\
                    &&+
                    U_{j}(t)
                    U_{i}(s)
                    (\onematrix - P_{j}(t))
                    P_{i}(s)
                    \otimes
                    V_{ij}^{m(n+1)}
                    S_{i}^{m}
                    S_{j}^{n+1}\\
                    &&+
                    U_{i}(s)
                    U_{j}(t)
                    (\onematrix - P_{i}(s))
                    (\onematrix - P_{j}(t))
                    \otimes
                    V_{ij}^{(m+1)(n+1)}
                    S_{i}^{m+1}
                    S_{j}^{n+1}\\
                &= &(U_{i}(s)U_{j}(t) \otimes \onematrix)\\
                    &&\cdot
                    \begin{displayarray}[t]{0l}
                        \Big(
                        P_{i}(s) P_{j}(t) \otimes V_{ij}^{mn}
                        +
                        P_{i}(s) (\onematrix - P_{j}(t)) \otimes V_{ij}^{m(n+1)}\\
                        +
                        (\onematrix - P_{i}(s)) P_{j}(t) \otimes V_{ij}^{(m+1)n}\\
                        +
                        (\onematrix - P_{i}(s)) (\onematrix - P_{j}(t)) \otimes V_{ij}^{(m+1)(n+1)}
                        \Big)
                    \end{displayarray}\\
                    &&\cdot (U_{i}(s)U_{j}(t) \otimes \onematrix)^{\ast}\\
                    &&\cdot
                    (
                        U_{i}(s) P_{i}(s) \otimes S_{i}^{m}
                        +
                        U_{i}(s) (\onematrix - P_{i}(s)) \otimes S_{i}^{m+1}
                    )\\
                    &&\cdot
                    (
                        U_{j}(t) P_{j}(t) \otimes S_{j}^{n}
                        +
                        U_{j}(t) (\onematrix - P_{j}(t)) \otimes S_{j}^{n+1}
                    )\\
                &\eqcrefoverset{eq:bs-dilation-family:sig:article-interpolation-raj-dahya}{=}
                    &V_{ij}(s,t) \cdot T_{i}(s)T_{j}(t),
        \end{longeqnarray}

    \continueparagraph
    where $V_{ij}(s,t)$ is completely determined by
        $s,t,P_{i},P_{j},U_{i},U_{j},V_{ij}$.
    If
        $V_{ij}(s,t) = U(2B_{ij}st)$
    for each $i,j \in \{1,2,\ldots,d\}$ with $i \neq j$
    and each $s,t\in\realsNonNeg$,
    where $B \in M_{d}(\complex)$ is some antisymmetric matrix
    and $U$ an $\topSOT$-continuous representation of $(\reals,+,0)$
    on $\BoundedOps{L^{2}(\Torus^{d})\otimes\HilbertRaum}$,
    then the semigroups $\{T_{i}\}_{i=1}^{d}$
    satisfy the CCR in Weyl form
    (\cf \cite[Example~1.8]{Dahya2024approximation}).
    From the above expression, it is however unclear whether
        $V_{ij}(s,t)$
    can be expressed in this way.
    It would thus be nice to know if there is an appropriate alternative
    to our generalisation of the Bhat--Skeide interpolation,
    such that a family of contractions
    satisfying commutation relations such as \eqcref{eq:1:\beweislabel}
    can be interpolated to a family of $\Cnought$-semigroups satisfying the CCR.
\end{rem}

\begin{rem}
    If possible, it would be nice to find an interpolation technique similar to that of Bhat--Skeide
    (or our generalisation to commuting families)
    for the Banach space setting.
\end{rem}



\subsection[Application to the embedding problem]{Application to the embedding problem}
\label{sec:interpolation:embed:sig:article-interpolation-raj-dahya}

\firstparagraph
In this section we consider the classical \highlightTerm{embedding problem} in semigroup theory
of finding for a given bounded operator $S$ on Banach space $\BanachRaum$
a ($1$-parameter) $\Cnought$-semigroup $T$ on $\BanachRaum$
satisfying $T(1) = S$.
Not all operators are embeddable
(\exempli non-bijective Fredholm operators or non-zero nilpotent matrices).
And in the literature various results provide either
    sufficient or necessary conditions
    or characterisations
    of embeddability for certain subclasses of operators
(see \exempli \cite{Eisner2009embeddingOpSemigroup,EisnerRadl2022embed}).
To the best of our knowledge,
no complete characterisation of the embeddability of bounded operators appears to be known.


In the context of Hilbert spaces, however,
the Bhat--Skeide interpolation demonstrates that embeddability
can always be achieved \textit{up to a certain modification}.
This leads us to the following problem:

\begin{problem}
\makelabel{problem:aux-embed-dim:single:sig:article-interpolation-raj-dahya}
    Given a Hilbert space $\HilbertRaum$,
    what is the smallest dimension
    that an auxiliary Hilbert space
        $\HilbertRaum^{\prime}$
    can have,
    such that
    for each
        $S \in \BoundedOps{\HilbertRaum}$,
    the tensor product
        $
            \onematrix_{\HilbertRaum^{\prime}} \otimes S
            \in \BoundedOps{\HilbertRaum^{\prime} \otimes \HilbertRaum}
        $
    can be embedded into a $\Cnought$-semigroup
    on $\HilbertRaum^{\prime} \otimes \HilbertRaum$?
\end{problem}

Note again, as mentioned in \S{}\ref{sec:introduction:notation:sig:article-interpolation-raj-dahya}
that all Hilbert spaces are taken to be at least $1$-dimensional.
If $\HilbertRaum$ and $\HilbertRaum^{\prime}$ are Hilbert spaces
with $\dim(\HilbertRaum^{\prime})=1$,
then
    $\HilbertRaum \cong \HilbertRaum^{\prime} \otimes \HilbertRaum$
via $\xi \mapsto e \otimes \xi$,
where $e \in \HilbertRaum^{\prime}$ can be taken to be any unit vector.
From this it is easy to observe for $S \in \BoundedOps{\HilbertRaum}$
that
    $\onematrix_{\HilbertRaum^{\prime}} \otimes S$
can be embedded into a $\Cnought$-semigroup on
    $\HilbertRaum^{\prime} \otimes \HilbertRaum$
if and only if
    $S$ can be embedded into a $\Cnought$-semigroup on $\HilbertRaum$.
If $\dim(\HilbertRaum^{\prime}) > 1$,
then the \usesinglequotes{if} direction continues to holds,
whilst the \usesinglequotes{only if} direction fails in general.
Hence the above notion of embedding covers and also strictly extends
the standard definition.

\begin{prop}
\makelabel{prop:auxiliary-embed-dimension:sig:article-interpolation-raj-dahya}
    The optimal dimension in \Cref{problem:aux-embed-dim:single:sig:article-interpolation-raj-dahya}
    for any Hilbert space is $\aleph_{0}$.%
    \footref{ft:aleph-0:sig:article-interpolation-raj-dahya}
\end{prop}

\footnotetext[ft:aleph-0:sig:article-interpolation-raj-dahya]{%
    $\aleph_{0}=\card{\naturals}$ denotes the first infinite cardinal,
    or simply: the cardinality of the naturals.
}

    \begin{proof}
        Let $\HilbertRaum$ be a Hilbert space and let
            $d_{\HilbertRaum}$
        denote the optimal dimension in
        \Cref{problem:aux-embed-dim:single:sig:article-interpolation-raj-dahya}
        for $\HilbertRaum$.
        Consider an arbitrary operator
            $S \in \BoundedOps{\HilbertRaum}$.
        Since the Bhat--Skeide interpolation of $S$
        (see \Cref{lemm:bhat-skeide:single:sig:article-interpolation-raj-dahya})
        yields a $\Cnought$-semigroup
            $T$ on $L^{2}(\Torus) \otimes \HilbertRaum$
        satisfying
            $T(1) = \onematrix_{L^{2}(\Torus)} \otimes S$,
        it follows that
            $d_{\HilbertRaum}$ is bounded above by $\dim(L^{2}(\Torus)) = \aleph_{0}$.
        To complete the proof, the possibility
        of $d_{\HilbertRaum}$ being finite
        needs to be ruled out.
        So suppose to the contrary, there be a finite-dimensional
        Hilbert space $\HilbertRaum^{\prime}$
        for which
            $
                \onematrix \otimes S
                \in \BoundedOps{\HilbertRaum^{\prime} \otimes \HilbertRaum}
            $
        can be embedded into a $1$-parameter $\Cnought$-semigroup
        for each $S \in \BoundedOps{\HilbertRaum}$.

        \begin{enumerate}{\itshape {Case} 1.}
        \item $\dim(\HilbertRaum)$ is finite.
            Consider $S \coloneqq \zeromatrix \in \BoundedOps{\HilbertRaum}$.
            We show that $\onematrix_{\HilbertRaum^{\prime}} \otimes S$
            cannot be embedded into a $\Cnought$-semigroup.
            So suppose to the contrary that
                $T$ is a $\Cnought$-semigroup on $\HilbertRaum^{\prime} \otimes \HilbertRaum$
            with
                $T(1) = \onematrix_{\HilbertRaum^{\prime}} \otimes S = \zeromatrix$.
            The finite-dimensionality of
                $\HilbertRaum^{\prime} \otimes \HilbertRaum$
            implies norm-continuity of the semigroup.
            And since the subspace of invertible operators is norm-open
            and
                ${T(t) \longrightarrow \onematrix}$
            in norm for
                ${t \longrightarrow 0^{+}}$,
            it follows that $T(\frac{1}{n})$ and thus $T(1)=T(\frac{1}{n})^{n}$
            are invertible for sufficiently large $n\in\naturals$.
            This contradicts $T(1) = \zeromatrix$.

        \item $\dim(\HilbertRaum)$ is infinite.
            By infinite-dimensionality,
            there exists a non-bijective Fredholm operator
                $S \in \BoundedOps{\HilbertRaum}$,
            \idest
                $\dim(\ker(S))$
                and
                $\dim(\ker(S^{\ast})) = \dim(\quer{\ran(S)}^{\perp})$
            are finite
            with at least one these being non-zero.
            Now, one can readily verify that
                $
                    \ker(\onematrix_{\HilbertRaum^{\prime}} \otimes S)
                    = \HilbertRaum^{\prime} \otimes \ker(S)
                $
            and thus
                $
                    \dim(\ker(\onematrix \otimes S))
                    = \dim(\HilbertRaum^{\prime})\dim(\ker(S))
                $.
            Similarly,
                $
                    \dim(\ker((\onematrix_{\HilbertRaum^{\prime}} \otimes S)^{\ast}))
                    = \dim(\ker(\onematrix_{\HilbertRaum^{\prime}} \otimes S^{\ast}))
                    = \dim(\HilbertRaum^{\prime})\dim(\ker(S^{\ast}))
                $.
            From these expressions and the assumptions on $S$ and the auxiliary space,
            both
                $\dim(\ker(\onematrix_{\HilbertRaum^{\prime}} \otimes S))$
                and
                $\dim(\quer{\ran(\onematrix_{\HilbertRaum^{\prime}} \otimes S)}^{\perp})$
            are finite,
            with at least one of these being non-zero.
            Thus $\onematrix_{\HilbertRaum^{\prime}} \otimes S$
            is a non-bijective Fredholm operator
            and thus by \cite[Corollary~3.2]{Eisner2009embeddingOpSemigroup}
            cannot be embedded into a $\Cnought$-semigroup.
        \end{enumerate}

        In both cases, one sees that no finite-dimensional auxiliary space suffices.
        Hence $d_{\HilbertRaum}$ must be at least $\aleph_{0}$.
    \end{proof}

\begin{rem}
    \makelabel{rem:alternative-proof-aux-dim-not-finite:sig:article-interpolation-raj-dahya}
    We thank the anonymous referee for the suggested
    argument to prove \textit{Case 1} above,
    which shortened our original algebraic approach.
    More generally, note that no non-zero nilpotent operator
    on a finite-dimensional Hilbert
    space can be embedded into a $\Cnought$-semigroup
    (see \exempli \cite[Theorem~3.1]{Eisner2009embeddingOpSemigroup}).
\end{rem}



The embedding problem can also readily be extended to families of operators.
In light of out generalisation of the Bhat--Skeide interpolation,
we consider the following:

\begin{problem}
\makelabel{problem:aux-embed-dim:family:sig:article-interpolation-raj-dahya}
    Let $\HilbertRaum$ be a Hilbert space and $\mathcal{I}$ a non-empty index set.
    What is the smallest dimension
    that an auxiliary Hilbert space
        $\HilbertRaum^{\prime}$,
    can have,
    such that
    for each commuting family
        $\{S_{i}\}_{i\in\mathcal{I}} \subseteq \BoundedOps{\HilbertRaum}$
    of bounded operators,
    the commuting family
        $
            \{\onematrix_{\HilbertRaum^{\prime}} \otimes S_{i}\}_{i\in\mathcal{I}}
            \subseteq \BoundedOps{\HilbertRaum^{\prime} \otimes \HilbertRaum}
        $
    can be \highlightTerm{simultaneously embedded} into a commuting family
        $\{T_{i}\}_{i\in\mathcal{I}}$
        of $\Cnought$-semigroups
        on $\HilbertRaum^{\prime} \otimes \HilbertRaum$,
    in the sense that
        $T_{i}(1) = \onematrix_{\HilbertRaum^{\prime}} \otimes S_{i}$
    for $i\in\mathcal{I}$?
\end{problem}

Let $\kappa$ denote this dimension.
\Cref{lemm:bhat-skeide:multi:sig:article-interpolation-raj-dahya}
yields that the $\kappa$ is bounded above by
    $\kappa_{\max} \coloneqq \dim(L^{2}(\prod_{i\in\mathcal{I}}\Torus))$.
The arguments in the proof of \Cref{prop:auxiliary-embed-dimension:sig:article-interpolation-raj-dahya}
also yield
    $\kappa \geq \aleph_{0}$.
If $\mathcal{I}$ is countable%
\footref{ft:prose:countable:sig:article-interpolation-raj-dahya}
then $\aleph_{0} \leq \kappa \leq \kappa_{\max}=\aleph_{0}$,
whence the choice of auxiliary space
in the Bhat--Skeide interpolation is optimal.
If $\mathcal{I}$ is uncountable, then $\kappa_{\max} > \aleph_{0}$.%
\footref{ft:prose:dim-of-L^2-prod:sig:article-interpolation-raj-dahya}
It would be nice to know if our generalisation
of the Bhat--Skeide interpolation
continues to provide the optimal solution for embedding
in this case.

\footnotetext[ft:prose:countable:sig:article-interpolation-raj-dahya]{%
    \idest finite or countably infinite
}

\footnotetext[ft:prose:dim-of-L^2-prod:sig:article-interpolation-raj-dahya]{%
    Using elementary cardinal arithmetic
    one can readily obtain that
    $
        \dim(L^{2}(\prod_{i\in\mathcal{I}}\Torus))
        =\max\{\card{\mathcal{I}},\aleph_{0}\}
    $.
    So in particular, $\kappa_{\max} > \aleph_{0}$
    if $\mathcal{I}$ is an uncountable index set.
}




Finally, consider the space of $d$-tuples of commuting contractions for some $d\in\naturals$.
In the above problem we determined that
in the optimal \emph{worst-case scenario},
the dimension of the auxiliary space
needed for simultaneous embeddings is $\aleph_{0}$.
In the \emph{best-case scenario}, this dimension is $1$,
which holds just in case a tuple can itself be simultaneously embedded
(\cf the discussion at the start of this section).
We can consider the subspace of such tuples and ask what topological properties it enjoys.

\begin{prop}
\makelabel{prop:sim-embeddable-tuples-pw-dense:sig:article-interpolation-raj-dahya}
    Let $d\in\naturals$ and $\HilbertRaum$ be an infinite-dimensional Hilbert space.
    Let $E_{d}$
    denote the set of all $d$-tuples of commuting contractions
    which can be simultaneously embedded into a commuting family
    of $\Cnought$-semigroups on $\HilbertRaum$.
    Then $E_{d}$ is dense in $(\OpSpaceC{\HilbertRaum}^{d}_{\textup{comm}},\topPW)$.
\end{prop}

    \begin{proof}
        Let
            $\{S_{i}\}_{i=1}^{d} \in \OpSpaceC{\HilbertRaum}^{d}_{\textup{comm}}$
        be arbitrary.
        By our generalisation of the Bhat--Skeide Interpolation,
        we know that the commuting $d$-tuple
            $
                \{\onematrix \otimes S_{i}\}_{i=1}^{d}
                \in
                \OpSpaceC{L^{2}(\Torus^{d}) \otimes \HilbertRaum}^{d}
            $
        can be simultaneously embedded.
        And for any unitary operator
            ${w \colon L^{2}(\Torus^{d}) \otimes \HilbertRaum \to \HilbertRaum}$
        one readily sees that
            $
                \{\adjoint_{w}(\onematrix \otimes S_{i})\}_{i=1}^{d}
                \in
                \OpSpaceC{\HilbertRaum}^{d}
            $
        is a commuting $d$-tuple
        that can also be simultaneously embedded.
        Hence

            \begin{restoremargins}
            \begin{equation}
            \label{eq:special-embeddable-tuples:sig:article-interpolation-raj-dahya}
                \begin{displayarray}[m]{rcl}
                    E_{d} \supseteq \Big\{
                        \{\adjoint_{w}(\onematrix \otimes S_{i})\}_{i=1}^{d}
                    &\mid
                        &\{S_{i}\}_{i=1}^{d} \in \OpSpaceC{\HilbertRaum}^{d}_{\textup{comm}},\\
                        &&{w \colon L^{2}(\Torus^{d}) \otimes \HilbertRaum \to \HilbertRaum}
                        ~\text{unitary}
                    \Big\}
                    \eqqcolon \tilde{E}_{d}\\
                \end{displayarray}
            \end{equation}
            \end{restoremargins}

        \continueparagraph
        and it suffices to show that
            $\tilde{E}_{d}$
            is dense in
            $(\OpSpaceC{\HilbertRaum}^{d}_{\textup{comm}}, \topPW)$.
        To this end consider an arbitrary
            $\{S_{i}\}_{i=1}^{d} \in \OpSpaceC{\HilbertRaum}^{d}_{\textup{comm}}$
            and
            finite $F \subseteq \HilbertRaum$.
        Since $\dim(\HilbertRaum)$ is infinite,
        there exists a unitary operator
            ${w \colon L^{2}(\Torus^{d}) \otimes \HilbertRaum \to \HilbertRaum}$
        satisfying
            $
                w\,(\einser \otimes \xi)
                = \xi
            $
        for each $\xi \in F$.%
        \footref{ft:pf:unitary-equivalence-L^2-tensor-H:sig:article-interpolation-raj-dahya}
        For
            $\mathbf{n} \in \naturalsZero^{d}$
            and $\xi,\eta \in F$
        one obtains

            \begin{longeqnarray}
                \brktLong{
                    \Big(
                    \prod_{i=1}^{d}
                        \Big(
                            \adjoint_{w}
                            (
                                \onematrix
                                \otimes
                                S_{i}
                            )
                        \Big)^{n_{i}}
                    \Big)
                    \:\xi
                }{
                    \eta
                }
                &= &\brktLong{
                        \Big(
                            \onematrix
                            \otimes
                            \prod_{i=1}^{d}
                                S_{i}^{n_{i}}
                        \Big)
                        \:w^{\ast}\xi
                    }{
                        w^{\ast}\eta
                    }\\
                &= &\brktLong{
                        \Big(
                            \onematrix
                            \otimes
                            \prod_{i=1}^{d}
                                S_{i}^{n_{i}}
                        \Big)
                        \:(\einser \otimes \xi)
                    }{
                        \einser \otimes \eta
                    }\\
                &= &\underbrace{
                        \brkt{\einser}{\einser}
                    }_{=1}
                    \cdot
                    \brktLong{
                        \Big(
                            \prod_{i=1}^{d}
                                S_{i}^{n_{i}}
                        \Big)
                        \xi
                    }{\eta}.\\
            \end{longeqnarray}

        \continueparagraph
        Since $\{S_{i}\}_{i=1}^{d}$ was arbitrarily chosen,
        the density of $\tilde{E}_{d}$ in
            $(\OpSpaceC{\HilbertRaum}^{d}_{\textup{comm}}, \topPW)$
        follows.
        And by \eqcref{eq:special-embeddable-tuples:sig:article-interpolation-raj-dahya}
        the claim holds.
    \end{proof}

    \footnotetext[ft:pf:unitary-equivalence-L^2-tensor-H:sig:article-interpolation-raj-dahya]{%
        in particular,
            $L^{2}(\Torus^{d}) \otimes \HilbertRaum$ and $\HilbertRaum$
        are unitarily equivalent,
        since by elementary arithmetic
        of transfinite cardinals,
        setting $\kappa \coloneqq \dim(\HilbertRaum)$
        one has
        $
            \dim(L^{2}(\Torus^{d}) \otimes \HilbertRaum)
            = \dim(L^{2}(\Torus^{d})) \cdot \dim(\HilbertRaum)
            = \aleph_{0} \cdot \kappa
            = \max\{\aleph_{0}, \kappa\}
            = \kappa
            = \dim(\HilbertRaum)
        $.
    }




\subsection[Numerical applications]{Numerical applications}
\label{sec:interpolation:application:sig:article-interpolation-raj-dahya}

\firstparagraph
In this section we demonstrate
how the Bhat--Skeide interpolation can be exploited
for the purposes of time-discretisation
of multi-parameter $\Cnought$-semigroups on Hilbert spaces.


\subsubsection[Interpolation under time-discretisation]{Interpolation under time-discretisation}
\label{sec:interpolation:application:discrete:sig:article-interpolation-raj-dahya}

\firstparagraph
Consider the following scenario:

\begin{problem}
\label{problem:signal-recovery:sig:article-interpolation-raj-dahya}
    Suppose a classical dynamical system is governed by an unknown commuting family
        $\{\breve{T}_{i}\}_{i=1}^{d}$
    of $d$ contractive $\Cnought$-semigroups on a Hilbert space $\HilbertRaum$.
    We are provided with full information about the values
        $\{S_{i} = \breve{T}_{i}(1)\}_{i=1}^{d}$.
    We seek a procedure to construct a commuting family
        $\{T_{i}\}_{i=1}^{d}$ of $\Cnought$-semigroups
    (on a finite-dimensional Hilbert space, if $\dim(\HilbertRaum) < \infty$),
    which recovers
        $\{\breve{T}_{i}\}_{i=1}^{d}$
    as much as possible based on the given information.
    The semigroups should be definable on arbitrarily fine
    discretisations of time points.
\end{problem}

Since $\{S_{i}\}_{i=1}^{d}$ is clearly a commuting family of contractions,
as a first attempt, one could apply
    \Cref{lemm:bhat-skeide:multi:sig:article-interpolation-raj-dahya}
to obtain an interpolation
    $\{T_{i}\}_{i=1}^{d}$
of
    $\{S_{i}\}_{i=1}^{d}$
on the Hilbert space $L^{2}(\Torus^{d}) \otimes \HilbertRaum$.
Doing so, however, yields a semigroup on an infinite-dimensional space.
Though, by discretising time points, one can control the growth of the
dimension of the Hilbert space.
For $N\in\naturals$ let
    $
        \Torus_{N}
        \coloneqq
        \{
            e^{\iunit 2\pi t}
            \mid
            t \in \frac{1}{N}\integers
        \}
        =
        \{
            e^{\iunit 2\pi k/N}
            \mid
            k \in \{0,1,\ldots,N\}
        \}
    $
and let
    $\{
        \delta_{\mathbf{z}}
        \mid
        \mathbf{z} \in \Torus_{N}^{d}
    \}$
be the canonical ONB for $\ell^{2}(\Torus_{N}^{d})$.
The following two propositions provide a solution to \Cref{problem:signal-recovery:sig:article-interpolation-raj-dahya}:

\begin{prop}
\makelabel{prop:discretised-interpolation:sig:article-interpolation-raj-dahya}
    Let $N,d\in\naturals$
    and $\{\breve{T}_{i}\}_{i=1}^{d}$
    be a commuting family of $\Cnought$-semigroups on a Hilbert space $\HilbertRaum$.
    Let
        $\{T_{i}\}_{i=1}^{d}$
    be the generalised Bhat--Skeide interpolation
    of $\{S_{i} \coloneqq \breve{T}_{i}(1)\}_{i=1}^{d}$
    as in \Cref{lemm:bhat-skeide:multi:sig:article-interpolation-raj-dahya}.
    There exists a contractive homomorphism
        $T^{(N)}$
    over the discrete monoid
        $((\frac{1}{N}\naturalsZero)^{d},+,\zerovector)$
    on $\ell^{2}(\Torus_{N}^{d}) \otimes \HilbertRaum$
    satisfying

        \begin{restoremargins}
        \begin{equation}
        \label{eq:discr-interpolation-dilation:sig:article-interpolation-raj-dahya}
            (r_{N} \otimes \onematrix) \: T^{(N)}(\mathbf{t})
                = \Big(
                \prod_{i=1}^{d}
                    T_{i}(t_{i})
                \Big)
                \:(r_{N} \otimes \onematrix)
        \end{equation}
        \end{restoremargins}

    \continueparagraph
    for all $\mathbf{t} \in (\frac{1}{N}\integers)^{d}$,
    where
        $r_{N} \in \BoundedOps{\ell^{2}(\Torus_{N}^{d})}{L^{2}(\Torus^{d})}$
    is a canonically defined isometry.
    In particular
        $
            T^{(N)}(0,0,\ldots,\underset{i}{n},\ldots,0)
            = \onematrix \otimes \breve{T}_{i}(n)
        $
    for $n\in\naturalsZero$ and $i\in\{1,2,\ldots,d\}$.
\end{prop}

    \begin{proof}
        We use the constructions of $P_{i},U_{i}$, $i\in\{1,2,\ldots,d\}$
        as in the proof of \Cref{lemm:bhat-skeide:multi:sig:article-interpolation-raj-dahya}.
        Set
            $
                u_{h}
                    \coloneqq
                    \frac{1}{\sqrt{h^{d}}}\einser_{\{e^{\imath 2\pi s} \mid s\in[0,\:h)\}^{d}}
            $
        for $h\in(0,\:1)$.
        Defining
            ${
                r_{N} \colon \ell^{2}(\Torus_{N}^{d}) \to L^{2}(\Torus^{d})
            }$
        via

            \begin{displaymath}
                r_{N} f
                    \coloneqq
                        \sum_{\mathbf{t} \in (\frac{1}{N}\{0,1,\ldots,N-1\})^{d}}
                            \brkt{f}{\delta_{(e^{\iunit 2\pi t_{i}})_{i=1}^{d}}}
                            \underbrace{
                                \Big(
                                    \prod_{i=1}^{d}
                                        U_{i}(t_{i})
                                \Big)
                                \:u_{1/N}
                            }_{
                                = \sqrt{N}^{d}
                                    \einser_{
                                        \prod_{i=1}^{d}
                                            \{
                                                e^{\imath 2\pi s}
                                                \mid
                                                s\in[t_{i},\:t_{i} + 1/N)
                                            \}
                                    }
                            }
            \end{displaymath}

        \continueparagraph
        for $f \in \ell^{2}(\Torus_{N}^{d})$,
        one can readily see that $r_{N}$ is an isometry.
        Moreover
            $
                \ran(r_{N})
                = \linspann(\{
                    (
                        \prod_{i=1}^{d}
                        U_{i}(t_{i})
                    )
                    \:u_{1/N}
                    \mid
                    \mathbf{t}\in(\frac{1}{N}\integers)^{d}
                \})
            $
        is clearly closed under $U_{i}(\tfrac{k}{N})$
        for $k\in\integers$, $i\in\{1,2,\ldots,d\}$.
        Clearly,
            $P_{i}(\frac{k}{N})u_{1/N} = u_{1/N}$,
        and relying on this observation as well as
        the relations in
            \eqcref{eq:bscr:sig:article-interpolation-raj-dahya},
        one readily derives that
            $\ran(r_{N})$
        is closed under $P_{i}(\frac{k}{N})$
        for $k\in\integers$, $i\in\{1,2,\ldots,d\}$.

        By construction of the interpolants
        (see \eqcref{eq:bs-dilation-family:sig:article-interpolation-raj-dahya}),
        it follows that
            $\ran(r_{N}) \otimes \HilbertRaum$
        is closed under $T_{i}(t)$
        for $t\in\frac{1}{N}\naturalsZero$, $i\in\{1,2,\ldots,d\}$.
        Equivalently stated, it holds that
            $
                \Big(
                    \prod_{i=1}^{d}
                        T_{i}(t_{i})
                \Big)
                \: (r_{N} \otimes \onematrix)
                = (r_{N} \otimes \onematrix) \: T^{(N)}(\mathbf{t})
            $
        for all $\mathbf{t} \in (\frac{1}{N}\naturalsZero)^{d}$,
        and for some function
            ${
                T^{(N)}
                    \colon
                    (\frac{1}{N}\naturalsZero)^{d}
                    \to
                    \BoundedOps{\ell^{2}(\Torus_{N}^{d}) \otimes \HilbertRaum}
            }$.
        Applying this expression repeatedly, one can verify that $T^{(N)}$ is a homomorphism.
        The interpolation properties of the $T_{i}$
        further yields
            $
                T^{(N)}(0,0,\ldots,\underset{i}{n},\ldots,0)
                = (r_{N} \otimes \onematrix)^{\ast}T_{i}(n)(r_{N} \otimes \onematrix)
                = (r_{N} \otimes \onematrix)^{\ast}(\onematrix \otimes S_{i}^{n})(r_{N} \otimes \onematrix)
                = \onematrix \otimes S_{i}^{n}
                = \onematrix \otimes \breve{T}_{i}(n)
            $
        for $n\in\naturalsZero$, $i\in\{1,2,\ldots,d\}$.
    \end{proof}

\begin{prop}
\makelabel{prop:discretised-interpolation-explicit:sig:article-interpolation-raj-dahya}
    The discretisation $T^{(N)}$ in \Cref{prop:discretised-interpolation:sig:article-interpolation-raj-dahya}
    is concretely given by

        \begin{restoremargins}
        \begin{equation}
        \label{eq:discr-interpolation:sig:article-interpolation-raj-dahya}
            T^{(N)}(\mathbf{t})\:(f \otimes \xi)
            =   \sum_{\mathbf{t}^{\prime} \in (\frac{1}{N}\{0,1,\ldots,N-1\})^{d}}
                    \brkt{f}{\delta_{(e^{\iunit 2\pi t^{\prime}_{i}})_{i=1}^{d}}}
                    \delta_{(e^{\iunit 2\pi (t_{i} + t^{\prime}_{i})})_{i=1}^{d}}
                    \otimes
                    \Big(
                        \prod_{i=1}^{d}
                            S_{i}^{\kappa(t_{i},t^{\prime}_{i})}
                    \Big)
                    \xi
        \end{equation}
        \end{restoremargins}

    \continueparagraph
    for
        $\mathbf{t}\in(\frac{1}{N}\naturalsZero)^{d}$,
        $f \in \ell^{2}(\Torus_{N}^{d})$,
        and $\xi \in \HilbertRaum$,
    where
        $\kappa(t,t^{\prime}) \coloneqq \floor{t} + \einser_{\fractional{t} + \fractional{t^{\prime}} \geq 1}$
    for $t,t^{\prime}\in\reals$.
\end{prop}

    \begin{proof}
        Without loss of generality, assume
            $
                f = \delta_{
                        (e^{\iunit 2\pi t^{\prime}_{i}})_{i=1}^{d}
                    }
                = r_{N}^{\ast}\Big(\prod_{i=1}^{d}U_{i}(t^{\prime}_{i})\Big)\:u_{1/N}
            $
        for some $\mathbf{t}^{\prime} \in (\frac{1}{N}\{0,1,\ldots,N-1\})^{d}$.
        For each $i\in\{1,2,\ldots,d\}$
        by \eqcref{eq:bscr:sig:article-interpolation-raj-dahya}
        one has
            $
                P_{i}(t_{i})U_{i}(t^{\prime}_{i})
                = U_{i}(t^{\prime}_{i})Q_{i}(t_{i},t^{\prime}_{i})
            $,
        where
            $Q_{i}(t_{i},t^{\prime}_{i}) = \onematrix - (P_{i}(t^{\prime}_{i}) - P(t^{\prime}_{i} + t_{i}))$
        if $\fractional{t_{i}} + \fractional{t^{\prime}_{i}}  < 1$,
        otherwise $Q_{i}(t_{i},t^{\prime}_{i}) = P_{i}(t^{\prime}_{i}) - P(t^{\prime}_{i} + t_{i})$.
        Observe that
            $P_{i}(\frac{k}{N})u_{1/N} = u_{1/N}$
        for $k\in\integers$.
        So
            $Q_{i}(t_{i},t^{\prime}_{i})u_{1/N} = u_{1/N}$
        if $\fractional{t_{i}} + \fractional{t^{\prime}_{i}} < 1$,
        otherwise
            $Q_{i}(t_{i},t^{\prime}_{i})u_{1/N} = \zerovector$.
        Since $U_{i},P_{j}$ (and thus $U_{i},Q_{j}$)
        commute for $i \neq j$
        induction yields

            \begin{longeqnarray}
                &&(r_{N} \otimes \onematrix)\:T^{(N)}(\mathbf{t})\:(f \otimes \xi)\\
                &\eqcrefoverset{eq:discr-interpolation-dilation:sig:article-interpolation-raj-dahya}{=}
                    &\Big(
                        \prod_{i=1}^{d}
                            T_{i}(t_{i})
                    \Big)
                    \:(r_{N} \otimes \onematrix)
                    \:(f \otimes \xi)\\
                &= &\prod_{i=1}^{d}
                        \begin{displayarray}[t]{0l}
                        \Big(
                            U_{i}(t_{i})P_{i}(t_{i}) \otimes S_{i}^{\floor{t_{i}}}\\
                            \:+\: U_{i}(t_{i})(\onematrix - P_{i}(t_{i}))
                                \otimes S_{i}^{\floor{t_{i}} + 1}
                        \Big)
                        \:\Big(\Big(\prod_{i=1}^{d}U_{i}(t^{\prime}_{i})\Big)u_{1/N} \otimes \xi\Big)
                        \end{displayarray}\\
                &\textoverset{ind}{=}
                    &\prod_{i=1}^{d}
                        \begin{displayarray}[t]{0l}
                        \Big(
                            U_{i}(t_{i})P_{i}(t_{i})U_{i}(t^{\prime}_{i}) \otimes S_{i}^{\floor{t_{i}}}\\
                            \:+\: U_{i}(t_{i})(\onematrix - P_{i}(t_{i}))U_{i}(t^{\prime}_{i})
                                \otimes S_{i}^{\floor{t_{i}} + 1}
                        \Big)
                        \:(u_{1/N} \otimes \xi)
                        \end{displayarray}\\
                &\eqcrefoverset{eq:bscr:sig:article-interpolation-raj-dahya}{=}
                    &\prod_{i=1}^{d}
                        \begin{displayarray}[t]{0l}
                        \Big(
                            U_{i}(t_{i} + t^{\prime}_{i})Q_{i}(t_{i},t^{\prime}_{i}) \otimes S_{i}^{\floor{t_{i}}}\\
                            \:+\:U_{i}(t_{i} + t^{\prime}_{i})(\onematrix - Q_{i}(t_{i},t^{\prime}_{i}))
                                \otimes S_{i}^{\floor{t_{i}} + 1}
                        \Big)
                        \:(u_{1/N} \otimes \xi)
                        \end{displayarray}\\
                &\textoverset{ind}{=}
                    &\Big(
                        \prod_{i=1}^{d}
                            \begin{cases}
                                U_{i}(t_{i} + t^{\prime}_{i}) \otimes S_{i}^{\floor{t_{i}}}
                                    &\colon &\fractional{t_{i}} + \fractional{t^{\prime}_{i}} < 1\\
                                U_{i}(t_{i} + t^{\prime}_{i}) \otimes S_{i}^{\floor{t_{i}} + 1}
                                    &\colon &\fractional{t_{i}} + \fractional{t^{\prime}_{i}} \geq 1\\
                            \end{cases}
                    \Big)
                    \:(u_{1/N} \otimes \xi).\\
            \end{longeqnarray}

        \continueparagraph
        Thus

            \begin{longeqnarray}
                T^{(N)}(\mathbf{t})\:(f \otimes \xi)
                = r_{N}^{\ast}\Big(\prod_{i=1}^{d}U_{i}(t_{i} + t^{\prime}_{i})\Big)\:u_{1/N}
                    \otimes
                    \Big(\prod_{i=1}^{d}S_{i}^{\kappa(t_{i},t^{\prime}_{i})}\Big)\:\xi,\\
            \end{longeqnarray}

        \continueparagraph
        and since the isometry satisfies
            $
                r_{N}\delta_{(e^{\iunit 2\pi(t_{i} + t^{\prime}_{i})})_{i=1}^{d}}
                = \Big(\prod_{i=1}^{d}U_{i}(t_{i} + t^{\prime}_{i})\Big)\:u_{1/N}
            $,
        \eqcref{eq:discr-interpolation:sig:article-interpolation-raj-dahya} holds.
    \end{proof}

Note that if $\dim(\HilbertRaum)$ is finite,
then the dimension of the Hilbert space,
on which the discretisation $T^{(N)}$ is defined,
satisfies
    $
        \dim(\ell^{2}(\Torus_{N}^{d}) \otimes \HilbertRaum)
        = N^{d}\dim(\HilbertRaum)
        < \infty
    $.
Furthermore, the expression in \eqcref{eq:discr-interpolation-dilation:sig:article-interpolation-raj-dahya}
entails that the generalised Bhat--Skeide interpolation
is a common contraction-valued dilation of its discretisations.

Finally, it is worth noting that
    the generalised Bhat--Skeide interpolation
    and
    its discretisations
are contraction-valued dilations of
intuitively defined piecewise multi-linear
extensions of the $d$-parameter discrete semigroup
    ${S \colon \naturalsZero^{d} \to \BoundedOps{\HilbertRaum}}$
corresponding to $\{S_{i}\}_{i=1}^{d}$
(\cf \S{}\ref{sec:introduction:definitions:sig:article-interpolation-raj-dahya}).
The following result demonstrates this precisely:

\begin{prop}
\makelabel{prop:multilinear-interpolation:sig:article-interpolation-raj-dahya}
    The Bhat--Skeide interpolation $\{T_{i}\}_{i=1}^{d}$
    of $\{S_{i}\}_{i=1}^{d}$ satisfies

        \begin{restoremargins}
        \begin{equation}
        \label{eq:multilinear-interpolation-bhat-skeide:sig:article-interpolation-raj-dahya}
            v^{\ast}
            \:\Big(\prod_{i=1}^{N}T_{i}(t_{i})\Big)
            \:v
                = \prod_{i=1}^{d}
                    \Big(
                        (1-\fractional{t_{i}})
                        S_{i}^{\floor{t_{i}}}
                        +
                        \fractional{t_{i}}
                        S_{i}^{\floor{t_{i}} + 1}
                    \Big)
        \end{equation}
        \end{restoremargins}

    \continueparagraph
    for all $\mathbf{t}\in\realsNonNeg^{d}$,
    where $v \in \BoundedOps{\HilbertRaum}{L^{2}(\Torus^{d}) \otimes \HilbertRaum}$
    is the isometric embedding defined by
        $v \xi = \einser \otimes \xi$
    for $\xi\in\HilbertRaum$,
    where $\einser$ is the constant function equal to $1$ everywhere on $\Torus^{d}$.

    Similarly, for each $N\in\naturals$,
    the discretisation $T^{(N)}$ in \Cref{prop:discretised-interpolation:sig:article-interpolation-raj-dahya}
    satisfies

        \begin{restoremargins}
        \begin{equation}
        \label{eq:multilinear-interpolation-discr:sig:article-interpolation-raj-dahya}
            v_{N}^{\ast}\:T^{(N)}(\mathbf{t})\:v_{N}
                =
                    \prod_{i=1}^{d}
                        \Big(
                            (1-\fractional{t_{i}})
                            S_{i}^{\floor{t_{i}}}
                            +
                            \fractional{t_{i}}
                            S_{i}^{\floor{t_{i}} + 1}
                        \Big)
        \end{equation}
        \end{restoremargins}

    \continueparagraph
    for all $\mathbf{t}\in(\frac{1}{N}\naturalsZero)^{d}$,
    where $v_{N} \in \BoundedOps{\HilbertRaum}{\ell^{2}(\Torus_{N}^{d}) \otimes \HilbertRaum}$
    is the isometric embedding defined by
        $v_{N}\xi = \frac{1}{\sqrt{N^{d}}}\onevector \otimes \xi$
    for $\xi\in\HilbertRaum$,
    where
        $\onevector \coloneqq \sum_{\mathbf{z} \in \Torus_{N}^{d}}\delta_{\mathbf{z}}$.
\end{prop}

    \begin{proof}
        We first handle the discretised case.
        Let
            $N\in\naturals$,
            $\mathbf{t}\in(\frac{1}{N}\naturalsZero)^{d}$,
            and
            $\xi,\eta \in \HilbertRaum$
        be arbitrary.
        Since
            $\brkt{\onevector}{\delta_{\mathbf{z}}} = 1$
        for $\mathbf{z} \in \Torus_{N}^{d}$,
        \eqcref{eq:discr-interpolation:sig:article-interpolation-raj-dahya}
        yields
            $
                T^{(N)}(\mathbf{t})\:(\onevector \otimes \xi)
                = \sum_{\mathbf{t}^{\prime} \in (\frac{1}{N}\{0,1,\ldots,N-1\})^{d}}
                    \delta_{
                        (
                            e^{\iunit 2\pi (t_{i} + t^{\prime}_{i})}
                        )_{i=1}^{d}
                    }
                    \otimes
                    \Big(
                        \prod_{i=1}^{d}
                            S_{i}^{\kappa(t_{i},t^{\prime}_{i})}
                    \Big)
                    \:\xi
            $.
        For $i\in\{1,2,\ldots,d\}$
        the set of $t_{i}^{\prime} \in \frac{1}{N}\{0,1,\ldots,N-1\}$
        for which
            $\fractional{t_{i}^{\prime}} + \fractional{t_{i}} < 1$
        is given by
            $\frac{1}{N}\{0,1,\ldots,(1-\fractional{t_{i}}) \cdot N - 1\}$,
        which has exactly $(1 - \fractional{t_{i}}) \cdot N$ elements.%
        \footref{ft:pf:1:\beweislabel}
        Thus

            \begin{longeqnarray}
                &&\brkt{
                    v_{N}^{\ast}
                    \:T^{(N)}(\mathbf{t})
                    \:v_{N}\xi
                }{
                    \eta
                }\\
                    &= &\brkt{
                            T^{(N)}(\mathbf{t})\:(\frac{1}{\sqrt{N^{d}}}\onevector \otimes \xi)
                        }{
                            \frac{1}{\sqrt{N^{d}}}\onevector \otimes \eta
                        }\\
                    &= &\frac{1}{N^{d}}
                        \sum_{\mathbf{t}^{\prime} \in (\frac{1}{N}\{0,1,\ldots,N-1\})^{d}}
                            \underbrace{
                                \brkt{
                                    \delta_{
                                        (
                                            e^{\iunit 2\pi (t_{i} + t^{\prime}_{i})}
                                        )_{i=1}^{d}
                                    }
                                }{\onevector}
                            }_{=1}
                            \brktLong{
                                \Big(
                                    \prod_{i=1}^{d}
                                        S_{i}^{\kappa(t_{i},t^{\prime}_{i})}
                                \Big)
                                \:\xi
                            }{\eta}\\
                    &= &\brktLong{
                            \prod_{i=1}^{d}
                            \frac{1}{N}
                            \Big(
                                \sum_{t_{i}^{\prime} \in \frac{1}{N}\{0,1,\ldots,N-1\}}
                                    S_{i}^{\kappa(t_{i},t^{\prime}_{i})}
                            \Big)
                            \:\xi
                        }{\eta}\\
                    &= &\brktLong{
                        \prod_{i=1}^{d}
                            \Big(
                                \tfrac{N(1-\fractional{t_{i}})}{N}
                                S_{i}^{\floor{t_{i}}}
                                +
                                \tfrac{N\fractional{t_{i}}}{N}
                                S_{i}^{\floor{t_{i}} + 1}
                            \Big)
                        \:\xi
                    }{\eta},
            \end{longeqnarray}

        \continueparagraph
        where the final expression holds by the above counting argument.
        Since this holds for all $\xi,\eta\in\HilbertRaum$,
        \eqcref{eq:multilinear-interpolation-discr:sig:article-interpolation-raj-dahya} holds.

        Observe that
            $
                (r_{N} \otimes \onematrix)\:v_{N}\xi
                = \frac{1}{\sqrt{N}^{d}}r_{N} \onevector \otimes \xi
                = \einser \otimes \xi
                = v\xi
            $
        for each $\xi\in\HilbertRaum$, $N\in\naturals$.
        Thus

            \begin{longeqnarray}
                v^{\ast}\:\Big(\prod_{i=1}^{d}T_{i}(t_{i})\Big)\:v
                    &= &v_{N}^{\ast}(r_{N} \otimes \onematrix)^{\ast}
                        \:\Big(\prod_{i=1}^{d}T_{i}(t_{i})\Big)
                        \:(r_{N} \otimes \onematrix)
                        v_{N}\\
                    &\eqcrefoverset{eq:discr-interpolation-dilation:sig:article-interpolation-raj-dahya}{=}
                        &v_{N}^{\ast}
                        \:T^{(N)}(\mathbf{t})
                        \:v_{N}\\
                    &\eqcrefoverset{eq:multilinear-interpolation-discr:sig:article-interpolation-raj-dahya}{=}
                        &\Big(
                            (1-\fractional{t_{i}})
                            S_{i}^{\floor{t_{i}}}
                            +
                            \fractional{t_{i}}
                            S_{i}^{\floor{t_{i}} + 1}
                        \Big),
            \end{longeqnarray}

        \continueparagraph
        for each
            $\mathbf{t} \in (\frac{1}{N}\naturalsZero)^{d}$,
            $N\in\naturals$.
        Hence \eqcref{eq:multilinear-interpolation-bhat-skeide:sig:article-interpolation-raj-dahya}
        holds for all $\mathbf{t} \in (\rationals \cap [0,\:\infty))^{d}$.
        By right-$\topSOT$-continuity of both sides of this expression
        (in $t_{i}$ for each $i$)
        the validity of
            \eqcref{eq:multilinear-interpolation-bhat-skeide:sig:article-interpolation-raj-dahya}
        extends to all $\mathbf{t}\in\realsNonNeg^{d}$.
    \end{proof}

\footnotetext[ft:pf:1:\beweislabel]{%
    Since $t_{i} \in \frac{1}{N}\naturalsZero$,
    one has
        $\fractional{t_{i}} \in \{0,\tfrac{1}{N},\tfrac{2}{N},\ldots,1-\tfrac{1}{N}\}$
    and thus
        $(1-\fractional{t_{i}}) \cdot N \in \{1,2,\ldots,N\}$.
}

\begin{rem}
    Let $\HilbertRaum$ be a Hilbert space and $d,N\in\naturals$.
    By \Cref{prop:discretised-interpolation:sig:article-interpolation-raj-dahya},
    the auxiliary space
        $\HilbertRaum^{\prime} = \ell^{2}(\Torus_{N}^{d})$,
    which has dimension $N^{d}$,
    is sufficient to guarantee for every commuting family
        $\{S_{i}\}_{i=1}^{d} \subseteq \BoundedOps{\HilbertRaum}$
    that
        $\{\onematrix \otimes S_{i}\}_{i=1}^{d} \subseteq \BoundedOps{\HilbertRaum^{\prime} \otimes \HilbertRaum}$
    can be simultaneously embedded into a commuting family
        $\{T_{i}\}_{i=1}^{d}$
    of discrete semigroups
        on $\HilbertRaum^{\prime} \otimes \HilbertRaum$
        over $(\frac{1}{N}\naturalsZero,+,0)$.
    Analogous to \S{}\ref{sec:interpolation:embed:sig:article-interpolation-raj-dahya},
    it would be useful to know whether the dimension of such a space is optimal,
    at least in the case of $d=1$ and $\dim(\HilbertRaum) < \aleph_{0}$.
\end{rem}

\begin{rem}
    It would be interesting to know whether the discretised interpolation
    presented in this subsection
    can be applied in the numerical analysis of stochastic partial differential equations
    or in resampling methods in signal processing.
\end{rem}




\subsubsection[Approximations via scaled Bhat--Skeide interpolations]{Approximations via scaled Bhat--Skeide interpolations}
\label{sec:interpolation:application-recovery:sig:article-interpolation-raj-dahya}

\firstparagraph
The methods in the previous section do not in general
recover the original multi-parameter semigroup.
However, by using ever finer interpolations,
reconstruction is possible up to arbitrary approximation.

Throughout this subsection we use the following notation:
Given a projection $p$ on a Hilbert space,
we denote
    $p^{\perp^{0}} \coloneqq p$
    and
    $p^{\perp^{1}} \coloneqq \onematrix - p$.
Given $d\in\naturals$
and a commuting $d$-tuple
    $\{S_{i}\}_{i=1}^{d}$
    of contractions on a Hilbert space $\HilbertRaum$
we let
    ${
        [\{S_{i}\}_{i=1}^{d}]_{\bhatskeide}
        \colon
        \realsNonNeg^{d} \to \OpSpaceC{L^{2}(\Torus^{d}) \otimes \HilbertRaum}
    }$
denote the $\topSOT$-continuous contractive homomorphism
over $\realsNonNeg^{d}$ on $\HilbertRaum$
corresponding to the generalised Bhat--Skeide interpolation
of $\{S_{i}\}_{i=1}^{d}$
as constructed in the proof of
    \Cref{lemm:bhat-skeide:multi:sig:article-interpolation-raj-dahya}.

Consider now a commuting family
    $\{T_{i}\}_{i=1}^{d}$
of contractive $\Cnought$-semigroups
on an infinite-dimensional Hilbert space $\HilbertRaum$,
or correspondingly an $\topSOT$-continuous homomorphism
    $T \in \SpHomCs(\realsNonNeg^{d},\HilbertRaum)$
(\cf the discussion in \S{}\ref{sec:introduction:definitions:sig:article-interpolation-raj-dahya}).
For
    $\eps > 0$
    and
    any unitary operator
    ${w \colon L^{2}(\Torus^{d}) \otimes \HilbertRaum \to \HilbertRaum}$
we construct
    $
        [T]_{\bhatskeide}^{\eps,w}
        \colon
            \realsNonNeg^{d}
            \to
            \OpSpaceC{\HilbertRaum}
    $
via

    \begin{restoremargins}
    \begin{equation}
    \label{eq:signal-recovery:sig:article-interpolation-raj-dahya}
        \begin{displayarray}[m]{rcl}
            [T]_{\bhatskeide}^{\eps,w}(\mathbf{t})
                &\coloneqq
                    &\adjoint_{w}[\{T_{i}(\eps)\}_{i=1}^{d}]_{\bhatskeide}(\eps^{-1}\mathbf{t})\\
                &= &w\Big(
                    \prod_{i=1}^{d}
                    \sum_{e \in \{0,1\}}
                        U_{i}(t_{i})
                        P_{i}(t_{i})^{\perp^{e_{i}}}
                        \otimes
                        T_{i}(\eps)^{\floor{\tfrac{t_{i}}{\eps}} + e_{i}}
                    \Big)w^{\ast}\\
                &= &\sum_{\mathbf{e} \in \{0,1\}^{d}}
                    w\Big(
                        \Big(
                            \prod_{i=1}^{d}
                                U_{i}(t_{i})
                                P_{i}(t_{i})^{\perp^{e_{i}}}
                        \Big)
                        \otimes
                        T(\mathbf{t}^{(\mathbf{e})})
                    \Big)w^{\ast}\\
        \end{displayarray}
    \end{equation}
    \end{restoremargins}

\continueparagraph
for $\mathbf{t} \in \realsNonNeg^{d}$,
where
    ${
        \mathbf{t}^{(\mathbf{e})}
        \coloneqq
        ((\floor{\tfrac{t_{i}}{\eps}} + e_{i})\eps)_{i=1}^{d}
    }$
    for each $\mathbf{e} \in \{0,1\}^{d}$
and where the
    unitaries $U_{i}(t_{i})$
    and projections $P_{i}(t_{i})$
are the operators on $L^{2}(\Torus^{d})$
defined as in the proof of \Cref{lemm:bhat-skeide:multi:sig:article-interpolation-raj-dahya}.
It is easy to see that
    $[T]_{\bhatskeide}^{\eps,w}$
defines an $\topSOT$-continuous contractive homomorphism
over $\realsNonNeg^{d}$ on $\HilbertRaum$,
\idest corresponds to a commuting family
of $d$ contractive $\Cnought$-semigroups on $\HilbertRaum$.
Note that as per the discussion at the end of
\S{}\ref{sec:interpolation:single:sig:article-interpolation-raj-dahya},
these semigroups are furthermore unitary if (and only if)
each of the $T_{i}$ are.

The construction in \eqcref{eq:signal-recovery:sig:article-interpolation-raj-dahya}
can be understood as being
unitarily similar to a Bhat--Skeide interpolation
between the discretely spaced values
    $\{
        \prod_{i=1}^{d}T_{i}(k_{i}\eps) \mid \mathbf{k} \in \naturalsZero^{d}
    \}$
of the family $\{T_{i}\}_{i=1}^{d}$.
For ever smaller values of $\eps$
and appropriately chosen unitary operators $w$,
one may ask whether these $d$-parameter semigroups
approximate $\{T_{i}\}_{i=1}^{d}$.
This is indeed the case.

\begin{prop}
\makelabel{prop:signal-recovery:approx:sig:article-interpolation-raj-dahya}
    Let $d \in \naturals$ and $\HilbertRaum$ be an infinite-dimensional Hilbert space.
    For each $T \in \SpHomCs(\realsNonNeg^{d},\HilbertRaum)$
    there exists a net
        $(T^{(\alpha)})_{\alpha \in \Lambda} \subseteq \SpHomCs(\realsNonNeg^{d},\HilbertRaum)$
    of homomorphisms unitarily similar to time-scaled Bhat--Skeide interpolations
    such that
        ${T^{(\alpha)} \underset{\alpha}{\overset{\tinytoplocWOT}{\longrightarrow}} T}$.
\end{prop}

    \begin{proof}
        Let $P \subseteq \BoundedOps{\HilbertRaum}$ be the index set
        consisting of finite-rank projections on $\HilbertRaum$
        directly ordered by
                $p \geq q \mathrel{\ratio\Leftrightarrow} \ran(p) \supseteq \ran(q)$
            for $p,q \in P$.
        For $p \in P$, since $\HilbertRaum$ is infinite-dimensional,
        we may find a unitary operator
            ${w_{p} \colon L^{2}(\Torus^{d}) \otimes \HilbertRaum \to \HilbertRaum}$
        satisfying
            $
                w_{p}\,(\einser \otimes \xi)
                = \xi
            $
        for each $\xi\in\ran(p)$.%
        \footref{ft:pf:unitary-equivalence-L^2-tensor-H:sig:article-interpolation-raj-dahya}
        Let $(0,\:\infty) \times P$ be the index set
        directly ordered by
            $
                (\eps^{\prime},p) \geq (\eps,q)
                \mathrel{\ratio\Leftrightarrow}
                \ran(p) \supseteq \ran(q)
                ~\text{and}~
                \eps^{\prime} \leq \eps
            $
        and consider the net
            $\Big(
                T^{(\eps,p)}
                \coloneqq
                [T]_{\bhatskeide}^{\eps,w_{p}}
            \Big)_{(\eps, p) \in (0,\:\infty) \times P}$.

        Fix arbitrary
            $\delta > 0$,
            $F \subseteq \HilbertRaum$ finite,
            and
            $K \subseteq \realsNonNeg^{d}$ compact.
        Let $\tau > 0$ be sufficiently large such that
            $[0,\:\tau]^{d} \supseteq K$.
        Since $T$ is uniformly $\topSOT$-continuous on $[0,\:2\tau]^{d}$,
        there exists
            $\eps_{0} \in (0,\:\tau)$
        such that

            \begin{restoremargins}
            \begin{equation}
            \label{eq:1:sig:article-interpolation-raj-dahya}
                \sup_{\substack{%
                    \mathbf{s},\mathbf{t} \in [0,\:2\tau]^{d}:\\
                    \max_{i}\abs{s_{i}-t_{i}} \leq \eps_{0}
                }}
                    \norm{
                        (
                            T(\mathbf{s})
                            -
                            T(\mathbf{t})
                        )\xi
                    }
                < \frac{\delta}{\max\{\norm{\eta} \mid \eta \in F\} + 1}
            \end{equation}
            \end{restoremargins}

        \continueparagraph
        holds for each $\xi \in F$.
        Let $p_{0} \in P$ be the finite-rank projection onto $\linspann(F)$.
        Consider an arbitrary $(\eps,p) \in (0,\:\infty) \times P$
        with $(\eps,p) \geq (\eps_{0},p_{0})$,
        \idest
            $\eps \leq \eps_{0}$ and $\ran(p) \supseteq \ran(p_{0})$.
        For $\xi,\eta \in F$ and $\mathbf{t} \in K$
        one computes

            \begin{longeqnarray}
                &&\brkt{[T]_{\bhatskeide}^{\eps,w_{p}}(\mathbf{t})\xi}{\eta}\\
                &\eqcrefoverset{eq:signal-recovery:sig:article-interpolation-raj-dahya}{=}
                    &\sum_{\mathbf{e} \in \{0,1\}^{d}}
                        \brktLong{
                            w_{p}
                            \Big(
                                \Big(
                                    \prod_{i=1}^{d}
                                        U_{i}(t_{i})
                                        P_{i}(t_{i})^{\perp^{e_{i}}}
                                \Big)
                                \otimes
                                T(\mathbf{t}^{(\mathbf{e})})
                            \Big)
                            w_{p}^{\ast}
                            \:\xi
                        }{\eta}\\
                &=
                    &\sum_{\mathbf{e} \in \{0,1\}^{d}}
                        \brktLong{
                            \Big(
                                \Big(
                                    \prod_{i=1}^{d}
                                        U_{i}(t_{i})
                                        P_{i}(t_{i})^{\perp^{e_{i}}}
                                \Big)
                                \otimes
                                T(\mathbf{t}^{(\mathbf{e})})
                            \Big)
                            \:
                            w_{p}^{\ast}\xi
                        }{
                            w_{p}^{\ast}\eta
                        }\\
                &=
                    &\sum_{\mathbf{e} \in \{0,1\}^{d}}
                        \brktLong{
                            \Big(
                                \Big(
                                    \prod_{i=1}^{d}
                                        U_{i}(t_{i})
                                        P_{i}(t_{i})^{\perp^{e_{i}}}
                                \Big)
                                \otimes
                                T(\mathbf{t}^{(\mathbf{e})})
                            \Big)
                            \:
                            (\einser \otimes \xi)
                        }{
                            \einser \otimes \eta
                        }\\
                &&\text{
                    per construction of $w_{p}$
                    and since
                        $
                            \xi,\eta \in F
                            \subseteq \ran(p_{0})
                            \subseteq \ran(p)
                        $
                }\\
                &=
                    &\sum_{\mathbf{e} \in \{0,1\}^{d}}
                        \underbrace{
                            \brktLong{
                                \Big(
                                    \prod_{i=1}^{d}
                                        U_{i}(t_{i})
                                        P_{i}(t_{i})^{\perp^{e_{i}}}
                                \Big)
                                \einser
                            }{\einser}
                        }_{\eqqcolon \nu_{\mathbf{e}}(\mathbf{t})}
                        \brkt{
                            T(\mathbf{t}^{(\mathbf{e})})
                            \xi
                        }{\eta}.\\
            \end{longeqnarray}

        Now,
        by the commutativity of each
            $U_{i}(t_{i})$ with $P_{j}(t_{j})$ for $i \neq j$
        as well as the commutativity
        of the projections
            $\{P_{i}(t_{i})\}_{i=1}^{d}$,
        and since
            $\{U_{i}(t_{i})\}_{i=1}^{d}$
        are Koopman operators
        (\cf the construction in the proof of \Cref{lemm:bhat-skeide:multi:sig:article-interpolation-raj-dahya}),
        one has

            \begin{longeqnarray}
                \nu_{\mathbf{e}}(\mathbf{t})
                &= &\brktLong{
                    \Big(
                        \prod_{i=1}^{d}
                            U_{i}(t_{i})
                        \prod_{i=1}^{d}
                            P_{i}(t_{i})^{\perp^{e_{i}}}
                    \Big)
                    \einser
                }{\einser}\\
                &= &\brktLong{
                    \Big(
                        \prod_{i=1}^{d}
                            P_{i}(t_{i})^{\perp^{e_{i}}}
                    \Big)
                    \einser
                }{
                    \Big(
                        \prod_{i=1}^{d}
                            U_{i}(t_{i})
                    \Big)^{\ast}
                    \einser
                }\\
                &= &\brktLong{
                    \Big(
                        \prod_{i=1}^{d}
                            P_{i}(t_{i})^{\perp^{e_{i}}}
                    \Big)
                    \einser
                }{
                    \einser
                }
                = \normLong{
                        \Big(\prod_{i=1}^{d}P_{i}(t_{i})^{\perp^{e_{i}}} \Big)
                        \einser
                    }^{2}
                \in [0,\:1]
            \end{longeqnarray}

        \continueparagraph
        for each $\mathbf{e}\in\{0,1\}^{d}$.
        Furthermore

            \begin{shorteqnarray}
                \sum_{\mathbf{e} \in \{0,1\}^{d}}
                    \nu_{\mathbf{e}}(\mathbf{t})
                &= &\sum_{\mathbf{e} \in \{0,1\}^{d}}
                    \brktLong{
                        \Big(
                            \prod_{i=1}^{d}
                                P_{i}^{\perp^{e_{i}}}(t_{i})
                        \Big)
                        \einser
                    }{
                        \einser
                    }\\
                &= &\brktLong{
                        \sum_{\mathbf{e} \in \{0,1\}^{d}}
                        \Big(
                            \prod_{i=1}^{d}
                                P_{i}^{\perp^{e_{i}}}(t_{i})
                        \Big)
                        \einser
                    }{
                        \einser
                    }\\
                &= &\brktLong{
                        \Big(
                            \prod_{i=1}^{d}
                                \underbrace{
                                    \sum_{e \in \{0,1\}}
                                        P_{i}^{\perp^{e}}(t_{i})
                                }_{
                                    =P_{i}(t_{i}) + (\onematrix - P_{i}(t_{i}))
                                    =\onematrix
                                }
                        \Big)
                        \einser
                    }{
                        \einser
                    }
                = \brkt{\einser}{\einser}
                = 1,
            \end{shorteqnarray}

        \continueparagraph
        \idest
            $\{
                \nu_{\mathbf{e}}(\mathbf{t})
            \}_{\mathbf{e}\in\{0,1\}^{d}}$
        are non-negative reals summing to $1$.
        Applying this to the above expression
        for the inner product thus yields

        \begin{shorteqnarray}
            \brkt{
                (
                    T^{(\eps,p)}(\mathbf{t})
                    -
                    T(\mathbf{t})
                )
                \xi
            }{\eta}
                &= &\brkt{
                        (
                            [T]_{\bhatskeide}^{\eps,w_{p}}
                            -
                            T(\mathbf{t})
                        )
                        \xi
                    }{\eta}\\
                &= &\sum_{\mathbf{e} \in \{0,1\}^{d}}
                        \nu_{\mathbf{e}}(\mathbf{t})
                        \cdot
                        \brkt{
                            (
                                T(\mathbf{t}^{(\mathbf{e})})
                                -
                                T(\mathbf{t})
                            )
                            \xi
                        }{\eta}.\\
        \end{shorteqnarray}

        Since $\mathbf{t} \in K \subseteq [0,\:\tau]^{d}$,
        by the choice of $\eps_{0}$
        one has
            $
                \mathbf{t}^{(\mathbf{e})}
                = ((\floor{\tfrac{t_{i}}{\eps}} + e_{i})\eps)_{i=1}^{d}
                \in [0,\:\tau + \eps]^{d}
                \subseteq [0,\:\tau + \eps_{0}]^{d}
                \subseteq [0,\:2\tau]^{d}
            $
        for each $\mathbf{e} \in \{0,1\}^{d}$.
        Noting further that
            $
                \abs{
                    t_{i}
                    -
                    (\floor{\tfrac{t_{i}}{\eps}} + e_{i})\eps
                }
                = \abs{\fractional{\tfrac{t_{i}}{\eps}} - e_{i}} \cdot \eps
                \leq \max\{
                        1-\fractional{\tfrac{t_{i}}{\eps}},
                        \fractional{\tfrac{t_{i}}{\eps}}
                    \} \cdot \eps
                \leq \eps
                \leq \eps_{0}
            $,
        one may apply the uniform continuity condition in
        \eqcref{eq:1:sig:article-interpolation-raj-dahya}
        to the above computation, yielding

            \begin{longeqnarray}
                &&\sup_{\mathbf{t} \in K}
                    \abs{
                    \brkt{
                        (
                            T^{(\eps,p)}(\mathbf{t})
                            -
                            T(\mathbf{t})
                        )
                        \xi
                    }{\eta}}\\
                &=
                    &\sup_{\mathbf{t} \in K}
                    \sum_{\mathbf{e} \in \{0,1\}^{d}}
                        \nu_{\mathbf{e}}(\mathbf{t})
                        \cdot
                        \abs{
                        \brkt{
                            (
                                T(\mathbf{t}^{(\mathbf{e})})
                                -
                                T(\mathbf{t})
                            )
                            \xi
                        }{\eta}
                        }\\
                &\leq
                    &\underbrace{
                        \sum_{\mathbf{e} \in \{0,1\}^{d}}
                            \nu_{\mathbf{e}}(\mathbf{t})
                    }_{=1}
                    \cdot
                    \sup_{\mathbf{t} \in K}
                    \max_{\mathbf{e} \in \{0,1\}^{d}}
                        \abs{
                        \brkt{
                            (
                                T(\mathbf{t}^{(\mathbf{e})})
                                -
                                T(\mathbf{t})
                            )
                            \xi
                        }{\eta}}\\
                &\leq
                    &\sup_{\substack{%
                        \mathbf{s},\mathbf{t} \in [0,\:2\tau]^{d}:\\
                        \max_{i}\abs{s_{i}-t_{i}} \leq \eps_{0}
                    }}
                        \abs{
                        \brkt{
                            (
                                T(\mathbf{s})
                                -
                                T(\mathbf{t})
                            )
                            \xi
                        }{\eta}}\\
                &\textoverset{C-S}{\leq}
                    &\sup_{\substack{%
                        \mathbf{s},\mathbf{t} \in [0,\:2\tau]^{d}:\\
                        \max_{i}\abs{s_{i}-t_{i}} \leq \eps_{0}
                    }}
                        \norm{
                            (
                                T(\mathbf{s})
                                -
                                T(\mathbf{t})
                            )
                            \xi
                        }
                        \norm{\eta}\\
                &\eqcrefoverset{eq:1:sig:article-interpolation-raj-dahya}{<}
                    &\delta\\
            \end{longeqnarray}

        \continueparagraph
        for $\xi,\eta \in F$.
        Hence
            ${
                (T^{(\eps, p)})_{(\eps, p) \in (0,\:\infty) \times P}
                \overset{\tinytoplocWOT}{\longrightarrow}
                T
            }$.
    \end{proof}

By construction, the multi-parameter semigroups
in \eqcref{eq:signal-recovery:sig:article-interpolation-raj-dahya}
are unitarily similar to time-scaled Bhat--Skeide interpolations.
As such, each such semigroup is completely determined
by a few pieces of information:
    a time-scale $\eps > 0$,
    a $d$-tuple $\{S_{i}\}_{i=1}^{d}$
    of commuting contractions on the Hilbert space $\HilbertRaum$,
and
    a unitary operator
    ${w \colon L^{2}(\Torus^{d}) \otimes \HilbertRaum \to \HilbertRaum}$.
The class of such semigroups thus admits a simple parameterisation
and simple connections to discrete-time processes.
This provides good reason to find these approximants interesting
in and of themselves.

The following result is likely well-known,
but demonstrates a simple application
of these approximants.

\begin{prop}
\makelabel{prop:dilatability-iff-ptwise-power-dilation:sig:article-interpolation-raj-dahya}
    Let
        $d \in \naturals$
        and
        $\{T_{i}\}_{i=1}^{d}$ be a commuting family
        of contractive $\Cnought$-semigroups
        on an infinite-dimensional Hilbert space $\HilbertRaum$.
    Then $\{T_{i}\}_{i=1}^{d}$ has a simultaneous unitary dilation
    if and only if
        $\{T_{i}(t_{i})\}_{i=1}^{d}$
    has a simultaneous power-dilation
    for all $\mathbf{t} \in \realsNonNeg^{d}$.
\end{prop}

    \begin{proof}
        The \usesinglequotes{only if}-direction is a straightforward observation.
        Towards the \usesinglequotes{if}-direction,
        first let
            ${T \colon \realsNonNeg^{d} \to \OpSpaceC{\HilbertRaum}}$
        be the corresponding $\topSOT$-continuous homomorphism
        corresponding to $\{T_{i}\}_{i=1}^{d}$.
        By
            \cite[Theorem~1.18]{Dahya2024approximation}
        (see also
            \Cref{prop:dilatability-equiv-to-weak-unitary-approximability:sig:article-interpolation-raj-dahya}
        below)
        it suffices to find $\toplocWOT$-approximants
        for $T$ which are unitarily dilatable.
        To this end it suffices to consider the approximants
            $
                (
                    [T]_{\bhatskeide}^{\eps,w_{p}}
                )_{(\eps,p) \in (0,\:\infty) \times P}
            $
        constructed in
        \Cref{prop:signal-recovery:approx:sig:article-interpolation-raj-dahya}.

        Let
            $\eps > 0$
            and
            ${w \colon L^{2}(\Torus^{d}) \otimes \HilbertRaum \to \HilbertRaum}$
        be arbitrary.
        By assumption,
            $\{S_{i} \coloneqq T_{i}(\eps)\}_{i=1}^{d}$
        admits some power-dilation
            $(
                \HilbertRaum^{\prime},
                \{V_{i}\}_{i=1}^{d},
                r
            )$.
        Consider now the Bhat--Skeide interpolations
            $[\{S_{i}\}_{i=1}^{d}]_{\bhatskeide}$
            and
            $[\{V_{i}\}_{i=1}^{d}]_{\bhatskeide}$.
        By the discussion at the end of
        \S{}\ref{sec:interpolation:single:sig:article-interpolation-raj-dahya},
        we know that
        the latter corresponds to
        a commuting family of $d$ unitary $\Cnought$-semigroups.
        By working with the generalised Bhat--Skeide construction
        (see \eqcref{eq:bs-dilation-family:sig:article-interpolation-raj-dahya})
        it is a straightforward exercise to see that
            $(
                L^{2}(\Torus^{d}) \otimes\HilbertRaum^{\prime},
                [\{V_{i}\}_{i=1}^{d}]_{\bhatskeide},
                \id \otimes r
            )$
        constitutes a unitary dilation of
            $[\{S_{i}\}_{i=1}^{d}]_{\bhatskeide}$.
        That is,
            $[\{T_{i}(\eps)\}_{i=1}^{d}]_{\bhatskeide}$
        and hence also
            $
                [T]_{\bhatskeide}^{\eps,w}
                =
                \adjoint_{w} \circ [\{T_{i}\}_{i=1}^{d}]_{\bhatskeide}(\eps^{-1}\cdot)
            $
        are unitarily dilatable.
        Since this holds for all $\eps$, $w$,
        it follows that $T$ is unitarily dilatable.
    \end{proof}




\subsection[Non-dilatable families of semigroups]{Non-dilatable families of semigroups}
\label{sec:interpolation:counterexamples:sig:article-interpolation-raj-dahya}

\firstparagraph
Building on the following counter-examples,
the generalised Bhat--Skeide interpolation provides sufficient means to prove
\Cref{thm:existence-of-non-dilatable:sig:article-interpolation-raj-dahya}.

\begin{prop}[Parrott, 1970]
\makelabel{prop:parrotts-counterexample:sig:article-interpolation-raj-dahya}
    Let $H_{0}$ be a Hilbert space.
    Every pair of non-commuting unitaries $R_{1},R_{2}\in\BoundedOps{H_{0}}$
    yields a tuple

        \begin{displaymath}
            \Big(
                R_{1} \otimes E_{21},R_{2} \otimes E_{21}, \onematrix_{H_{0}} \otimes E_{21}
            \Big)
            \in \BoundedOps{H_{0} \otimes \complex^{2}}^{3}
        \end{displaymath}

    \continueparagraph
    of commuting contractions which admits no power-dilation.%
    \footref{ft:Eij:sig:article-interpolation-raj-dahya}
\end{prop}

\footnotetext[ft:Eij:sig:article-interpolation-raj-dahya]{%
    Recall that $E_{ij} \in M_{2}(\complex)$
    denotes the elementary matrix with $1$ in the $(i,j)$-th entry, and $0$'s elsewhere.
}

A proof of this can be found
in \cite[\S{}3]{Parrott1970counterExamplesDilation}
and
\cite[\S{}I.6.3]{Nagy1970}.
Note that in Parrott's presentation,
it suffices to require $R_{1}$ to be unitary
and $R_{2}$ to be a contraction.
We also make use of the following relation
established in \cite{Dahya2024approximation}:

\begin{prop}[Equiv. of dilations and approximations]
\makelabel{prop:dilatability-equiv-to-weak-unitary-approximability:sig:article-interpolation-raj-dahya}
    Let $\HilbertRaum$ be an infinite-dimensional Hilbert space
    and
        $
            (G,M) \in \{
                (\reals^{d},\realsNonNeg^{d}),
                (\integers^{d},\naturalsZero^{d}),
                \mid
                d \in \naturals
            \}
        $.
    Let $A,A^{\ast},\quer{A},\quer{A^{\ast}},D$
    be as defined above (see \S{}\ref{sec:introduction:notation:sig:article-interpolation-raj-dahya}).
    Then
        $A = A^{\ast}$
        and
        $\quer{A} = \quer{A^{\ast}} = D$.
\end{prop}

    \begin{proof}
        By the \onetoone-correspondence between
            $\SpHomUs(M,\HilbertRaum)$
            and
            $\Repr{G}{\HilbertRaum}$
        mentioned in \S{}\ref{sec:introduction:definitions:sig:article-interpolation-raj-dahya},
        we have $A = A^{\ast}$.
        By \cite[Theorem~1.18]{Dahya2024approximation},
        which is applicable to $(G,M)$
        (see
            \cite[Example~1.10]{Dahya2024approximation}
            and
            \cite[Appendix~A, Example~4.5, p.~522]{Dahya2022complmetrproblem}%
        )
        and to our infinite-dimensional Hilbert space $\HilbertRaum$,
        unitary dilatability of an element $T \in X$ is equivalent to
        $T$ being $\toplocWOT$-approximable via elements of $A^{\ast}$
        (\cf \cite[Definition~1.15]{Dahya2024approximation}).
        Thus
            $\quer{A} = \quer{A^{\ast}} = D$.
    \end{proof}

\def\beweislabel{thm:existence-of-non-dilatable:sig:article-interpolation-raj-dahya}
\begin{proof}[of \Cref{\beweislabel}]
    Without loss of generality, we can assume that $\mathcal{I} \supseteq \{1,2,3\}$.
    Fix $Z = \Torus^{3}$.
    Since $\dim(\HilbertRaum) \geq \aleph_{0} = \dim(L^{2}(Z))$,
    we can assume without loss of generality that
        $\HilbertRaum = L^{2}(Z) \otimes H$
    where
        $H = H_{0} \otimes \complex^{2}$
    for some Hilbert space
        $H_{0}$ with $\dim(H_{0}) = \dim(\HilbertRaum)$.

    Choose now two non-commuting unitaries $R_{1},R_{2} \in \BoundedOps{H_{0}}$
        (this is possible, since $\dim(H_{0}) \geq 2$).
    Set
        $S_{1} \coloneqq R_{1} \otimes E_{21}$,
        $S_{2} \coloneqq R_{2} \otimes E_{21}$,
        and
        $S_{3} \coloneqq \onematrix_{H_{0}} \otimes E_{21}$.
    Since
        $\{S_{i}\}_{i=1}^{3} \subseteq \BoundedOps{H}$
    is a commuting family of contractions
    (\cf \Cref{prop:parrotts-counterexample:sig:article-interpolation-raj-dahya}),
    by the generalised Bhat--Skeide interpolation
    (\Cref{lemm:bhat-skeide:multi:sig:article-interpolation-raj-dahya}),
    there exists a commuting family,
        $\{T_{i}\}_{i=1}^{3}$
    of contractive $\Cnought$-semigroups
    on $L^{2}(Z) \otimes H = \HilbertRaum$,
    such that
        $
            \prod_{i=1}^{3}T_{i}(n_{i})
            = \prod_{i=1}^{3}
                (\onematrix_{L^{2}(Z)} \otimes S_{i})^{n_{i}}
        $
    for all $\mathbf{n} = (n_{i})_{i=1}^{3}\in\naturalsZero^{3}$.
    Set $T_{i} \coloneqq (\onematrix_{\HilbertRaum})_{t\in\realsNonNeg}$
    for $i \in \mathcal{I} \without \{1,2,3\}$.

    We claim that
        $\{T_{i}\}_{i=1}^{3}$
    and hence
        $\{T_{i}\}_{i\in\mathcal{I}}$
    do not have simultaneous unitary dilations.
    Suppose, towards a contradiction, that
        $(\HilbertRaum^{\prime},\{U_{i}\}_{i=1}^{3},r)$
    is a simultaneous unitary dilation for $\{T_{i}\}_{i=1}^{3}$.
    Set $V_{i} \coloneqq U_{i}(1)$ for each $i\in\mathcal{I}$.
    Then

        \begin{shorteqnarray}
            r^{\ast}\Big(\prod_{i=1}^{3}V_{i}^{n_{i}}\Big)r
            &= &r^{\ast}\Big(\prod_{i=1}^{3}U_{i}(1)^{n_{i}}\Big)r\\
            &= &r^{\ast}\Big(\prod_{i=1}^{3}U_{i}(n_{i})\Big)r\\
            &= &\prod_{i=1}^{3}T_{i}(n_{i})\\
            &= &\prod_{i=1}^{3}(\onematrix_{L^{2}(Z)} \otimes S_{i})^{n_{i}}
        \end{shorteqnarray}

    \continueparagraph
    for $\mathbf{n} = (n_{i})_{i=1}^{3}\in\prod_{i=1}^{3}\naturalsZero$.
    Thus

        \begin{displaymath}
            \{ \onematrix_{L^{2}(Z)} \otimes S_{i} \}_{i=1}^{3}
            =
            \Big(
                R^{\prime}_{1} \otimes E_{21},
                R^{\prime}_{2} \otimes E_{21},
                \onematrix_{H_{0}^{\prime}} \otimes E_{21}
            \Big)
            \in \BoundedOps{H_{0}^{\prime} \otimes \complex^{2}}^{3}
        \end{displaymath}

    \continueparagraph
    has a power-dilation,
    where
        $R^{\prime}_{1} \coloneqq \onematrix_{L^{2}(Z)} \otimes R_{1}$,
        $R^{\prime}_{2} \coloneqq \onematrix_{L^{2}(Z)} \otimes R_{2}$,
        and
        $H_{0}^{\prime} \coloneqq L^{2}(Z) \otimes H_{0}$.
    But applying Parrott's result (\Cref{prop:parrotts-counterexample:sig:article-interpolation-raj-dahya}),
    since $R^{\prime}_{1},R^{\prime}_{2} \in \BoundedOps{L^{2}(Z)\otimes H_{0}}$
    are clearly non-commuting unitaries,
        $\{
            R^{\prime}_{1} \otimes E_{21},
            R^{\prime}_{2} \otimes E_{21},
            \onematrix_{H_{0}^{\prime}} \otimes E_{21}
        \}$
    cannot have a power-dilation.
    As this is a contradiction,
        $\{T_{i}\}_{i=1}^{3}$
    admits no simultaneous unitary dilation.

    \paragraph{Construction with bounded generators:}
    For each $i \in \mathcal{I}$
    fix a net
        $(T_{i}^{(\alpha)})_{\alpha \in \Lambda_{i}}$
    of \highlightTerm{Yosida-approximants},
    which converges to $T_{i}$
    in the topology of uniform $\topSOT$-convergence
    on compact subsets of $\realsNonNeg$
    (written: $\toplocSOT$-convergence).%
    \footref{ft:pf:2:\beweislabel}
    Set $\Lambda \coloneqq \prod_{i=1}^{3}\Lambda_{i}$.
    The net
        $(\{T_{i}^{(\alpha_{i})}\}_{i=1}^{3})_{\boldsymbol{\alpha}\in\Lambda}$
    consists of commuting families of contractive $\Cnought$-semigroups
    and converges to
        $\{T_{i}\}_{i=1}^{3}$
    \wrt the topology of uniform $\topSOT$-convergence
    on compact subsets of $\realsNonNeg^{3}$
    (\cf \cite[Lemma~2.4 and Proposition~2.6--7]{Dahya2024approximation}).
    Let
        $T,T^{(\boldsymbol{\alpha})}\in\SpHomCs(\realsNonNeg^{3},\HilbertRaum)$
    be the $3$-parameter contractive $\Cnought$-semigroups
    corresponding to
        $\{T_{i}\}_{i=1}^{3}$
        and
        $\{T_{i}^{(\alpha_{i})}\}_{i=1}^{3}$
        respectively
    for $\boldsymbol{\alpha}\in\Lambda$
    (\cf the discussion in \S{}\ref{sec:introduction:definitions:sig:article-interpolation-raj-dahya}).

    We now claim that for some index $\boldsymbol{\alpha}\in\Lambda$,
    the commuting family
        $\{T_{i}^{(\alpha_{i})}\}_{i=1}^{3}$
    of contractive $\Cnought$-semigroups
    (with bounded generators)
    has no simultaneous unitary dilation.
    If this were not the case, then each
        $T^{(\boldsymbol{\alpha})}$
    has a unitary dilation.
    By \Cref{prop:dilatability-equiv-to-weak-unitary-approximability:sig:article-interpolation-raj-dahya},
    it follows that
        $\{T^{(\boldsymbol{\alpha})} \mid \boldsymbol{\alpha} \in \Lambda\} \subseteq D = \quer{A^{\ast}}$
    where
        $D \subseteq \SpHomCs(\realsNonNeg^{3},\HilbertRaum) \eqqcolon X$
        is the set of unitarily dilatable elements
        and
        $A^{\ast} = \{ U\restr{\realsNonNeg^{3}} \mid U \in \Repr{\reals^{3}}{\HilbertRaum}\}$
    and the closure is computed inside
        $(X,\toplocWOT)$.
    From the above $\toplocSOT$-convergence,
    it follows that
        ${
            (T^{(\boldsymbol{\alpha})})_{\boldsymbol{\alpha}\in\Lambda}
            \overset{\tinytoplocWOT}{\longrightarrow}
            T
        }$
    and thus $T \in \quer{A^{\ast}} = D$.
    Thus $T$ must be unitarily dilatable,
    \idest $\{T_{i}\}_{i=1}^{3}$ has a simultaneous unitary dilation.
    This contradicts the above construction.
    Thus setting
        $\tilde{T}_{i} \coloneqq T^{(\alpha_{i})}_{i}$
        for $i\in\{1,2,3\}$
        and some suitable index $\boldsymbol{\alpha}\in\Lambda$,
        and
        $\tilde{T}_{i} \coloneqq (\onematrix_{\HilbertRaum})_{t\in\realsNonNeg}$
        for $i \in \mathcal{I} \without \{1,2,3\}$,
    one has that $\{\tilde{T}_{i}\}_{i\in\mathcal{I}}$
    is a commuting family of contractive $\Cnought$-semigroups
    with bounded generators and which admits no simultaneous unitary dilation.
\end{proof}

\footnotetext[ft:pf:2:\beweislabel]{%
    Let $A_{i}$ be the generator of $T_{i}$.
    For each $\alpha \in (0,\:\infty)$,
    the $\alpha$-th Yosida-approximant
    is given by
        $
            T_{i}^{(\alpha)} = (e^{tA_{i}^{(\alpha)}})_{t\in\realsNonNeg}
        $
    where
        $
            A_{i}^{(\alpha)}
            = \alpha A_{i}(\alpha\onematrix - A_{i})^{-1}
        $.
    For the well-definedness of these approximants
    and proof of the convergence
        ${
            T_{i}^{(\alpha)}
            \overset{\tinytoplocSOT}{\longrightarrow}
            T_{i}
        }$
    for ${\alpha \to \infty}$,
    see \cite[Theorem~G.4.3]{HytNeervanMarkLutz2016bookVol2}, \cite[Theorem~II.3.5, pp.~73--74]{EngelNagel2000semigroupTextBook}, \cite[(12.3.4), p.~361]{Hillephillips1957faAndSg}.%
}




\section[Residuality results]{Residuality results}
\label{sec:residuality:sig:article-interpolation-raj-dahya}

\firstparagraph
As discussed in the introduction,
there is a \onetoone-correspondence between
commuting families
    $\{T_{i}\}_{i=1}^{d}$
of $d$ contractive $\Cnought$-semigroups on a Hilbert space $\HilbertRaum$
and $\topSOT$-continuous contractive homomorphisms
${T \colon \realsNonNeg^{d} \to \BoundedOps{\HilbertRaum}}$.
Similarly, there is a \onetoone-correspondence between
commuting families
    $\{S_{i}\}_{i=1}^{d}$
of $d$ contractions on $\HilbertRaum$
and contractive homomorphism
${S \colon \naturalsZero^{d} \to \BoundedOps{\HilbertRaum}}$.
Thus an appropriate general setting is to consider
topological submonoids $M$ of a topological group $G$,
and the space
    $\SpHomCs(M,\HilbertRaum)$
of all $\topSOT$-continuous contractive homomorphisms
    ${T \colon M \to \BoundedOps{\HilbertRaum}}$.
We endow this with the $\toplocWOT$-topology
(the topology of uniform $\topWOT$-convergence on compact subsets of $M$,
see \S{}\ref{sec:introduction:definitions:sig:article-interpolation-raj-dahya})
and consider in particular the subspace
    $\SpHomUs(M,\HilbertRaum)$
of all $\topSOT$-continuous unitary homomorphisms
    ${U \colon M \to \BoundedOps{\HilbertRaum}}$.
Recall that in the case of
    $(G,M) \in \{(\reals^{d},\realsNonNeg^{d}), (\integers^{d},\naturalsZero^{d}) \mid d \in \naturals\}$
the map
    ${
        \Repr{G}{\HilbertRaum} \ni U \mapsto U\restr{M} \in \SpHomUs(M,\HilbertRaum)
    }$
is a bijection.

In this section we
first demonstrate the genericity of elements of $\SpHomUs(M,\HilbertRaum)$
with dense orbits under the unitary action.
This is then used to develop a \zeroone-dichotomy for the residuality of
    $\SpHomUs(M,\HilbertRaum)$ within $(\SpHomCs(M,\HilbertRaum),\toplocWOT)$
under modest assumptions.
Since we now have a more complete picture on the existence of
    non-dilatable commuting families of contractive $\Cnought$-semigroups
    as well as non-dilatable commuting families of contractions,
we are able to decide this dichotomy for the concrete cases of
    $M \in \{\realsNonNeg^{d}, \naturalsZero^{d} \mid d \in \naturals\}$.


\subsection[Universal elements]{Universal elements}
\label{sec:residuality:universal:sig:article-interpolation-raj-dahya}

\firstparagraph
Let $M$ be a topological monoid and $\HilbertRaum$ be a Hilbert space.
We define an action of the group
    $\OpSpaceU{\HilbertRaum}$ (the unitary operators on $\HilbertRaum$)
on $\SpHomCs(M,\HilbertRaum)$ via

    \begin{displaymath}
        (\adjoint_{w}T)(x) \coloneqq \adjoint_{w}T(x) \coloneqq w\:T(x)\:w^{\ast}
    \end{displaymath}

\continueparagraph
for
    $w \in \OpSpaceU{\HilbertRaum}$,
    $T \in \SpHomCs(M,\HilbertRaum)$,
    and
    $x\in M$.
One can readily see that this is a well-defined action.
This gives rise to the equivalence relation

    \begin{displaymath}
        T \similarToUnitary T^{\prime}
        :\Leftrightarrow
        \exists{w \in \OpSpaceU{\HilbertRaum}:~}
            \adjoint_{w}T = T^{\prime}
    \end{displaymath}

\continueparagraph
for $T, T^{\prime} \in \SpHomCs(M,\HilbertRaum)$.
We denote the equivalence classes, or \highlightTerm{orbits}, by

    \begin{displaymath}
        \orbitUnitary{T}
            \coloneqq \{
                    T^{\prime} \in \SpHomCs(M,\HilbertRaum)
                    \mid
                    T \similarToUnitary T^{\prime}
                \}
            = \{ \adjoint_{w}T \mid w \in \OpSpaceU{\HilbertRaum}\}
    \end{displaymath}

\continueparagraph
for $T \in \SpHomCs(M,\HilbertRaum)$.

\begin{defn}
    Let
        $\universalElementsUnitary{M,\HilbertRaum} \subseteq \SpHomCs(M,\HilbertRaum)$
    be the set elements ${T \in \SpHomCs(M,\HilbertRaum)}$
    for which $\orbitUnitary{T}$ is dense in $(\SpHomCs(M,\HilbertRaum),\toplocWOT)$.
    That is, $\universalElementsUnitary{M,\HilbertRaum}$
    is the set of elements with dense orbits.
    Refer to the elements in $\universalElementsUnitary{M,\HilbertRaum}$
    as \highlightTerm{universal elements}.
\end{defn}

\begin{prop}[Residuality of universal elements]
\makelabel{prop:residuality-universal-elements:sig:article-interpolation-raj-dahya}
    Let $M$ be a locally compact Polish topological monoid
    and $\HilbertRaum$ be a separable infinite-dimensional Hilbert space.
    Then
        $\universalElementsUnitary{M,\HilbertRaum}$
    is a dense $G_{\delta}$-subset in $(X,\toplocWOT)$
    where $X \coloneqq \SpHomCs(M,\HilbertRaum)$.
\end{prop}

    \begin{proof}
        \paragraph{Borel complexity:}
            As per the discussion in \S{}\ref{sec:introduction:definitions:sig:article-interpolation-raj-dahya}
            after \Cref{qstn:residuality-of-unitaries:sig:article-interpolation-raj-dahya},
            given the conditions on $M$ and $\HilbertRaum$
            we know that $(X,\toplocWOT)$ is a second countable space.
            There thus exists a countable basis $\mathcal{O}$
            for $(X,\toplocWOT)$.
            Observe that for $w \in \OpSpaceU{\HilbertRaum}$
            the map
                ${
                    X \ni T \mapsto \adjoint_{w}T = w\:T(\cdot)\:w^{\ast} \in X
                }$
            is $\toplocWOT$-continuous.
            Thus for $W \in \mathcal{O}$

                \begin{displaymath}
                    \{
                        T \in X
                        \mid
                        \orbitUnitary{T} \cap W \neq \emptyset
                    \}
                    = \bigcup_{w \in \OpSpaceU{\HilbertRaum}}
                        \{
                            T \in X
                            \mid
                                \adjoint_{w}T \in W
                        \}
                \end{displaymath}

            \continueparagraph
            is open.
            Since $\mathcal{O}$ is a (countable) basis,
            it follows that

                \begin{displaymath}
                    \universalElementsUnitary{M,\HilbertRaum}
                    = \bigcap_{W \in \mathcal{O}}
                        \{
                            T \in X
                            \mid
                            \orbitUnitary{T} \cap W \neq \emptyset
                        \},
                \end{displaymath}

            \continueparagraph
            which is a $G_{\delta}$-set.

        \paragraph{Density:}
            For $T \in \universalElementsUnitary{M,\HilbertRaum}$
            and $T^{\prime} \in \orbitUnitary{T}$,
            it holds that $\orbitUnitary{T^{\prime}} = \orbitUnitary{T}$ is dense
            and thus $T^{\prime} \in \universalElementsUnitary{M,\HilbertRaum}$.
            Thus if $\universalElementsUnitary{M,\HilbertRaum}$ is non-empty,
            then it contains all the elements in the (dense) orbit of an element,
            and is thus itself dense.
            Thus it suffices to show that $\universalElementsUnitary{M,\HilbertRaum}$ is non-empty.
            To construct a universal element,
            since $(X,\toplocWOT)$ is second countable and thus separable,
            we can fix a family $(T_{n})_{n\in\naturals} \subseteq X$
            such that $\{T_{n} \mid n\in\naturals\}$ is dense in $(X,\toplocWOT)$.
            Since $\HilbertRaum$ is infinite-dimensional,
            it holds that
                $\HilbertRaum^{\prime} \coloneqq \bigoplus_{n\in\naturals}\HilbertRaum$
            is isomorphic to $\HilbertRaum$,
            \idest there exists a unitary operator
                $v \in \BoundedOps{\HilbertRaum^{\prime}}{\HilbertRaum}$.
            We now construct
                ${T \colon M \to \BoundedOps{\HilbertRaum}}$
            via

                \begin{displaymath}
                    T(x) \coloneqq v\:\Big(\bigoplus_{n\in\naturals}T_{n}(x)\Big)\:v^{\ast}
                \end{displaymath}

            \continueparagraph
            for $x \in M$.
            One can readily check that $T$ is pointwise well-defined and contraction-valued.
            It is routine to check that $T$ is also $\topSOT$-continuous and a homomorphism.
            Thus $T \in \SpHomCs(M,\HilbertRaum) = X$.
            We now show that $\orbitUnitary{T}$ is dense.
            To this end it suffices to prove that
                $T_{n} \in \quer{\orbitUnitary{T}}$
            for each $n\in\naturals$.

            Let $(r_{n})_{n\in\naturals} \subseteq \BoundedOps{\HilbertRaum}{\HilbertRaum^{\prime}}$
            be isometries associated with the direct sum,
            \idest
                $r_{m}^{\ast}r_{n} = \delta_{mn}\onematrix_{\HilbertRaum}$
                for $m,n\in\naturals$
            and
                $
                    \sum_{n \in \naturals}
                        r_{n}r_{n}^{\ast}
                    = \onematrix_{\HilbertRaum^{\prime}}
                $
            (computed strongly).
            In particular we can express the above construction more concretely as

                \begin{restoremargins}
                \begin{equation}
                \label{eq:0:\beweislabel}
                    T(x) = \sum_{n \in \naturals}
                        v r_{n}
                        \:T_{n}(x)
                        \:r_{n}^{\ast} v^{\ast}
                \end{equation}
                \end{restoremargins}

            \continueparagraph
            for $x \in M$, where the sum is computed strongly.

            Let $n_{0}\in\naturals$ be arbitrary.
            Let $P \subseteq \BoundedOps{\HilbertRaum}$ be the index set
            consisting of finite-rank projections on $\HilbertRaum$,
            directly ordered by
                $p \geq q \mathrel{\ratio\Leftrightarrow} \ran(p) \supseteq \ran(q)$
            for $p,q \in P$.
            Since
                $v r_{n_{0}} \in \BoundedOps{\HilbertRaum}$
            is an isometry and $\dim(\HilbertRaum)$ is infinite,
            for each $p \in P$
            there exists a unitary operator
                $w_{p} \in \BoundedOps{\HilbertRaum}$
            such that
                $w_{p}^{\ast}p = v r_{n_{0}}p$
            and thus

                \begin{restoremargins}
                \begin{equation}
                \label{eq:1:\beweislabel}
                    w_{p} v r_{n_{0}}p = p.
                \end{equation}
                \end{restoremargins}

            \continueparagraph
            We now show that
            ${
                \orbitUnitary{T}
                \supseteq (\adjoint_{w_{p}}T)_{p \in P}
                \overset{\tinytoplocWOT}{\longrightarrow}
                T_{n_{0}}
            }$.
            To this end, let $\xi,\eta\in\HilbertRaum$ be arbitrary.
            Let $p_{0} \in P$ be the finite-rank projection onto $\linspann\{\xi,\eta\}$.
            For $p \in P$ with $p \geq p_{0}$ one has

                \begin{restoremargins}
                \begin{equation}
                \label{eq:2:\beweislabel}
                \everymath={\displaystyle}
                \begin{array}[m]{rcl}
                    w_{p} v r_{n_{0}}\xi
                        &=
                            &w_{p} v r_{n_{0}}p\xi
                        \eqcrefoverset{eq:1:\beweislabel}{=}
                            p \xi
                        =
                            \xi,~\text{and}\\
                    w_{p} v r_{n_{0}}\eta
                        &=
                            &w_{p} v r_{n_{0}}p\eta
                        \eqcrefoverset{eq:1:\beweislabel}{=}
                            p \eta
                        =
                            \eta,\\
                \end{array}
                \end{equation}
                \end{restoremargins}

            \continueparagraph
            and thus

                \begin{longeqnarray}
                    &&\brkt{((\adjoint_{w_{p}}T)(x) - T_{n_{0}}(x))\xi}{\eta}\\
                    &= &\brkt{
                            w_{p}
                            T(x)
                            w_{p}^{\ast}
                            \xi
                        }{\eta}
                        - \brkt{T_{n_{0}}(x)\xi}{\eta}\\
                    &\eqcrefoverset{eq:2:\beweislabel}{=}
                        &\brkt{
                            T(x)
                            w_{p}^{\ast}
                            \:
                            w_{p} v r_{n_{0}}
                            \xi
                        }{
                            w_{p}^{\ast}
                            \:
                            w_{p} v r_{n_{0}}
                            \eta
                        }
                        - \brkt{T_{n_{0}}(x)\xi}{\eta}\\
                    &= &\brkt{
                            r_{n_{0}}^{\ast} v^{\ast}
                            \:
                            T(x)
                            \:
                            v r_{n_{0}}
                            \xi
                        }{\eta}
                        - \brkt{T_{n_{0}}(x)\xi}{\eta}\\
                    &\eqcrefoverset{eq:0:\beweislabel}{=}
                        &\sum_{n\in\naturals}
                            \brkt{
                                r_{n_{0}}^{\ast}v^{\ast}
                                (
                                    v r_{n} T_{n}(x) r_{n}^{\ast} v^{\ast}
                                )
                                v r_{n_{0}}
                                \xi
                            }{\eta}
                            - \brkt{T_{n_{0}}(x)\xi}{\eta}\\
                    &= &\sum_{n\in\naturals}
                        \delta_{nn_{0}}
                        \brkt{
                            T_{n} (x)
                            \xi
                        }{\eta}
                        - \brkt{T_{n_{0}}(x)\xi}{\eta}
                    = 0
                \end{longeqnarray}

            \continueparagraph
            for all $x \in M$.
            Hence
                ${
                    (\adjoint_{w_{p}}T)_{p \in P}
                    \overset{\tinytoplocWOT}{\longrightarrow}
                    T_{n_{0}}
                }$
            as claimed.
            Thus $T_{n} \in \quer{\orbitUnitary{T}}$ for all $n\in\naturals$,
            which implies that $\orbitUnitary{T}$ is dense.
            Hence $\universalElementsUnitary{M,\HilbertRaum}$ is non-empty (it contains $T$)
            and thereby, as argued above, dense in $(X,\toplocWOT)$.
    \end{proof}



\subsection[Zero-One dichotomy]{Zero-One dichotomy}
\label{sec:residuality:zero-one:sig:article-interpolation-raj-dahya}

\firstparagraph
We can utilise the residuality of the set of universal elements to reduce
the possibilities for the set of unitarily approximable homomorphisms.
Before doing so, we observe the following basic result:

\begin{prop}
\makelabel{prop:pre-zero-one-dichotomy:sig:article-interpolation-raj-dahya}
    Let $M$ be a topological monoid
    and $\HilbertRaum$ be a Hilbert space.
    Let $A,\quer{A},X$
    be as defined above (see \S{}\ref{sec:introduction:notation:sig:article-interpolation-raj-dahya}).
    Further let $E \subseteq X$ be $\similarToUnitary$-invariant
    (\exempli $E=A$)
    and $\quer{E}$ denote its closure in $(X,\toplocWOT)$.
    Then
        $\quer{E} = X$
        $\Leftrightarrow$
        $\quer{E} \supseteq \universalElementsUnitary{M,\HilbertRaum}$
        $\Leftrightarrow$
        $\quer{E} \cap \universalElementsUnitary{M,\HilbertRaum} \neq \emptyset$.
\end{prop}

    \begin{proof}
        The first two $\Rightarrow$-implications are clear.
        For the remaining implication,
        suppose that $\quer{E} \cap \universalElementsUnitary{M,\HilbertRaum} \neq \emptyset$
        and fix some $T_{0} \in \quer{E} \cap \universalElementsUnitary{M,\HilbertRaum}$.
        Since
            ${
                X \ni T \mapsto \adjoint_{w}T = w\:T(\cdot)\:w^{\ast} \in X
            }$
        is $\toplocWOT$-continuous
        and $E$ is $\similarToUnitary$-invariant,
        it follows that
            $\quer{E} \supseteq \orbitUnitary{T_{0}}$.
        Since $T_{0}$ is universal,
        it follows that $\quer{E}$ is dense in $(X,\toplocWOT)$,
        \idest $\quer{E} = X$.
    \end{proof}

Call a topological monoid $M$ \usesinglequotes{good}
if $\topWOT$-continuous homomorphisms
    ${T \colon M \to \BoundedOps{\HilbertRaum}}$
are automatically $\topSOT$-continuous.
This holds for example for
    $M \in \{\realsNonNeg^{d} \mid d\in\naturals\}$,
    all discrete monoids,
    and a number of non-commutative non-discrete examples
(see \cite[Examples~A.3--A.6, Proposition~A.7, and Theorem~A.8]{Dahya2022complmetrproblem}).

\begin{schattierteboxdunn}[backgroundcolor=leer,nobreak=true]
\begin{lemm}[\zeroone-Dichotomy]
\makelabel{lemm:zero-one-dichotomy:sig:article-interpolation-raj-dahya}
    Let $M$ be a locally compact Polish topological monoid
    and $\HilbertRaum$ be a separable infinite-dimensional Hilbert space.
    Consider the subspace
        $A \coloneqq \SpHomUs(M,\HilbertRaum)$
        of
        $X \coloneqq \SpHomCs(M,\HilbertRaum)$
    under the $\toplocWOT$-topology.
    Further let $E \subseteq X$ be $\similarToUnitary$-invariant
    (\exempli $E=A$)
    and $\quer{E}$ denote its closure in $(X,\toplocWOT)$.
    Then either

    \begin{kompaktenum}{\bfseries (a)}[\rtab]
    \item\punktlabel{1}
        $E$ is dense in $(X,\toplocWOT)$; or
    \item\punktlabel{2}
        $\quer{E}$ is meagre in $(X,\toplocWOT)$,
    \end{kompaktenum}

    \continueparagraph
    and if $M$ is \usesinglequotes{good}
    (\exempli $M \in \{\realsNonNeg^{d},\naturalsZero^{d} \mid d\in\naturals\}$),
    then this is a strict dichotomy.
    In the case of $E=A$,
    the first option in this dichotomy can be furthermore strengthened to:
        $E$ is a dense $G_{\delta}$-subset.
\end{lemm}
\end{schattierteboxdunn}

    \begin{proof}
        Note that by \Cref{prop:residuality-universal-elements:sig:article-interpolation-raj-dahya},
            $\universalElementsUnitary{M,\HilbertRaum}$ is co-meagre.
        By the equivalences in \Cref{prop:pre-zero-one-dichotomy:sig:article-interpolation-raj-dahya},
        either
            $\quer{E} = X$
            or else
            $\quer{E} \subseteq X \without \universalElementsUnitary{M,\HilbertRaum}$,
            in which case $\quer{E}$ is meagre.
        So either \punktcref{1} or \punktcref{2} hold.
        If $M$ is furthermore \usesinglequotes{good},
        then by \cite[Theorem~4.2]{Dahya2022complmetrproblem}
            $(X,\toplocWOT)$
        is a Polish space and thus satisfies the Baire category theorem.
        In particular, it is impossible for $E$ to be dense and have meagre closure,
        \idest \punktcref{1} and \punktcref{2} cannot both be true.
        Towards the final statement,
        given the conditions on $M$ and $\HilbertRaum$,
        we know that $A$ is a $G_{\delta}$-subset in $(X,\toplocWOT)$
        (\cf the discussion in \S{}\ref{sec:introduction:definitions:sig:article-interpolation-raj-dahya}
        after \Cref{qstn:residuality-of-unitaries:sig:article-interpolation-raj-dahya}).
    \end{proof}



\subsection[Application to the multi-parameter setting]{Application to the multi-parameter setting}
\label{sec:residuality:multiparam:sig:article-interpolation-raj-dahya}

\firstparagraph
We now consider the concrete cases of
    $(G,M) \in \{
        (\integers^{d},\naturalsZero^{d}),
        (\reals^{d},\realsNonNeg^{d})
        \mid
        d\in\naturals
    \}$.
For the discrete setting, we shall make use of the following
reformulation of the \zeroone-dichotomy for the $\topPW$-topology:

\begin{prop}
\makelabel{prop:zero-one-for-spaces-of-d-tuples-of-contractions:sig:article-interpolation-raj-dahya}
    Let $\HilbertRaum$ be a separable infinite-dimensional Hilbert space
    and $d\in\naturals$.
    Let
        $
            E
            \subseteq \OpSpaceC{\HilbertRaum}^{d}_{\textup{comm}}
        $
    and assume that $E$ is unitarily invariant
    in the sense that
        $
            \{\adjoint_{w}S_{i}\}_{i=1}^{d}
            = \{w\:S_{i}\:w^{\ast}\}_{i=1}^{d}
            \in E
        $
    for all $\{S_{i}\}_{i=1}^{d} \in E$
    and unitaries $w\in\OpSpaceU{\HilbertRaum}$
    (\exempli $E=\OpSpaceU{\HilbertRaum}^{d}_{\textup{comm}}$).
    Then exactly one of the following holds:
        either $E$ is dense in $(\OpSpaceC{\HilbertRaum}^{d}_{\textup{comm}},\topPW)$
        or $\quer{E}$ is meagre in this space.
    In the case of $E=\OpSpaceU{\HilbertRaum}^{d}_{\textup{comm}}$,
    the first option in this dichotomy can be furthermore strengthened to:
        $E$ is a dense $G_{\delta}$-subset.
\end{prop}

    \begin{proof}
        Set $M \coloneqq \naturalsZero^{d}$
        and $X \coloneqq \SpHomCs(M,\HilbertRaum)$.
        By \Cref{rem:pw-top:sig:article-interpolation-raj-dahya},
            ${(X,\toplocWOT) \cong (\OpSpaceC{\HilbertRaum}^{d}_{\textup{comm}},\topPW)}$
        topologically
        via ${\phi \colon S \mapsto \{S(0,0,\ldots,\underset{i}{1},\ldots,0)\}_{i=1}^{d}}$.
        By the assumptions on $E$ it is straightforward to see that
            $A \subseteq \phi^{-1}(E) \subseteq X$
        and that $\phi^{-1}(E)$ is $\similarToUnitary$-invariant.
        Via the homeomorphism one also has
            $\phi^{-1}(\quer{E}) = \quer{\phi^{-1}(E)}$,
        where $\quer{E}$ denotes the closure of $E$ in
            $(\OpSpaceC{\HilbertRaum}^{d}_{\textup{comm}},\topPW)$.
        By \Cref{lemm:zero-one-dichotomy:sig:article-interpolation-raj-dahya}
        exactly one of the following holds:
        either
            $\phi^{-1}(E)$ is dense in $(X,\toplocWOT)$
            (\respectively a dense $G_{\delta}$-set if $\phi^{-1}(E) = \SpHomUs(M,\HilbertRaum)$)
            or
            $\phi^{-1}(\quer{E}) = \quer{\phi^{-1}(E)}$ is meagre in this space.
        Applying the homeomorphism $\phi$ yields the claimed dichotomy.
    \end{proof}

We can now decide the \zeroone-dichotomy in both the discrete and continuous settings.

\def\beweislabel{cor:zero-one-for-unitaries:sig:article-interpolation-raj-dahya}
\begin{proof}[of \Cref{\beweislabel}]
    \paragraph{Discrete setting:}
    Let $d\in\naturals$ and $(G,M) \coloneqq (\integers^{d},\naturalsZero^{d})$.
    Let $X = \OpSpaceC{\HilbertRaum}^{d}_{\textup{comm}}$,
        $A = \OpSpaceU{\HilbertRaum}^{d}_{\textup{comm}}$,
        and $D \subseteq X$ be the set of $d$-tuples admitting a power-dilation.
    The claim $\quer{A} = D$ follows from \Cref{prop:dilatability-equiv-to-weak-unitary-approximability:sig:article-interpolation-raj-dahya}
    under consideration of the correspondence between commuting families of contractions and homomorphisms
    mentioned in \S{}\ref{sec:introduction:definitions:sig:article-interpolation-raj-dahya}
    and the topological isomorphism mentioned in \Cref{rem:pw-top:sig:article-interpolation-raj-dahya}.
    By the strict \zeroone-dichotomy in \Cref{prop:zero-one-for-spaces-of-d-tuples-of-contractions:sig:article-interpolation-raj-dahya}
        either $A$ is a dense $G_{\delta}$-subset in $(X,\topPW)$
        or $\quer{A}$ is meagre (in particular not dense) in this space.
    If $d \leq 2$,
    then by And\^{o}'s power-dilation theorem for pairs of contractions
    (see \cite{Ando1963pairContractions})
    one has $\quer{A} = D = X$,
    whence the first option in the dichotomy holds.
    And if $d \geq 3$,
    then by the counter-examples of Parrott, Varopoulos, \etAlia
    (see
        \cite[\S{}3]{Parrott1970counterExamplesDilation},
        \cite[\S{}I.6.3]{Nagy1970},
        \cite[Theorem~1]{Varopoulos1974counterexamples},
        \cite[Remark~3.6]{Shalit2021DilationBook}%
    )
    one has $X \without \quer{A} = X \without D \neq \emptyset$,
    whence the second option in the dichotomy holds.

    \paragraph{Continuous setting:}
    Let $d\in\naturals$ and $(G,M) \coloneqq (\reals^{d},\realsNonNeg^{d})$.
    Let $X = \SpHomCs(M,\HilbertRaum)$,
        $A = \SpHomUs(M,\HilbertRaum)$,
        and $D \subseteq X$ be the set of $d$-parameter semigroups admitting a unitary dilation.
    The claim $\quer{A} = D$ was established in \Cref{prop:dilatability-equiv-to-weak-unitary-approximability:sig:article-interpolation-raj-dahya}.
    By the strict \zeroone-dichotomy in \Cref{lemm:zero-one-dichotomy:sig:article-interpolation-raj-dahya},
        either $A$ is a dense $G_{\delta}$-subset in $(X, \toplocWOT)$
        or $\quer{A}$ is meagre in this space.
    If $d \leq 2$, then by S\l{}oci\'{n}ski's theorem
    on simultaneous unitary dilations of pairs of commuting contractive $\Cnought$-semigroups
    (see
        \cite{Slocinski1974},
        \cite[Theorem~2]{Slocinski1982},
        and
        \cite[Theorem~2.3]{Ptak1985}%
    )
    one has $\quer{A} = D = X$,
    whence the first of option in the dichotomy holds.
    And if $d \geq 3$,
    by our counter-examples in \Cref{thm:existence-of-non-dilatable:sig:article-interpolation-raj-dahya}
    one has $X \without \quer{A} = X \without D \neq \emptyset$,
    whence the second option in the dichotomy holds.
\end{proof}



\begin{rem}
    In \cite[Proposition~5.3]{Dahya2023dilation},
    examples of $d$-parameter contractive $\Cnought$-semigroups
    are constructed
    which admit no \highlightTerm{regular} unitary dilation
    (a stronger notion of dilation, see \cite[Definition~2.2]{Dahya2023dilation})
    for $d \geq 2$.
    It would be interesting to know whether for $d \geq 3$
    these examples admit unitary dilations.
\end{rem}

\begin{rem}
    It would be nice to find a simple characterisation
    for the simultaneous unitary dilatability
    of arbitrary $d$-parameter $\Cnought$-semigroups
        $\{T_{i}\}_{i=1}^{d}$
    purely in terms of operator-theoretic properties of the generators
        $\{A_{i}\}_{i=1}^{d}$
    at least in the case of bounded generators.
    This has been established for example in
        \cite{Dahya2023dilation}
    for simultaneous regular unitary dilatability
    via the concept of \highlightTerm{complete dissipativity}
    (see Theorem~1.1 and Definition~2.8 in this reference).
\end{rem}

\begin{rem}[von Neumann inequality]
\makelabel{rem:vN-inequality:sig:article-interpolation-raj-dahya}
    Let $\HilbertRaum$ be a Hilbert space and $d \in \naturals$.
    Let $E_{d} \subseteq \OpSpaceC{\HilbertRaum}^{d}_{\textup{comm}}$
    be the set of $d$-tuples
        $\{S_{i}\}_{i=1}^{d}$
    of commuting contractions on $\HilbertRaum$
    satisfying the \highlightTerm{von Neumann inequality}:

        \begin{restoremargins}
        \begin{equation}
        \label{eq:vN-parrott:sig:article-interpolation-raj-dahya}
            \norm{p(S_{1},S_{2},\ldots,S_{d})}
            \leq
                \sup_{\boldsymbol{\lambda}\in\Torus^{d}}
                    \abs{p(\lambda_{1},\lambda_{2},\ldots,\lambda_{d})}
        \end{equation}
        \end{restoremargins}

    \continueparagraph
    for all polynomials $p\in\complex[X_{1},X_{2},\ldots,X_{d}]$.%
    \footref{ft:rem:1:\beweislabel}
    Parrott conjectured \cite[\S{}5, p.~488]{Parrott1970counterExamplesDilation} that \eqcref{eq:vN-parrott:sig:article-interpolation-raj-dahya} holds
    for all $\{S_{i}\}_{i=1}^{d} \in \OpSpaceC{\HilbertRaum}^{d}_{\textup{comm}}$,
    \idest that $E_{d} = \OpSpaceC{\HilbertRaum}^{d}_{\textup{comm}}$.
    For $d \in \{1,2\}$ the validity of this inequality
    is a simple consequence of
    Sz.-Nagy's and And\^{o}'s dilation results
    (see \cite[Proposition~I.8.3 and Chapter~I Notes, ($\ast$)]{Nagy1970}).

    This conjecture, however, turned out to be false.
    For $d = 3$,
    Varopoulos and Kaijser
    (%
        see
        \cite{Varopoulos1973Article},
        \cite[Theorem~1 and Addendum]{Varopoulos1974counterexamples}%
    ),
    Crabb and Davie \cite{CrabbDavie1975Article},
    and various others
    (%
        \cf
        \cite[Remark~3.6]{Shalit2021DilationBook}%
    )
    provided counterexamples
    to the von Neumann inequality
    for finite-dimensional Hilbert spaces.
    From these counterexamples,
    it is easy to derive that for any
    infinite-dimensional Hilbert space, $\HilbertRaum$, and $d \geq 3$,
    there is a commuting family of $d$ contractions
    for which the von Neumann inequality fails,%
    \footref{ft:rem:2:\beweislabel}
    \idest $E_{d} \neq \OpSpaceC{\HilbertRaum}^{d}_{\textup{comm}}$.

    Now, clearly, $E_{d}$ is unitarily invariant,%
    \footref{ft:rem:3:\beweislabel}
    contains $\OpSpaceU{\HilbertRaum}^{d}_{\textup{comm}}$,
    and
    is closed within $(\OpSpaceC{\HilbertRaum}^{d}_{\textup{comm}},\topPW)$.
    Supposing further that $\dim(\HilbertRaum)=\aleph_{0}$,
    by \Cref{prop:zero-one-for-spaces-of-d-tuples-of-contractions:sig:article-interpolation-raj-dahya}
    it follows that either
        $E_{d} = \OpSpaceC{\HilbertRaum}^{d}_{\textup{comm}}$
        or
        $E_{d}$ is meagre in $(\OpSpaceC{\HilbertRaum}^{d}_{\textup{comm}},\topPW)$.
    Since the first option is ruled out,
    it follows for each $d \geq 3$ that
    $E_{d}$ is meagre,
    or in other words
    the collection of $d$-tuples of commuting contractions
    not satisfying the von Neumann inequality
    is residual in
    $(\OpSpaceC{\HilbertRaum}^{d}_{\textup{comm}},\topPW)$.
\end{rem}

\footnotetext[ft:rem:1:\beweislabel]{%
    The von Neumann inequality is stated this way
    \exempli in
        \cite{GacparSuciu2001vN},
        \cite[\S{}5]{HolbrookHalmost1971vNPoly}.
    In \cite{Parrott1970counterExamplesDilation,Nagy1970},
    the right-hand bound in \eqcref{eq:vN-parrott:sig:article-interpolation-raj-dahya}
    is given as
        $
            \sup_{\boldsymbol{\lambda}\in\{z \in \complex \mid \abs{z} \leq 1\}^{d}}
                \abs{p(\lambda_{1},\lambda_{2},\ldots,\lambda_{d})}
        $.
    The two bounds are, however, the same
    due to the maximum modulus principle for holomorphic functions.
}

\footnotetext[ft:rem:2:\beweislabel]{%
    Letting $\HilbertRaum_{0}$ be a sufficiently large finite-dimensional Hilbert
    space such that there are commuting contractions
    $\{R_{i}\}_{i=1}^{3}$
    for which the von Neumann inequality fails.
    Set $S_{i} \coloneqq R_{i}$ for $i \in \{1,2,3\}$
    and $S_{i} \coloneqq \onematrix_{\HilbertRaum_{0}}$ for $i \in \{4,5,\ldots,d\}$.
    Given an infinite-dimensional $\HilbertRaum$,
    we may fix an isometry ${r \colon \HilbertRaum_{0} \to \HilbertRaum}$.
    Then $\{r\,S_{i}\,r^{\ast}\}_{i=1}^{d}$
    is a $d$-tuple of commuting contractions
    for which the von Neumann inequality continues to fail.
}

\footnotetext[ft:rem:3:\beweislabel]{%
    in the sense that $\{w\:S_{i}\:w^{\ast}\}_{i=1}^{d} \in E_{d}$
    for all $\{S_{i}\}_{i=1}^{d} \in E_{d}$
    and $w\in\OpSpaceU{\HilbertRaum}$.
}

\begin{rem}[Rigidity]
    A $\Cnought$-semigroup $T$ on $\HilbertRaum$ is said to be \highlightTerm{rigid}
    if a sequence ${[0,\:\infty) \ni (t_{n})_{n\in\naturals} \to \infty}$ exists
    such that
        ${(T(t_{n}))_{n\in\naturals} \overset{\tinytopSOT}{\longrightarrow} \onematrix}$.
    In \cite[Theorem~IV.3.20]{Eisner2010buchStableOpAndSemigroups},
    rigidity for unitary $\Cnought$-semigroups
    on separable infinite-dimensional Hilbert spaces
    was shown to be residual
    \idest within $(\SpHomUs(\realsNonNeg,\HilbertRaum),\toplocWOT)$.
    This result was later extended in \cite[Theorem~1.3]{Dahya2022weakproblem}
    to the contractive case
    \idest within $(\SpHomCs(\realsNonNeg,\HilbertRaum),\toplocWOT)$.
    It would be nice to know if there is an appropriate notion of rigidity
    for $d$-parameter $\Cnought$-semigroups, $d \in \naturals$.
    In light of \Cref{cor:zero-one-for-unitaries:sig:article-interpolation-raj-dahya},
    it would be of interest to determine whether for $d \geq 2$
    either the rigidity or non-rigidity of $d$-parameter contractive $\Cnought$-semigroups is residual.
\end{rem}

\begin{rem}[Genericity of embeddability]
\makelabel{rem:genericity-embeddable-tuples-of-contractions:sig:article-interpolation-raj-dahya}
    Let $\HilbertRaum$ be a separable infinite-dimensional Hilbert space and $d\in\naturals$.
    Let $E_{d} \subseteq \OpSpaceC{\HilbertRaum}_{\textup{comm}}$
    be the subset of $d$-tuples of commuting contractions
        $\{S_{i}\}_{i=1}^{d} \in \OpSpaceC{\HilbertRaum}^{d}_{\textup{comm}}$
    that can be \emph{simultaneously embedded}
    into commuting families
        $\{T_{i}\}_{i=1}^{d}$
    of $\Cnought$-semigroups on $\HilbertRaum$.%
    \footref{ft:rem:1:\beweislabel}
    For $d=1$ it is shown in \cite[Theorem~3.2]{Eisner2010typicalContraction}, that $E_{1}$ is residual in
        $(\OpSpaceC{\HilbertRaum},\topWOT)$.
    This in fact also holds \wrt the $\topPW$-topology,
    since the result simply builds on the embeddability of all unitary operators
    and since $\OpSpaceU{\HilbertRaum}$
    is residual in $(\OpSpaceC{\HilbertRaum},\topPW)$
    (see \cite[Theorem~4.1]{Eisnermaitrai2013typicalOperators}).

    For $d \geq 1$, this approach can be generalised as follows:
    For $\{S_{i}\}_{i=1}^{d} \in \OpSpaceU{\HilbertRaum}^{d}_{\textup{comm}}$,
    similar to the proof of
    \cite[Lemma~3.1]{Eisner2010typicalContraction},
    the spectral mapping theorem for commutative $C^{\ast}$-algebras
    can be applied to simultaneously diagonalise
    the $S_{i}$ to multiplication operators
    over a semi-finite measure space
    (%
        see \exempli
        \cite[Theorem~1.3.6]{Murphy1990},
        \cite[\S{}3.3.1 and \S{}3.4.1]{Pederson2018BookCStarAuto},
        and
        \cite{Haase2020spectralTheory}%
    ).
    That is, one can find
        a semi-finite measure space $(X,\mu)$,
        measurable $\reals$-valued functions
        $\theta_{1},\theta_{2},\ldots,\theta_{d}\in L^{\infty}(X,\mu)$,
        and
        a unitary operator $u\in\BoundedOps{\HilbertRaum}{L^{2}(X,\mu)}$,
    such that
        $%
            S_{i} = u^{\ast} M_{e^{\iunit \theta_{i}(\cdot)}} u
        $
    for each $i\in\{1,2,\ldots,d\}$.
    Clearly,
        $\{
            T_{i} \coloneqq
            (
                u^{\ast} M_{e^{\iunit t\theta_{i}(\cdot)}} u
            )_{t\in\realsNonNeg}
        \}_{i=1}^{d}$
    is a commuting family of
    (unitary) $\Cnought$-semigroups
    satisfying $T_{i}(1) = S_{i}$ for each $i$.
    Thus $E_{d} \supseteq \OpSpaceU{\HilbertRaum}^{d}_{\textup{comm}}$
    and clearly $E_{d}$ is unitarily invariant.
    By \Cref{prop:zero-one-for-spaces-of-d-tuples-of-contractions:sig:article-interpolation-raj-dahya}
    it follows that exactly one of the following holds:
        $E_{d}$ is dense in $(\OpSpaceC{\HilbertRaum}^{d}_{\textup{comm}},\topPW)$
        or
        $\quer{E_{d}}$ is meagre in this space.

    Applying \Cref{cor:zero-one-for-unitaries:sig:article-interpolation-raj-dahya},
    if $d \in \{1,2\}$ then
        $\OpSpaceU{\HilbertRaum}^{d}_{\textup{comm}}$
    and thus also $E_{d}$
    are residual in $(\OpSpaceC{\HilbertRaum}^{d}_{\textup{comm}},\topPW)$.
    For $d \geq 3$, one can no longer argue in this way,
    since $\OpSpaceU{\HilbertRaum}^{d}_{\textup{comm}}$
    is no longer residual in
        $(\OpSpaceC{\HilbertRaum}^{d}_{\textup{comm}},\topPW)$.
    We nonetheless ascertained in
    \Cref{prop:sim-embeddable-tuples-pw-dense:sig:article-interpolation-raj-dahya}
    that $E_{d}$
    (%
        as well as a particular subset
        $\tilde{E}_{d} \subseteq E_{d}$
        defined in \eqcref{eq:special-embeddable-tuples:sig:article-interpolation-raj-dahya}%
    )
    is at least dense in
        $(\OpSpaceC{\HilbertRaum}^{d}_{\textup{comm}},\topPW)$,
    provided $\dim(\HilbertRaum) \geq \aleph_{0}$.
    It is however unclear whether
        $E_{d}$ (or $\tilde{E}_{d}$)
    is a $G_{\delta}$-set,
    from which the residuality of embeddable $d$-tuples
    of commuting contractions would follow.
\end{rem}

\footnotetext[ft:rem:1:\beweislabel]{%
    \idest $T_{i}(1) = S_{i}$ for each $i$,
    \cf \Cref{problem:aux-embed-dim:family:sig:article-interpolation-raj-dahya}.
}







\firstparagraph
\paragraph{Acknowledgement.}
The author is grateful to
    Orr Shalit
        for bringing the construction of Bhat and Skeide to our attention
        and
        for his valuable expert insights,
to Tanja Eisner
    for her patient and supportive feedback,
to Jochen Gl\"{u}ck
    for helpful discussions,
and to the referee for their constructive feedback,
corrections and improvement suggestions.


\bibliographystyle{siam}

\footnotesize
\begin{thebibliography}{10}

\bibitem{Aliprantis2006BookAn}
{\sc C.~D. Aliprantis and K.~C. Border}, {\em {Infinite dimensional analysis. A
  hitchhiker's guide.}}, Springer, Berlin, 3rd ed.~ed., 2006.

\bibitem{Ando1963pairContractions}
{\sc T.~And\^{o}}, {\em {On a pair of commutative contractions}}, Acta Sci.
  Math. (Szeged), 24 (1963), pp.~88--90.

\bibitem{BhatSkeide2015PureCounterexamples}
{\sc B.~R. Bhat and M.~Skeide}, {\em {Pure semigroups of isometries on Hilbert
  $C^{\ast}$-modules}}, Journal of Functional Analysis, 269 (2015),
  pp.~1539--1562.

\bibitem{CrabbDavie1975Article}
{\sc M.~J. Crabb and A.~M. Davie}, {\em {von {N}eumann's inequality for
  {H}ilbert space operators}}, Bull. London Math. Soc., 7 (1975), pp.~49--50.

\bibitem{Dahya2022complmetrproblem}
{\sc R.~Dahya}, {\em {On the complete metrisability of spaces of contractive
  semigroups}}, Arch. Math. (Basel), 118 (2022), pp.~509--528.

\bibitem{Dahya2022weakproblem}
\leavevmode\vrule height 2pt depth -1.6pt width 23pt, {\em {The space of
  contractive $C_{0}$-semigroups is a Baire space}}, J. Math. Anal. Appl., 508
  (2022).
\newblock Paper No. 125881, 12.

\bibitem{Dahya2023dilation}
\leavevmode\vrule height 2pt depth -1.6pt width 23pt, {\em {Dilations of
  commuting $C_{0}$-semigroups with bounded generators and the von Neumann
  polynomial inequality}}, J. Math. Anal. Appl., 523 (2023).
\newblock Paper No. 127021.

\bibitem{Dahya2024approximation}
\leavevmode\vrule height 2pt depth -1.6pt width 23pt, {\em {Characterisations
  of dilations via approximants, expectations, and functional calculi}}, J.
  Math. Anal. Appl., 529 (2024), p.~127607.
\newblock Paper No. 127607.

\bibitem{Doob1990BookStochProc}
{\sc J.~L. Doob}, {\em {Stochastic processes}}, {Wiley Classics Library}, John
  Wiley \& Sons, Inc., New York, 1990.
\newblock Reprint of the 1953 original, A Wiley-Interscience Publication.

\bibitem{Dynkin1965BookMarkovVolI}
{\sc E.~B. Dynkin}, {\em {Markov Processes}}, vol.~I, Springer Berlin
  Heidelberg, 1965.

\bibitem{Eisner2009embeddingOpSemigroup}
{\sc T.~Eisner}, {\em {Embedding operators into strongly continuous
  semigroups}}, Arch. Math. (Basel), 92 (2009), pp.~451--460.

\bibitem{Eisner2010typicalContraction}
\leavevmode\vrule height 2pt depth -1.6pt width 23pt, {\em {A ‘typical’
  contraction is unitary}}, Enseign. Math. (2), 56 (2010), pp.~403--410.

\bibitem{Eisner2010buchStableOpAndSemigroups}
\leavevmode\vrule height 2pt depth -1.6pt width 23pt, {\em {Stability of
  operators and operator semigroups}}, vol.~209 of {Operator Theory: Advances
  and Applications}, Birkh\"{a}user Verlag, Basel, 2010.

\bibitem{EisnerNagel2015BookErgTh}
{\sc T.~Eisner, B.~Farkas, M.~Haase, and R.~Nagel}, {\em {Operator theoretic
  aspects of ergodic theory}}, vol.~272 of {Graduate Texts in Mathematics},
  Springer, Cham, 2015.

\bibitem{Eisnermaitrai2013typicalOperators}
{\sc T.~Eisner and T.~M\'{a}trai}, {\em {On typical properties of Hilbert space
  operators}}, Israel J. Math., 195 (2013), pp.~247--281.

\bibitem{EisnerRadl2022embed}
{\sc T.~Eisner and A.~Radl}, {\em {Embeddability of real and positive
  operators}}, Linear Multilinear Algebra, 70 (2022), pp.~3747--3767.

\bibitem{Eisner2008categoryThmStableOpHilbert}
{\sc T.~Eisner and A.~Ser\'{e}ny}, {\em {Category theorems for stable operators
  on Hilbert spaces}}, Acta Sci. Math. (Szeged), 74 (2008), pp.~259--270.

\bibitem{Eisnersereny2009catThmStableSemigroups}
\leavevmode\vrule height 2pt depth -1.6pt width 23pt, {\em {Category theorems
  for stable semigroups}}, Ergodic Theory Dynam. Systems, 29 (2009),
  pp.~487--494.

\bibitem{EngelNagel2000semigroupTextBook}
{\sc K.-J. Engel and R.~Nagel}, {\em {One-parameter semigroups for linear
  evolution equations}}, vol.~194 of {Graduate Texts in Mathematics},
  Springer-Verlag, New York, 2000.

\bibitem{GacparSuciu2001vN}
{\sc D.~Ga\c{S}par and N.~Suciu}, {\em {On the generalized von Neumann
  inequality}}, in {Recent Advances in Operator Theory and Related Topics},
  L.~K\'{e}rchy, I.~Gohberg, C.~I. Foias, and H.~Langer, eds., Basel, 2001,
  Birkh\"{a}user Basel, pp.~291--304.

\bibitem{Grivaux2022localspecLp}
{\sc S.~Grivaux and {\'E}.~Matheron}, {\em {Local Spectral Properties of
  Typical Contractions on $\ell_{p}$-Spaces}}, Anal. Math., 48 (2022),
  pp.~755--778.

\bibitem{Grivaux2021invariantsubspaceLp}
{\sc S.~Grivaux, {\'E}.~Matheron, and Q.~Menet}, {\em {Does a typical
  $\ell_{p}$-space contraction have a non-trivial invariant subspace?}}, Trans.
  Amer. Math. Soc., 374 (2021), pp.~7359--7410.

\bibitem{Grivaux2021typicalexamples}
\leavevmode\vrule height 2pt depth -1.6pt width 23pt, {\em {Linear dynamical
  systems on Hilbert spaces: typical properties and explicit examples}}, Mem.
  Amer. Math. Soc., 269 (2021), pp.~v+147.

\bibitem{Haase2020spectralTheory}
{\sc M.~Haase}, {\em {The functional calculus approach to the spectral
  theorem}}, Indagationes Mathematicae, 31 (2020), pp.~1066--1098.

\bibitem{Hillephillips1957faAndSg}
{\sc E.~Hille and R.~S. Phillips}, {\em {Functional analysis and semi-groups}},
  vol.~31 of {American Mathematical Society Colloquium Publications}, American
  Mathematical Society, Providence, R.I., rev.~ed., 1957.

\bibitem{HolbrookHalmost1971vNPoly}
{\sc J.~A.~R. Holbrook and P.~R. Halmos}, {\em {Spectral Dilations and
  Polynomially Bounded Operators}}, Indiana University Mathematics Journal, 20
  (1971), pp.~1027--1034.

\bibitem{HytNeervanMarkLutz2016bookVol2}
{\sc T.~Hyt\"onen, J.~van Neerven, M.~Veraar, and L.~Weis}, {\em {Analysis in
  Banach spaces. Vol. II. Probabilistic Methods and Operator Theory}}, vol.~67
  of {Ergebnisse der Mathematik und ihrer Grenzgebiete. 3. Folge. A Series of
  Modern Surveys in Mathematics}, Springer, Cham, 2017.

\bibitem{KadisonRingrose1983volI}
{\sc R.~V. Kadison and J.~R. Ringrose}, {\em {Fundamentals of the theory of
  operator algebras. Vol. I}}, vol.~100 of {Pure and Applied Mathematics},
  Academic Press, Inc. [Harcourt Brace Jovanovich, Publishers], New York, 1983.

\bibitem{Kechris1995BookDST}
{\sc A.~S. Kechris}, {\em {Classical descriptive set theory}}, vol.~156 of
  {Graduate Texts in Mathematics}, Springer-Verlag, New York, 1995.

\bibitem{Klenke2008probTheory}
{\sc A.~Klenke}, {\em {Probability theory}},  (2008), pp.~xii+616.
\newblock Translated from the 2006 German original.

\bibitem{Krol2009}
{\sc S.~Kr\'{o}l}, {\em {A note on approximation of semigroups of contractions
  on Hilbert spaces}}, Semigroup Forum, 79 (2009), pp.~369--376.

\bibitem{Murphy1990}
{\sc G.~J. Murphy}, {\em {C\textsuperscript{$\ast$}-algebras and operator
  theory}}, Academic Press, Inc., Boston, MA, 1990.

\bibitem{Parrott1970counterExamplesDilation}
{\sc S.~Parrott}, {\em {Unitary dilations for commuting contractions}}, Pacific
  J. Math., 34 (1970), pp.~481--490.

\bibitem{Pedersen1989analysisBook}
{\sc G.~K. Pedersen}, {\em {Analysis now}}, vol.~118 of {Graduate Texts in
  Mathematics}, Springer-Verlag, New York, 1989.

\bibitem{Pederson2018BookCStarAuto}
{\sc G.~K. Pedersen}, {\em {C\textsuperscript{$\ast$}-algebras and their
  automorphism groups}}, {Pure and Applied Mathematics (Amsterdam)}, Academic
  Press, London, 2018.

\bibitem{Peller1976}
{\sc V.~V. Peller}, {\em {Estimates of operator polynomials in an $L^{p}$ space
  in terms of the multiplier norm}}, Zap. Nau\v{c}n. Sem. Leningrad. Otdel.
  Mat. Inst. Steklov. (LOMI), 65 (1976), pp.~133--148.
\newblock Investigations on linear operators and the theory of functions, VII.

\bibitem{Peller1981estimatesOperatorPolyLp}
\leavevmode\vrule height 2pt depth -1.6pt width 23pt, {\em {Estimates of
  operator polynomials in the space $L^{p}$ with respect to the multiplier
  norm}}, J. Sov. Math., 16 (1981), pp.~1139--1149.

\bibitem{Ptak1985}
{\sc M.~Ptak}, {\em {Unitary dilations of multi-parameter semi-groups of
  operators}}, Annales Polonici Mathematici, XLV (1985), pp.~237--243.

\bibitem{ShalitSkeide2022multiparam}
{\sc O.~Shalit and M.~Skeide}, {\em {CP-Semigroups and Dilations, Subproduct
  Systems and Superproduct Systems: The Multi-Parameter Case and Beyond}}.
\newblock Preprint available under
  \url{https://doi.org/10.48550/arXiv.2003.05166}, 2022.

\bibitem{Shalit2021DilationBook}
{\sc O.~M. Shalit}, {\em {Dilation theory: a guided tour}}, in {Operator
  theory, functional analysis and applications}, vol.~282 of {Oper. Theory Adv.
  Appl.}, Birkh\"{a}user/Springer, Cham, 2021, pp.~551--623.

\bibitem{Slocinski1974}
{\sc M.~S\l{}oci\'{n}ski}, {\em {Unitary dilation of two-parameter semi-groups
  of contractions}}, Bull. Acad. Polon. Sci. S\'{e}r. Sci. Math. Astronom.
  Phys., 22 (1974), pp.~1011--1014.

\bibitem{Slocinski1982}
\leavevmode\vrule height 2pt depth -1.6pt width 23pt, {\em {Unitary dilation of
  two-parameter semi-groups of contractions II}}, Zeszyty Naukowe Uniwersytetu
  Jagiellońskiego, 23 (1982), pp.~191--194.

\bibitem{Nagy1953}
{\sc B.~Sz\H{o}kefalvi-Nagy}, {\em {Sur les contractions de l'espace de
  Hilbert}}, Acta Sci. Math. (Szeged), 15 (1953), pp.~87--92.

\bibitem{Nagy1970}
{\sc B.~Sz\H{o}kefalvi-Nagy and C.~Foia\c{s}}, {\em {Harmonic analysis of
  operators on Hilbert space}}, North-Holland Publishing Co., Amsterdam-London;
  American Elsevier Publishing Co., Inc., New York; Akad\'{e}miai Kiad\'{o},
  Budapest, 1970.
\newblock Translated from the French and revised.

\bibitem{Varopoulos1973Article}
{\sc N.~T. Varopoulos}, {\em {Sur une in\'{e}galit\'{e} de von {N}eumann}}, C.
  R. Acad. Sci. Paris S\'{e}r. A-B, 277 (1973), pp.~A19--A22.

\bibitem{Varopoulos1974counterexamples}
\leavevmode\vrule height 2pt depth -1.6pt width 23pt, {\em {On an inequality of
  von Neumann and an application of the metric theory of tensor products to
  operators theory}}, J. Functional Analysis, 16 (1974), pp.~83--100.

\bibitem{Vershik2006Article}
{\sc A.~M. Vershik}, {\em {What does a generic {Markov} operator look like?}},
  St. Petersbg. Math. J., 17 (2006), pp.~763--772.

\end{thebibliography}
\def\bibname{References}
\bgroup

\egroup


\addresseshere
\end{document}
